\documentclass[leqno, a4paper, 12pt]{article}

\usepackage{amsmath}
\usepackage{hyperref}
\usepackage[all]{xy}
\usepackage{rotating}
\usepackage{mathrsfs}
\usepackage{xcolor}

\usepackage{amscd, amstext, amsfonts, amsthm, amsxtra, amssymb}

\theoremstyle{plain}
\newtheorem{theorem}{Theorem}[section]
\newtheorem{lemma}[theorem]{Lemma}
\newtheorem{proposition}[theorem]{Proposition}
\newtheorem{corollary}[theorem]{Corollary}

\theoremstyle{definition}
\newtheorem{definition}[theorem]{Definition}
\newtheorem{remark}[theorem]{Remark}

\newtheorem{example}[theorem]{Example}

\newcommand\bA{{\mathbb A}}
\newcommand\bC{{\mathbb C}}

\newcommand\bG{{\mathbb G}}

\newcommand\bL{{\mathbb L}}

\newcommand\bP{{\mathbb P}}
\newcommand\bQ{{\mathbb Q}}
\newcommand\bR{{\mathbb R}}

\newcommand\bZ{{\mathbb Z}}

\newcommand{\cC}{{\mathcal C}}
\newcommand{\cD}{{\mathcal D}}

\newcommand{\cK}{{\mathcal K}}
\newcommand{\cL}{{\mathcal L}}

\newcommand{\cN}{{\mathcal N}}
\newcommand{\cO}{{\mathcal O}}

\newcommand{\cU}{{\mathcal U}}
\newcommand{\cV}{{\mathcal V}}

\newcommand\fg{\mathfrak{g}}
\newcommand\fh{\mathfrak{h}}
\newcommand\fk{\mathfrak{k}}
\newcommand\fl{\mathfrak{l}}

\newcommand\fn{\mathfrak{n}}

\newcommand\ft{\mathfrak{t}}
\newcommand\fu{\mathfrak{u}}
\newcommand\fz{\mathfrak{z}}
\newcommand\fX{\mathfrak{X}}
\newcommand\fsl{\mathfrak{sl}}

\newcommand\Sb{\overline{S}}

\newcommand\alphab{{\overline{\alpha}}}
\newcommand\alphah{{\widehat{\alpha}}}

\newcommand\betab{{\overline{\beta}}}
\newcommand\gammab{{\overline{\gamma}}}

\newcommand\sigmab{{\overline{\sigma}}}
\newcommand\Thetab{{\overline{\Theta}}}
\newcommand\Thetabv{{\overline{\Theta}^\vee}}
\newcommand\chib{\overline{\chi}}
\newcommand\Db{\overline{\Delta}}
\newcommand\Xb{\overline{\fX}}
\newcommand\Rb{\overline{R}}

\newcommand\mL{\mathscr{L}}

\newcommand\ad{{\rm ad}}

\newcommand\bir{{\rm bir}}

\renewcommand\deg{{\rm deg}}
\newcommand\diag{{\rm diag}}

\newcommand\emfr{{\rm emfr}}

\newcommand\f{{\rm f}}
\newcommand\fr{{\rm fr}}
\newcommand\id{{\rm id}}

\renewcommand\ker{{\rm ker}}

\newcommand\s{{\rm s}}
\newcommand\n{{\rm n}}

\newcommand\rk{{\rm rk}}

\DeclareMathOperator\supp{Supp}

\newcommand\A{{\rm A}}
\DeclareMathOperator\Ad{Ad}

\newcommand\Aut{{\rm Aut}}

\DeclareMathOperator\C{C}

\newcommand\End{{\rm End}}

\DeclareMathOperator\Flag{Flag}

\newcommand\GL{{\rm GL}}

\DeclareMathOperator\Hom{Hom}

\newcommand\Int{{\rm Int}}
\newcommand\Ker{{\rm Ker}}

\DeclareMathOperator\N{N}
\newcommand\Pic{{\rm Pic}}
\newcommand\Nef{{\rm Nef}}
\newcommand\NE{{\rm NE}}
\newcommand\PGL{{\rm PGL}}
\DeclareMathOperator\PSL{PSL}

\newcommand\Q{{\rm Q}}

\DeclareMathOperator\RatCurves{RatCurves}
\DeclareMathOperator\R{R}

\DeclareMathOperator\SL{SL}
\newcommand\Sp{{\rm Sp}}
\DeclareMathOperator\SO{SO}

\newcommand\OO{{\rm O}}

\newcommand\Supp{{\rm Supp}}

\DeclareMathOperator\Univ{Univ}
\newcommand\VMRT{{\rm VMRT}}

\newcommand\yes{{\rm yes}}
\newcommand\no{{\rm no}}
\newcommand\Gr{{\rm Gr}}
\newcommand\OG{{\rm OG}}
\newcommand\LG{{\rm LG}}
\newcommand\IG{{\rm IG}}

\newcommand{\scal}[1]{\langle #1 \rangle}

\newcommand\adj{{{\rm adj}}}
\newcommand\summ{{{\rm sum}}}

\title{Minimal rational curves \\
on complete symmetric varieties}

\author{Michel Brion\footnote{Universit\'e Grenoble Alpes, Institut Fourier, CS 40700, 38058 Grenoble Cedex 9, France, 
{\tt michel.brion@univ-grenoble-alpes.fr}}, 
Shin-young Kim\footnote{Center for Geometry and Physics, Institute for Basic Science, Pohang 37673, Korea, {\tt shinyoungkim@ibs.re.kr}}  
and Nicolas Perrin\footnote{\'Ecole Polytechnique, Centre de Mathématiques Laurent Schwartz, 91128 Palaiseau Cedex, France, 
{\tt nicolas.perrin.cmls@polytechnique.edu}}}

\date{}

\begin{document}

\maketitle

\begin{abstract}
We describe the families of minimal rational curves on any complete
symmetric variety, and the corresponding varieties of minimal rational tangents (\VMRT). In particular, we prove that these varieties are homogeneous and that for non-exceptional indecomposable wonderful varieties, there is a unique family of minimal rational curves, and hence a unique \VMRT. We relate these results to the restricted root system of the associated symmetric space. 
\end{abstract}

\section{Introduction}
\label{sec:int}

Let $X$ be a projective uniruled variety over the field of complex numbers. 
An irreducible family $\cK$ of rational curves 
on $X$ is called 
a \emph{covering family} if there is a member of $\cK$ 
passing through a general point $x \in X$. 
If in addition the subfamily $\cK_x$ of curves in $\cK$ 
passing through $x$ is projective,  then $\cK$ is called a 
\emph{family of minimal rational curves}. 

\medskip

These curves play a prominent role in the study of the variety $X$.
There is a rational map $\tau_x : \cK_x \dasharrow \bP(T_xX)$ sending a curve to its tangent direction at $x$ and 
the closure of its image $\cC_x \subset \bP(T_xX)$ is an important invariant of $X$  called \emph{the variety of minimal rational tangents} or \VMRT~of $X$, see \cite{Hwang}, \cite{HM} and references therein. 

The families of minimal rational curves on projective rational homogeneous spaces $G/P$ for $G$ reductive and $P$ a parabolic subgroup are well understood. For example, if $G/P$ has Picard rank $1$, then there is a unique family of minimal rational curves which was used to prove its rigidity in a series of papers by Hwang and Mok (see e.g.  \cite{HM0, HM2}), with the unique exception of $B_3/P_2$ which admits an explicit degeneration constructed in \cite{PP} (see also \cite{HL}). 
If the Picard number of $G/P$ is greater than $1$, there are several families of minimal rational curves.

In \cite{BF}, the authors 
consider another case where $X$ has a large Picard group, namely the wonderful compactifications of adjoint simple groups. Suprisingly, they prove that there is a unique family of minimal rational curves for any such wonderful compactification and also that the corresponding \VMRT~is a rational homogeneous variety. These results were used in \cite{FL} to prove the rigidity of wonderful compactifications of semisimple groups, under the condition that the special fiber is Fano.

In this paper we generalize the results of \cite{BF} and describe the families of minimal rational curves on any complete symmetric variety. Rigidity of symmetric varieties of Picard number $1$ has already attracted some attention (see \cite{KP} and \cite{CFL}), we hope that our results will open new directions for higher Picard numbers.

\medskip

To state our main results, we recall basic definitions and properties of complete symmetric varieties.  Let $G$ be a connected reductive group and let $\sigma$ be a group involution of $G$. A symmetric subgroup is a closed subgroup 
$H \subset G$ such that $G^{\sigma,0} \subset H \subset G^\sigma$,
where $G^{\sigma,0}$ denotes the neutral component of the 
algebraic group $G^{\sigma}$.
The homogeneous space $G/H$ is a \emph{symmetric space}. 
We denote by $\fg$ and $\fh$ the Lie algebras of $G$ and $H$. 
Note the decomposition 
$\fg = \fh \oplus \fg^{-\sigma}$ as $H$-representations,
where $\fh = \fg^{\sigma}$.

Consider the normalizer $N = \N_G(H)$; the quotient $G/N$ is the 
\emph{adjoint homogeneous space} of the symmetric space $G/H$. 
The homogeneous space $G/N$ admits by \cite{decp} 
a unique wonderful compactification $X_\ad$. 
This is a smooth projective $G$-variety having an open dense orbit 
$G \cdot x_\ad = X_\ad^0 \simeq G/N$,  such that the boundary
$\partial X_\ad = X_\ad \setminus X^0_\ad$ is a simple normal 
crossing divisor: 
$\partial X_\ad = X_\ad^1 \cup \cdots \cup X_\ad^r$ 
where $X_\ad^i$ is a prime $G$-stable divisor for all $i \in [1,r]$. 
Furthermore for any $y,z \in X_\ad$ we have $G \cdot y = G \cdot z$ 
if and only if $\{i \ | y \in X_\ad^i\} = \{i \ | z \in X_\ad^i\}$. 
The integer $r$ is \emph{the rank} of $G/N$. 
A \emph{complete symmetric variety} is a smooth projective $G$-variety 
$X$ having a dense orbit $G \cdot x = X^0 \simeq G/H$ 
such that the natural map $G/H \to G/N \subset X_\ad$ 
extends to a $G$-equivariant morphism $\pi : X \to X_\ad$.
(We do not assume that $X$ contains a unique closed $G$-orbit).
The boundary $\partial X = X \setminus X^0$ is also a simple
normal crossing divisor with $G$-stable prime components.

Let $X$ be a complete symmetric variety with base point $x$
and map $\pi : X \to X_\ad$, 
and let $\cK$ be a family of minimal rational curves on $X$.
We will prove the following results.

\begin{theorem}[Theorem \ref{thm:VMRT=orbit}]
\label{thm:main0}
$\cK_x$ is smooth and $\tau_x: \cK_x \to \cC_x$ is an isomorphism.
\end{theorem}

In particular, understanding the \VMRT~ as an abstract variety is equivalent to understanding $\cK_x$. If the map $\pi$ contracts curves of the family $\cK$, then the description of the \VMRT~ follows easily from the case of toric varieties treated in \cite{CFH}, see Lemma \ref{lem:red}.

\begin{theorem}\label{thm:main1}
If $\pi$ contracts a curve of $\cK$,  then $\cC_x$ is a linear subspace 
of $\bP(\fg^{-\sigma})$.
\end{theorem}

We are therefore left to consider curves not contracted by $\pi$.

\begin{theorem}[Proposition \ref{prop:red} and Remark \ref{rem-2-ou-1}]
  \label{thm:main2}
  Assume that $\pi$ contracts no curve in $\cK$ and let $C \in \cK_x$.
  \begin{enumerate}
  \item There exists a unique family of minimal rational curves $\cL$ 
  in $X_\ad$ such that $\pi$ maps curves of $\cK$ to curves of $\cL$.
    \item We have 
    $1 \leq \partial X \cdot C \leq 
    \partial X_\ad \cdot \pi(C) \leq 2$.
    \item If $\partial X \cdot C = 
    \partial X_\ad \cdot \pi(C)$,  then every component of $\cK_x$ is isomorphic to 
    a component of $\cL_{x_\ad}$. Otherwise, $\cK_x$ is isomorphic to 
    a divisor of $\cL_{x_\ad}$.
  \end{enumerate}
\end{theorem}

One of the key ingredients for proving the above results consists
of the highest weight curves, 
introduced in \cite[Section 2]{BF} and studied further in Subsection \ref{subsec:hwc}. 
Given a Borel subgroup $B_H$ of $H$, the Borel Fixed Point Theorem  implies that every irreducible component of $\cK_x$ contains a $B_H$-fixed point $C$. Moreover, if $C$
is not contracted by $\pi$, 
then $C$ is mapped to a $B_H$-fixed point $C_\ad$ in $\cL_{x_\ad}$ that determines the associated component of $\cL_{x_\ad}$. Furthermore, 
the tangent space at $x_\ad$ of the highest weight curve $C_\ad$ in $\cL_{x_\ad}$
is a highest weight line in $T_{x_\ad}(X_\ad)$. 

\medskip

In view of the above results, we focus on wonderful compactifications of adjoint symmetric spaces. Decomposing $G$ into a product of indecomposable $\sigma$-stable factors, we obtain 
a decomposition of $G/N$ into a product of \emph{indecomposable symmetric spaces}. There are three possible types for these indecomposable factors (see Subsection \ref{subsec:ass} for more details):
\begin{enumerate}
\item \emph{Group type}: $(H \times H)/\diag(H)$, where $H$ is simple.
\item \emph{Hermitian type}: $G/\N_G(L)$,  where $G$ is simple and 
$L \subset G$ is a Levi subgroup. 
\item \emph{Simple type}: $G/H$, where $G$ is simple and $H^0$ is semisimple.
\end{enumerate}
Given a highest weight curve $C$ on $X$, we prove that there is a unique 
indecomposable factor  $X_C$ of $X_{\ad}$ such that the composition of  $\pi: X \to X_{\ad}$ with the projection $X_{\ad} \to X_C$ sends $C$ isomorphically to its image. We may thus replace $X_\ad$ by $X_C$ and assume that $X_\ad$ is indecomposable. 
In  particular $G/N$ is as in one of the above three cases. To understand the geometry of the indecomposable factors, we use the \emph{restricted root system}.

There exists a maximal torus $T_{\s}$, called of \emph{split type}, such that $T_{\s}$ is $\sigma$-stable and 
the subtorus
$S = \{t \in T_{\s} \ | \ \sigma(t) = t^{-1} \}^0$ has maximal dimension. The root system $R$ of $(G,T_{\s})$ is stable under the action of $\sigma$ and there is a basis $\Delta$ of $R$ such that, for $\alpha \in \Delta$, either $\sigma(\alpha) = \alpha$ or $\sigma(\alpha) < 0$. Set $\Delta_1 = \{ \alpha \in \Delta \ | \ \sigma(\alpha) < 0 \}$ and $\alphab = \alpha - \sigma(\alpha)$. The set $\Rb = \{\alphab \ | \ \alpha \in R \}$ is a (possibly non-reduced) root system with basis $\Db = \{ \alphab \ | \ \alpha \in \Delta_1 \}$ called the \emph{restricted root system} of the symmetric space. The rank of $\Rb$ is 
the rank $r$ of $G/H$. In Subsection \ref{subsec:div-rrs}, we relate curves and divisors in $X_\ad$ to the restricted root system (the results are probably well known to the experts but we could not find a convenient reference). Let $\Rb^\vee$ be the dual root system of  $\Rb$ 
with basis $\Db^\vee$ and 
coroot lattice $\bZ\Db^\vee$, and denote by $A_1(X)$ the Chow group of curves modulo rational equivalence. We prove the following result.

\begin{proposition}[Proposition \ref{prop-psi}]
  There is a surjective $\bZ$-linear map 
  $\psi: A_1(X_\ad) \to \bZ\Db^\vee$ such that:
  \begin{enumerate}
  \item The image of the monoid of effective curves is the monoid generated by the positive coroots.
  \item The image of the monoid of curves having non-negative intersection with any component of $\partial X_\ad$ is the  intersection of $\bZ\Db^\vee$ with the monoid of dominant cocharacters.
  \end{enumerate}
\end{proposition}

If the map $\psi$ is not injective, then $X_\ad$ is called \emph{exceptional} (see Propositions \ref{pic-rk} and \ref{prop:fibres-tau} for more details on this case). Since the class of a curve $C$ in a covering family $\cL$ of rational curves is effective and has non-negative intersection with any component of $\partial X_\ad$, it has to be contained in the intersection of the monoids generated by 
the positive coroots and by the dominant cocharacters. There is a unique minimal such element $\Thetabv$, the coroot of the highest root $\Thetab \in \Rb$. This gives a very natural candidate for classes of  minimal rational curves. Indeed we prove the following result.

\begin{theorem}[Corollary \ref{coro:vir-cov}]
  \label{thm:main3}
  Assume that $X_\ad$ is indecomposable.
  \begin{enumerate}
    \item If $X_\ad$ is not exceptional, there is a unique family 
    of minimal rational curves $\cL$ and the class of any $C_\ad \in \cL$ satisfies  $\psi([C_\ad]) = \Thetab^\vee$.
    \item If $X_\ad$ is exceptional, then there are exactly two such families $\cL^+$ and $\cL^-$ and the class of any $C_\ad^\pm \in \cL^\pm$ satifies $\psi([C_\ad^\pm]) = \Thetab^\vee$.
  \end{enumerate}  
\end{theorem}

Note that if $C_\ad$ is a highest weight curve in $\cL_{x_\ad}$, then its $H$-orbit $H \cdot C_\ad$ is contained in $\cL_{x_\ad}$. We describe the family $\cL_{x_\ad}$ by comparing the dimension of this orbit with the dimension of the family $\cL_{x_\ad}$ of curves whose class is described by the previous result. To compute the dimension of the $H$-orbits, we prove that the tangent line $T_{x_\ad}C_\ad$ lies in very specific nilpotent orbits in $\fg$. Let $\cO_{\min}$ be the minimal non-zero nilpotent orbit in $\fg$ and 
$\cO_{\summ,\sigma}$ be the nilpotent orbit of 
$e_\Theta - \sigma(e_\Theta)$, 
where $\Theta$ is the highest root of $G$ and 
$e_\Theta \in \fg_\Theta \setminus \{0\}$.  
Let $m \in T_{x_\ad}C_\ad \setminus \{0\}$, we prove 
the following in Corollary \ref{cor-hwc}.

\begin{proposition}
We have $m \in \cO_{\min}$ if $\sigma(\Theta)  = - \Theta$ 
and $m \in \cO_{\summ,\sigma}$ otherwise.
\end{proposition}

Using results of Kostant and Rallis \cite{ko-ra} we prove that the orbit $H \cdot m$ is Lagrangian in the nilpotent orbit $G \cdot m$ (equipped with the Kirillov-Kostant-Souriau invariant symplectic structure). Using this,  we compute the dimension of these orbits and obtain:

\begin{theorem}[Theorem \ref{theo-adj}]
  \label{thm:main4}
  \begin{enumerate}
    \item 
  If the restricted root system $\Rb$ is not of type $\A_r$, 
  then $\partial X_\ad \cdot C_\ad = 1$ and 
  $\dim H \cdot C_\ad = \dim \cL_{x_\ad}$. 
  Otherwise, we have 
  $\partial X_\ad \cdot C_\ad = 2$ and 
  $\dim H \cdot C_\ad = \dim \cL_{x_\ad} - 1$.
  \item
  If $\Rb$ is not of type ${\rm A}_r$, then $\cL_{x_\ad} = H \cdot C_\ad$. Furthermore, $\cL_{x_\ad}$ has two components if $X$ is Hermitian non-exceptional and is irreducible otherwise.
  \item If $\Rb$ is of type $\A_1$, then 
  $\cL_{x_\ad} \simeq \bP(\fg^{-\sigma})$.
      \item If $\Rb$ is of type $\A_r$ with $r \geq 2$, then there exists a $G$-equivariant birational morphism 
      $X_\ad \to \bP(V)$, for some irreducible $G$-representation $V$, and $\cL_{x_\ad}$ is isomorphic to the closed $G$-orbit in $\bP(V)$. The orbit $H \cdot C_\ad$ is a prime divisor in $\cL_{x_\ad}$.
      \item The orbit $H \cdot C_\ad$ and the variety $\cL_{x_\ad}$ are described in Table \ref{table-class}.
  \end{enumerate}
\end{theorem}

By results of Ruzzi \cite{Ruzzi}, the wonderful compactifications of 
indecomposable symmetric spaces are weak Fano varieties and most of them are Fano, with exceptions classified in loc.~cit., Table 2 (see also Subsection \ref{subsec:ex}). As a consequence, 
the wonderful compactifications of all Hermitian non-exceptional symmetric spaces are Fano, except in type CI. We thus obtain the following result.

\begin{corollary}
Let $X_\ad$ be the wonderful compactification of a Hermitian non-exceptional symmetric space not of type ${\rm CI}$ and whose restricted root system is not of type $\A_r$. Then $X_\ad$ is Fano and its \VMRT~ has two irreducible components.
\end{corollary}

The assumptions of the above corollary hold for four 
types in the classification: AIII, BDI, DIII and EVII.
This yields examples of Fano varieties with reducible 
\VMRT~ of positive dimension, thereby giving a negative
answer to a question of Hwang, 
see \cite[Section 5, Question 2]{Hwang}.
Note that these Fano varieties have Picard number 
at least~$2$, whereas there are examples of Fano 
varieties with Picard group $\bZ$ and reducible
\VMRT~ of positive dimension, see 
\cite[Proposition 3.15]{IM} and 
\cite[Remark A.9]{MOS}.

\medskip

In Table \ref{table-class}, we also give the embedding of 
$\cL_\ad \simeq \VMRT(X_\ad)$ in $\bP(\fg^{-\sigma})$. 
All \VMRT~ are
disjoint unions of projective rational homogeneous 
varieties, which are in turn products of homogeneous 
varieties of Picard rank one. In most cases, 
the embedding is the minimal embedding. For two cases 
(types AI and CI), the embedding is twice the minimal embedding. There is also a mixed case in type G.

From these results and Theorem \ref{thm:main2}, we obtain a full description of the \VMRT~of any complete symmetric variety $X$. We refer to Theorem \ref{thm:VMRT=orbit} for more details.

\setcounter{tocdepth}{2}

\begin{small}
\tableofcontents
\end{small}

\vskip 0.5 cm

\noindent
{\bf Acknowledgements.} The authors thank the Institut Fourier for hosting several meetings that led to this work. They also thank the Tsinghua Sanya International Mathematics Forum (TSIMF) and Baohua Fu for 
his invitation. 
The stay of the authors in TSIMF was very influential 
for this work. They thank Baohua Fu and Jun-Muk Hwang
for helpful comments, in particular for pointing out
the references \cite{IM} and \cite{MOS}, 
and the five referees for their careful reading and 
constructive remarks and corrections.
The authors finally thank the Centre Math\'ematique Laurent Schwartz (CMLS) at  \'Ecole Polytechnique for hosting the final meeting which led to this work. S.K. was partially supported by NRF Korea (2018R1A6A3A03012791) and the Institute for Basic Science (IBS-R003-D1). N.P. was partially supported by ANR Catore (ANR-18-CE40-0024) and ANR FanoHK (ANR-20-CE40-0023). 

\section{Rational curves and symmetric spaces}
\label{sec:prel}

In this section, we recall basic results on rational curves on uniruled varieties and then specialise to the case of almost homogeneous varieties. We also introduce symmetric homogeneous spaces and their adjoint symmetric space, and we obtain 
the existence of highest weight curves and their basic properties.

\subsection{Families of rational curves}
\label{subsec:frc}

In this subsection, we recall some notions and results 
on rational curves, after \cite[Sections II.2.2 and II.2.3]{Kollar} 
and \cite[Sections 2.1 and 2.2]{BK}.

Let $X$ be a smooth projective variety. Consider the scheme 
of morphisms $\Hom(\bP^1,X)$ and the open subscheme 
$\Hom_{\bir}(\bP^1,X)$ consisting of morphisms which are
birational onto their image.
The (normalized) \emph{space of rational curves} 
$\RatCurves(X)$ is the quotient of the normalization
$\Hom^{\n}_{\bir}(\bP^1,X)$ by the free action of 
$\Aut(\bP^1) = \PGL_2$ via reparametrization. We have 
a universal family
\[ \rho : \Univ(X) \longrightarrow \RatCurves(X) \] 
which is a $\bP^1$-bundle, and an evaluation map
\[ \mu : \Univ(X) \longrightarrow X \]
such that the morphism
$\rho \times \mu : \Univ(X) \to \RatCurves(X) \times X$
is finite.

Let $f \in \Hom_{\bir}(\bP^1,X)$ with image $C \subset X$.
We say that $C$ is \emph{free} if the pull-back 
$f^*(T_X)$ is globally generated, where $T_X$ denotes
the tangent bundle. Every free morphism yields a smooth
point of $\Hom_{\bir}(\bP^1,X)$, and hence of
$\RatCurves(X)$. Also, we say that $C$ is \emph{embedded}
if $f$ is 
an isomorphism to its image; equivalently, $C$ is smooth. 
The free (resp.~embedded free) 
curves form smooth open subschemes
$\RatCurves_{\emfr}(X) \subset \RatCurves_{\fr}(X)$
of the space of rational curves.

A \emph{family of rational curves} on $X$
is a component $\cK$ of $\RatCurves(X)$. 
We then have a universal family 
$\rho: \cU = \rho^{-1}(\cK)\to \cK$ which is again
a $\bP^1$-bundle, and an evaluation map
$\mu : \cU \to X$. 
For any $x \in X$, let $\cU_x = \mu^{-1}(x)$ and 
$\cK_x = \rho(\cU_x)$; then $\cK_x$ is the subfamily 
of curves through $x$. The restriction
$\rho_x : \cU_x \to \cK_x$ is finite, and is an isomorphism
above the smooth open subset of embedded free curves
(see \cite[Lemma 2.1]{BK}).

The family $\cK$ is \emph{covering} if $\cK_x$ is non-empty
for $x$ general. If in addition $\cK_x$ is projective
for $x$ general,  we say that $\cK$ is a
\emph{family of minimal rational curves}.

By sending every embedded free curve in $\cK_x$ to 
its tangent direction at $x$,  we obtain a morphism
$\tau_x : \cK_{\emfr,x} \to \bP(T_x X)$,
where $\bP(T_x X)$ denotes the projectivization of the
tangent space. We will view $\tau$ as a rational map
$\cK_x \dasharrow \bP(T_x X)$, defined at every curve
which is smooth at $x$.  The closure of the image
of $\tau$ is denoted by $\cC_x$ and called the
\emph{variety of tangents} of $\cK$ at $x$.

Let $\cK$ be a family of minimal rational curves on $X$. 
By \cite[Theorem 3.3]{Kebekus},
for a general point $x$, there are only finitely
many curves in $\cK_x$ which are singular at $x$. 
Thus, $\tau_x$ is defined along every 
positive-dimensional irreducible component of $\cK_x$.
In view of \cite[Theorem 3.4]{Kebekus},  
$\tau_x$ extends to a finite morphism
\[ \tau_x^{\n} : \cK_x^{\n} \longrightarrow \bP(T_x X), \] 
where $\cK_x^{\n}$ denotes the normalization.
Moreover, $\tau_x^{\n}$ is birational onto its image by 
\cite[Theorem 1]{HM}. The image $\cC_x$ is called the 
\emph{variety of minimal rational tangents} of $\cK$ at $x$ 
(\VMRT).

\medskip

Next, we consider covariance properties of families 
under a morphism of smooth projective varieties 
$\pi : X \to Y$.  Let $\cK$ be a family of rational curves 
on $X$. 
Assume that some $C \in \cK$ is represented by
a free morphism $f : \bP^1 \to X$ which is birational
onto its image,  and such that 
the composition $\pi \circ f : \bP^1 \to Y$ is free and 
birational onto its image as well. 
Let $D$ be the corresponding rational curve in $Y$,
and $\cL$ the family on $Y$ containing the free 
rational curve $D$. Finally, 
let $x = f(0) \in C$ and $y = \pi(x)\in D$.

\begin{lemma}\label{lem:dir}
With the above notation and assumptions, 
the morphism $\pi : X \to Y$ induces rational maps 
\[ \pi_* : \cK \dasharrow \cL,  \quad 
\pi_{*,x} : \cK_x \dasharrow \cL_y \]
which are defined at $C$ and send $C$ to $D$. 
If the differential $d\pi_x : T_x X \to T_y Y$
is injective, then so is the differential of $\pi_{*,x}$
at $C$.
\end{lemma}

\begin{proof}
Composing by $\pi$ yields a morphism
$\Hom(\bP^1,X) \to \Hom(\bP^1,Y)$
which is $\Aut(\bP^1)$-equivariant, and 
hence an equivariant rational map
between open subschemes of free morphisms
$\Hom_{\fr}(\bP^1,X) \dasharrow \Hom_{\fr}(\bP^1,Y)$
which is defined at $f$.  This readily yields the
rational map $\pi_*$.  The rational map $\pi_{*,x}$ 
is obtained from the analogous morphism
$\Hom(\bP^1,X;0 \mapsto x) \to 
\Hom(\bP^1,Y;0 \mapsto y)$
with the notation of \cite[Section II.1]{Kollar}.
By loc.~cit.,  Section II.2.3, 
the differential of the above morphism 
at $f$ is identified with the natural map
$H^0(\bP^1,(f^*T_X)(-1)) \to H^0(\bP^1, (f^* \pi^*T_Y)(-1))$.
This implies the final assertion
as $d\pi$ is injective on an open dense subset of $X$.
\end{proof}

In the opposite direction, assume that $\pi$ 
\emph{contracts} a curve $C \in \cK$, i.e., the composition 
$\rho^{-1}(C) \stackrel{\mu}{\longrightarrow}
X \stackrel{\pi}{\longrightarrow} Y$
is constant; then $\pi$ contracts all the curves in $\cK$
(see e.g.  \cite[Section 2.2]{BK}).
With this terminology, we may recall a useful observation
(see loc.~cit.,  Lemma 2.3):

\begin{lemma}\label{lem:product}

Consider two smooth projective varieties $Y$, $Z$, 
and let $X := Y \times Z$ with projections $p: X \to Y$, 
$q: X \to Z$. 

\begin{enumerate}

\item
The pull-back map 
$p^* : \Hom(\bP^1,Y) \times Z \to \Hom(\bP^1,X)$, 
$(f,z) \to (t \mapsto (f(t),z))$ induces a closed immersion 
$\RatCurves(Y) \times Z \to \RatCurves(X)$
with image a union of components. 

\item
The map $p^*$ sends covering families
(resp.~families of minimal rational curves) to covering 
families (resp.~families of minimal rational curves).

\item
A family of rational curves $\cK$ on $X$
is the pull-back of a family on $Y$ if and only if $q$ 
contracts some curve in $\cK$.

\item
Every family of minimal rational curves on $X$ 
is the pull-back of a unique family of minimal rational curves 
on $Y$ or $Z$. 
\end{enumerate}

\end{lemma}

\subsection{Almost homogeneous varieties}
\label{subsec:ahv}

We now assume that $X$ is \emph{almost homogeneous}, 
i.e., it is equipped with an action of a connected 
linear algebraic group $G$, and contains an open 
$G$-orbit $X^0$.  We recall and slightly generalize
results from \cite[Section 2]{BF} and \cite[Section 2.3]{BK}.

Choose a base point $x \in X^0$, and denote by 
$H = G_x$ its isotropy group. Then the orbit
$X^0 = G \cdot x$ is identified with the homogeneous 
space $G/H$, and the pair $(X,x)$, with an equivariant
embedding of this homogeneous space. Denoting by 
$\fg$ (resp.~$\fh$) the Lie algebra of $G$ (resp.~$H$), 
the tangent space $T_x X$ is identified with the
quotient $\fg/\fh$ as a representation of $H$ 
(the \emph{isotropy representation}).

Since $G$ is a rational variety, $X$ is unirational;
as a consequence, covering families exist. Also, 
$G$ acts on $\RatCurves(X)$ and on $\Univ(X)$
so that $\rho$ and $\mu$ are equivariant. Since $G$
is connected, it stabilizes every family $\cK$, as 
well as the open subset $\cK^0$ consisting of curves
which meet $X^0$. Every such curve is free (see e.g.
\cite[Lemma 2.1(i)]{BF}); thus, $\cK^0$ is smooth.
The subgroup $H \subset G$ acts compatibly
on $\cU_x$, $\cK_x$, $\bP(T_x X)$ and $\cC_x$.

We now obtain a variant of \cite[Lemma 2.4]{BK}:

\begin{lemma}\label{lem:cov}
A family of rational curves $\cK$ on $X$ is covering
if and only if $\cU_x$ is non-empty; equivalently,
$\cK_x$ is non-empty. Under these assumptions, 
$\cU_x$ is smooth and its components are permuted
transitively by $H$.
\end{lemma}

\begin{proof}
The morphism $\mu$ restricts to a $G$-equivariant
morphism
\[ \mu^0 : \cU^0 = \mu^{-1}(X^0) 
\longrightarrow X^0 = G/H \]
with fiber at $x$ being $\cU_x$. This yields an
isomorphism $\cU^0 \simeq G \times^H \cU_x$,
where the right-hand side denotes the quotient
of $G \times \cU_x$ by the $H$-action via
$h \cdot (g,z) = (gh^{-1}, h \cdot z)$.
Since $\cK^0$ is smooth, so are $\cU^0$ and hence
$\cU_x$. Also, $\cU^0$ is irreducible; thus, 
$H$ acts transitively on the components of $\cU_x$.
\end{proof}

Next, let $\pi : X \to Y$ be a surjective morphism,
where $Y$ is a smooth projective variety. Assume that
$Y$ is equipped with a $G$-action such that $\pi$
is equivariant. Let $y = \pi(x)$ and $Y^0 = G \cdot y$; 
then $Y^0 = \pi(X^0)$ is open in $Y$.  We now have
the following variants of 
\cite[Lemma 2.6,  Remark 2.7]{BK}:

\begin{lemma}\label{lem:push}
Keep the above notation and assumptions, and consider
a covering family of rational curves $\cK$ on $X$.
Assume that there exists $C \in \cK^0$ such that
$\pi\vert_C$ is birational onto its image $D$. Then:

\begin{enumerate}
 
\item
$D \in \cL$ for a unique covering family $\cL$ of
rational curves on $Y$. 

\item
$\pi$ induces a $G$-equivariant rational map 
$\pi_* : \cK \dasharrow \cL$, 
which is defined at $C$ and satisfies $\pi_*(C) = D$,
and an $H$-equivariant rational map
\[ \pi_{*,x} : \cK_x \dasharrow \cL_y, \quad
C \longmapsto D. \] 

\item
We have a commutative diagram of $H$-equivariant 
rational maps
\[ \xymatrix{
\cK_x \ar@{-->}[r]^{\pi_{*,x}} \ar@{-->}[d]_{\tau_x} 
& \cL_y \ar@{-->}[d]^{\tau_y} \\
\bP(T_x X) \ar@{-->}[r]^{d \pi_x} & \bP(T_y Y). \\
}\]

\end{enumerate}

\end{lemma}

\begin{proof}
(1) Replacing $C$ with a translate $g \cdot C$ for some
$g \in G$, we may assume that $x \in C$.  Then
the assertion follows from Lemma \ref{lem:cov}.

(2) This is a consequence of Lemma \ref{lem:dir},
except for the equivariance assertions which are
easily checked. 

(3) This follows readily from the definitions.
\end{proof}

\begin{remark}\label{rem:push}
If $\pi$ is birational (equivalently, it induces an isomorphism $X^0 \to Y^0$), then the assumptions of Lemma \ref{lem:push} hold and moreover  $\pi_{*,x}$ is an immersion. 
Indeed, $\pi_{*,x}$ is clearly an injective morphism.
Moreover, the differential of $\pi_{*,x}$ at every
$C \in \cK_x$ is injective by Lemma \ref{lem:dir}.
\end{remark}

Still considering a covering family of rational curves
$\cK$ on $X$, we now assume that $\pi$ contracts some 
curve in $\cK$, and hence all curves in $\cK$. Let 
\[ X \stackrel{\pi'}{\longrightarrow}
Y' \stackrel{\eta}{\longrightarrow} Y \]
be the Stein factorization of $\pi$,
where $Y'$ is a normal projective variety (possibly
singular), $\pi'$ is a contraction (that is, 
$\pi'_*(\cO_X) = \cO_{Y'}$), and $\eta$ is finite surjective. 
Then there is a unique action of $G$ on $Y'$ such that
$\pi'$ and $\eta$ are equivariant. Let $y' = \pi'(x)$
and $I = G_{y'}$; then $H \subset I \subset G$ and the
orbit $G \cdot y' \simeq G/I$ is open in $Y'$. Also,
let $F = \pi'^{-1}(y')$; then $F$ is the connected
component of $x$ in the fiber $\pi^{-1}(y)$, and hence
is a smooth projective variety
(by generic smoothness).  Moreover, $F$ is stable
by $I$ and contains $I \cdot x$ as its open orbit. 
Clearly, every curve in $\cK_x$ is contained in $F$.

\begin{lemma}\label{lem:cont}
Keep the above notation and assumptions, and 
assume that $I$ normalizes $H$. Then $\cK_x$ is
irreducible and there exists a unique covering family 
of rational curves $\cL$ on $F$ such that 
$\cK_x = \cL_x$. Moreover,  $\cK^0 = G \cdot \cL^0$. 
\end{lemma}

\begin{proof}
Note that $H$ acts trivially on $I/H$, since
$H \triangleleft I$. Thus, $H$ acts trivially on $F$, 
and hence on $\cU_x$. As $\cU^0 \simeq G \times^H \cU_x$,
we obtain $\cU^0 \simeq G/H \times \cU_x$. Since $\cU^0$
is irreducible, so are $\cU_x$ and $\cK_x$.

The inclusion $\iota : F \to X$ induces compatible 
immersions
\[ \RatCurves_{\fr}(F) \longrightarrow \RatCurves_{\fr}(X), 
\quad \Univ_{\fr}(F) \longrightarrow \Univ_{\fr}(X), \]
since they are injective and their differentials are
injective as well.
It follows that $\cK_x$ is an irreducible component
of $\mu_F^{-1}(x)$, where $\mu_F : \Univ(F) \to F$
denotes the evaluation map. So $\cK_x$ is an irreducible
component of $\cL_x$ for a unique family of rational
curves $\cL$ on $F$. Since $(F,x)$ is an equivariant embedding
of the homogeneous space $I^0/I^0 \cap H$ and 
$I^0 \cap H \triangleleft I^0$, we see that $\cL_x$ is
irreducible. Thus, $\cK_x = \cL_x$ and 
$\cL^0 = I^0 \cdot \cL_x = I^0 \cdot \cK_x$, so that
$\cK^0 = G \cdot \cK_x = G \cdot \cL_x= G \cdot \cL^0$.
\end{proof}

\begin{remark}\label{rem:cont}
The above assignement $\cK \mapsto \cL$ yields a bijection
between covering families of rational curves on $X$
which are contracted by $\pi$, and covering families of
rational curves on $F$. This restricts to a bijection
between families of minimal rational curves.
\end{remark}

\subsection{Symmetric spaces}
\label{subsec:ss}

In this subsection, we recall some basic facts
on symmetric spaces, after \cite[Section 26]{Timashev}
and its references.
We begin with some notation and conventions
which will be used throughout the sequel.
We will consider linear algebraic groups; for any
such group $H$, we denote by $H^0$ its neutral
component, i.e., the connected component of $H$
containing the neutral element $e$.

Let $G$ be a connected reductive algebraic group.
Let $T \subset G$ be a maximal torus,  
and $B \subset G$ a Borel subgroup containing $T$.  
We denote the character group of $T$ by 
$\fX = \fX(T)$, and the root system of $(G,T)$ by 
$R = R(G,T) \subset \fX$.  The roots of
$(B,T)$ form the set of positive roots $R^+$,
with basis $\Delta$ (the set of simple roots).
The Weyl group of $(G,T)$ is denoted by $W$.

Recall the decomposition of Lie algebras
$\fg = \ft \oplus \bigoplus_{\alpha \in R} \fg_{\alpha}$.
For any $\alpha \in R$,  we denote by $U_{\alpha}$
the closed subgroup of $G$ with Lie algebra
$\fg_{\alpha}$,  and by $G_{\alpha}$ the 
subgroup of $G$ generated by $U_{\alpha}$
and $U_{-\alpha}$.  Then $G_{\alpha}$ is
a closed subgroup of $G$, 
isomorphic to $\SL_2$ or $\PSL_2$.

Next, let $\sigma$ be a group involution of $G$. 
Denote by $G^{\sigma}$ the fixed point subgroup, 
and by $G^{\sigma,0}$ its neutral component. 
Let $H$ be a subgroup of $G$ such that 
$G^{\sigma,0} \subset H \subset G^{\sigma}$;
then we say that $H$ is a \emph{symmetric subgroup}
of $G$,  and the homogeneous space $G/H$ is a 
\emph{symmetric space}. 

By \cite[Section 8]{Steinberg},
the group $H$ is reductive; equivalently, 
the variety $G/H$ is affine. Also,
$\sigma$ induces an involution of the Lie algebra
$\fg$
that we still denote by $\sigma$ for simplicity. 
The Lie algebra of $H$ satisfies 
$\fh = \fg^{\sigma}$ and 
$\fg = \fh \oplus \fg^{-\sigma}$, where 
\[ \fg^{-\sigma} =
\{ x \in \fg ~\vert~ \sigma(x) = - x \} \]
is a $G^{\sigma}$-stable complement of $\fh$
in $\fg$. Thus, $\fg^{-\sigma}$ is the isotropy
representation of the symmetric space $G/H$. 
Note that $\fg^{-\sigma}$ is orthogonal to $\fh$ 
with respect to any $(G,\sigma)$-invariant scalar 
product on $\fg$, and hence is a self-dual
representation of $G^{\sigma}$. 

The involution $\sigma$ stabilizes a maximal torus 
$T$ of $G$, as follows e.g.~from 
\cite[Lemma 26.5]{Timashev}.  
Thus, $\sigma$ acts on the character
group $\fX$ and stabilizes the root system $R$;
it also acts on the Weyl group $W$ by conjugation.
We may choose a scalar product $(-,-)$ on
the real vector space $\fX_{\bR} = \fX \otimes_\bZ \bR$
which is invariant under $W$ and $\sigma$.

\begin{definition}
For $\alpha \in R$, one of the following cases occurs:
\begin{enumerate}
\item
$\sigma(\alpha) = \alpha$
and $\sigma$
fixes pointwise
$\fg_{\alpha}$. Then $\alpha$ is called a 
\emph{compact imaginary} root.
\item
$\sigma(\alpha) = \alpha$
and $\sigma$ acts on $\fg_{\alpha}$ by $-1$. Then
$\alpha$ is \emph{non-compact imaginary}.
\item
$\sigma(\alpha) = -\alpha$. Then $\alpha$
is \emph{real}.
\item
$\sigma(\alpha) \neq \pm \alpha$. Then
$\alpha$ is \emph{complex}.
\end{enumerate}
\end{definition}

Recall that any two maximal tori of $G$ are conjugate
and hence $\fX$,  $R$ and $W$ are independent of 
the choice of $T$.  But the action of $\sigma$ on these
objects depends on the choice of the $\sigma$-stable 
torus $T$,  up to conjugacy by $H$.  We now consider 
two special conjugacy classes of $\sigma$-stable maximal 
tori, that we call of fixed (resp.~split) type. 
These are constructed as follows.

\paragraph{Maximal tori of fixed type.} 
Choose a maximal torus $T_H$ of $H$; then its centralizer 
$T = \C_G(T_H)$ is a $\sigma$-stable maximal torus of $G$
and we have $T_H = T^{\sigma, 0}$.  
Moreover, $T$ is contained in a $\sigma$-stable Borel
subgroup $B$ of $G$; then $B_H = B^{\sigma,0}$
is a Borel subgroup of $H$ 
(see \cite[Lemma 26.7]{Timashev} for these facts).
Thus, $B_H = U_H T_H$, where $U_H = U \cap H$ is 
a maximal unipotent subgroup of $H$.

Since $B$ is $\sigma$-stable,
the action of $\sigma$ on the root system $R$ 
stabilizes $R^+$.  In particular, there are no real roots.
The subset of simple roots $\Delta$ is 
$\sigma$-stable as well.

Clearly, the maximal tori obtained in this way are 
exactly those containing a $\sigma$-fixed torus
(i.e., a subtorus $S\subset G$ such that
$\sigma(s) = s$ for all $s \in S$)
that is maximal for this property; they are
all conjugate under $H^0 = G^{\sigma,0}$.
We call every such maximal torus \emph{of fixed type}
and denote it by $T_{\f}$.

\paragraph{Maximal tori of split type.} 
In the opposite direction, a subtorus $S \subset G$ 
is called \emph{$\sigma$-split} if $\sigma(s) = s^{-1}$ for all 
$s \in S$.  Choose such a torus $S$ maximal for this property. 
Then $L = \C_G(S)$ satisfies $[L,L] \subset H^0$;
as a consequence, every maximal torus of $G$ containing
$S$ is $\sigma$-stable
(see \cite[Proposition 2]{Vust1} for these facts).  
We call every such maximal torus
\emph{of split type} and denote it by $T_{\s}$.
We then have $L = T_{\s} (L \cap H) = S (L \cap H)$.
Also, $L$ is a Levi subgroup of a minimal $\sigma$-split 
parabolic subgroup $P$, that is, $P$ is a parabolic subgroup 
of $G$ which is opposite to $\sigma(P)$, and minimal for 
this property (see loc.~cit.,  Section 1.2). 
Moreover,  $P G^{\sigma,0}$ is open in $G$ by loc.~cit.,
Theorem 1; as a consequence, $PH$ is open in $G$. 
Also, recall from loc.~cit.,  Section 1.3 
that the maximal split tori are all conjugate 
under $H^0$, as well as the minimal split parabolic
subgroups and the maximal tori of split type.

Denote by $x= H/H$ the base point of the homogeneous
space $G/H$; then the orbit $P \cdot x$ is 
isomorphic to $P/P \cap H = PH/H$, and hence is
open in $G \cdot x = G/H$.
Moreover, we have 
$P \cap H = P \cap \sigma(P) \cap H = L \cap H$;
in particular, $P \cap H \subset L$. 
Denote by $\R_u(P)$ the unipotent radical of $P$, 
so that we have the Levi decomposition
$P = \R_u(P) \rtimes L$. Then the map
$\R_u(P) \times L/L\cap H \to P/P \cap H$, 
$(g,z) \mapsto g \cdot z$ is an isomorphism. 
Thus, the map
\[ i : \R_u(P) \times L/L \cap H \longrightarrow G/H,
\quad (g,z) \longmapsto g \cdot z \]
is an open immersion with image $P\cdot x$.
Moreover, the natural map 
$S/S\cap H \to L/L \cap H$ is an isomorphism
as $L = S(L \cap H)$. Also, note that $i$ is $P$-equivariant, where 
$P$ acts on $\R_u(P) \times S/S \cap H$
via $(u,l) \cdot (g,z) = (u l g l^{-1}, l \cdot z)$,
and on $G/H$ via left multiplication.

The maximal tori of fixed type will be used in
the rest of this section and in Section 
\ref{sec:css}. Those of split type, and 
the corresponding restricted root system,
feature prominently in the subsequent sections.

\subsection{The normalizer of a symmetric subgroup}
\label{subsec:nss}

We keep the notation of Subsection \ref{subsec:ss},
and obtain some auxiliary results on the structure of 
the normalizer $N = \N_G(G^{\sigma})$. 
For this,  we introduce additional notation:
let $Z = Z(G)$ be the center of 
$G$,  with Lie algebra $\fz$.  Since $Z$ is
$\sigma$-stable,  we have 
$\fz = (\fz \cap \fh) \oplus (\fz \cap \fg^{-\sigma})$.
Also,  we denote by 
\[ q : G \longrightarrow G/Z = G_{\ad} \]
the quotient homomorphism,  where $G_{\ad}$
is the adjoint group.
The involution $\sigma$ of $G$ induces an
involution of $G_{\ad}$ that we still denote by
$\sigma$ for simplicity.

\begin{lemma}\label{lem:nor}

\begin{enumerate}
 
\item
$N = \{ g \in G ~\vert~ \sigma(g) g^{-1} \in Z \}$.

\item $q$ induces an isomorphism
$G/N \simeq G_{\ad}/G_{\ad}^{\sigma}$.
Moreover,  $G_{\ad}^{\sigma}$ is its own 
normalizer in $G_{\ad}$.

\item
$N = \N_G(G^{\sigma,0}) = \N_G(\fh)$.

\item
$N^0 = Z^0 H^0$.

\item
$Z^0 T_H$ is a maximal torus of $N$ for any
maximal torus $T_H$ of $H$. 

\end{enumerate}

\end{lemma}

\begin{proof}
(1) This is obtained in \cite[Lemma 1]{Vust2}
(see also \cite[Section I.7]{decp}); we recall the 
argument for completeness. 

Let $g \in G$ such that $\sigma(g) g^{-1} \in Z$.
For any $h \in G^{\sigma}$, we have
$\sigma(g h g^{-1}) = \sigma(g) h \sigma(g)^{-1}
= g h g^{-1}$, that is, $g h g^{-1} \in G^{\sigma}$.
So $g \in N$. For the converse, observe that
$N$ is reductive and normalized by $\sigma$.
The corresponding semi-direct product
$N \rtimes \langle \sigma \rangle$
is a reductive algebraic group, which acts
linearly on $\fg$ and stabilizes $\fh$. 
Thus, $\fh$ has an 
$N \rtimes \langle \sigma \rangle$-stable
complement, which must be $\fg^{-\sigma}$. In particular,
$\fg^{-\sigma}$ is $N$-stable; thus, $\Ad(N)$ commutes
with $\sigma$. So 
$\Ad(\sigma(g) g^{-1}) = 
\sigma \Ad(g) \sigma^{-1} \Ad(g)^{-1} = \id$
for any $g \in N$, that is, 
$\sigma(g) g^{-1} \in Z$.

(2) By (1),  we have $N = q^{-1}(G_{\ad}^{\sigma})$;
this yields the first assertion.  Applying (1) again
to $G_{\ad}$,  we obtain the second assertion.

(3) Clearly, we have 
$N = \N_G(G^{\sigma}) \subset \N_G(G^{\sigma,0})
= \N_G(\fh)$. Moreover, $\N_G(\fh)$ is reductive
and normalized by $\sigma$. Arguing as in the proof
of (1), it follows that $\sigma(g) g^{-1} \in Z$
for any $g \in \N_G(\fh)$, and hence $g \in N$.

(4) Denote by $\fn$ the Lie algebra of $N$.
Then (1) yields that 
$\fn = \{ x \in \fg ~\vert~ \sigma(x) - x \in \fz \}$.
Using the $\sigma$-stable decomposition
$\fg = \fz \oplus [\fg,\fg]$, it follows that
$\fn = \fz \oplus [\fg,\fg]^{\sigma} = \fz + \fh$.
This yields the assertion.

(5) This follows readily from (4).
\end{proof}

\begin{lemma}\label{lem:nors}
Let $S$ be a maximal $\sigma$-split torus of $G$.

\begin{enumerate}

\item
$N = G^{\sigma,0}(N \cap S)$.

\item
$H = G^{\sigma,0}(H \cap S)$
and $H \cap S$ is an elementary abelian $2$-group.

\item
$N = \N_G(H)$ and
$N/H \simeq N \cap S/H \cap S$. In particular,
$N/H$ is diagonalizable.

\end{enumerate}

\end{lemma}

\begin{proof}
(1) Let $P$ be a minimal $\sigma$-split 
parabolic subgroup of $G$ containing $S$.
Then $P_{\ad} = P/Z$ is a minimal $\sigma$-split
parabolic subgroup of $G_{\ad}$, containing
$S_{\ad} = S/S \cap Z$ which is a maximal 
$\sigma$-split torus of $G_{\ad}$. 
As seen in the discussion of maximal tori
of split type in Subsection \ref{subsec:ss}, the map
\[ \R_u(P_{\ad}) \times 
S_{\ad}/S_{\ad}^{\sigma} \longrightarrow
G_{\ad}/G_{\ad}^{\sigma}, \quad 
(g,z) \longmapsto g \cdot z \]
is an open immersion. 
Also,  the isomorphism 
$G_{\ad}/G_{\ad}^{\sigma} \simeq G/N$
(Lemma \ref{lem:nor} (2)) restricts to an
isomorphism
$S_{\ad}/S_{\ad}^{\sigma} \simeq SN/N$.
Thus,  the
multiplication map 
$\R_u(P) \times S N \to G$ is an open immersion
as well. Its image is $\R_u(P) S N = P N$, the open
orbit of $P \times N$ in $G$. Likewise,
$PG^{\sigma,0}$ is the open orbit of 
$P \times G^{\sigma,0}$ in $G$. Let $g \in N$; then 
the orbit $P g G^{\sigma,0} = P G^{\sigma,0} g$
is open in $G$. So $g \in PG^{\sigma,0}$, and 
hence $P N = P G^{\sigma,0}$. It follows that
$S N = S G^{\sigma,0}$; this yields the assertion.

(2) The first assertion follows readily from (1).
For the second assertion, just note that every 
$g \in S \cap H$ satisfies $g^{-1} = \sigma(g) = g$.

(3) Clearly, we have 
$\N_G(H) \subset \N_G(H^0) = \N_G(G^{\sigma,0})$.
Moreover, $\N_G(G^{\sigma,0}) = N$ in view of 
Lemma \ref{lem:nor}(2). Also, by combining
(1) and (2) above, we see that $N$ normalizes $H$,
since $N \cap S$ normalizes $G^{\sigma,0}$ and 
centralizes $H \cap S$. Thus, $\N_G(H) = N$. 
By (1) again, we have $N = H (N \cap S)$,
and hence $N/H \simeq (N \cap S)/(H \cap S)$ is
diagonalizable.
\end{proof}

\subsection{The adjoint symmetric space}
\label{subsec:ass}

Recall the adjoint group $G_{\ad} = G/Z$
equipped with an involution $\sigma$.
The homogeneous space $G_{\ad}/G_{\ad}^{\sigma}$
is called an \emph{adjoint symmetric space};
it is isomorphic to $G/N$ by Lemma \ref{lem:nor}(2).
The semisimple adjoint group $G_\ad$ is the product of its simple factors, 
and these are permuted by $\sigma$. This gives a
decomposition into a product of simple adjoint groups 
\[ G_\ad = G_1 \times \cdots \times G_m \times 
(H_1 \times H_1) \times \cdots \times (H_n \times H_n), \]
where $\sigma$ stabilizes the $G_i$
and exchanges the two copies of the $H_j$. 
Using the classification of symmetric spaces 
(see \cite[Section 26.5]{Timashev}), 
one arrives at a decomposition of $G/N$ into a product of 
\emph{indecomposable symmetric spaces} of the 
following three types:

\begin{enumerate}

\item
(group) 
$(H \times H)/\diag(H)$, where $H$ is simple adjoint.

\item
(Hermitian) 
$G/\N_G(L)$,  where $G$ is simple adjoint and 
$L \subset G$ is a Levi subgroup. 

\item
(simple)
$G/H$,  where $G$ is simple adjoint and $H^0$ is 
semisimple.

\end{enumerate}

In type (1),  we have $\sigma(x,y) = (y,x)$ for all 
$x,y \in H$. Thus, $G/N$ is just the group $H$ on which 
$H \times H$ acts by left and right multiplication. 
The isotropy representation $\fg^{-\sigma}$ is the adjoint 
representation of $H$ in $\fh$. This is an irreducible
representation with highest weight the highest root
$\Theta$.

In type (2),  we have $L = P \cap Q$,  
where $P$ and $Q$ are opposite maximal 
parabolic subgroups of $G$.
Moreover,  $\sigma$ is the conjugation $\Int(c)$,
where $c \in Z(L)$ and $c^2 \in Z(G)$.
Denote by $\alpha$ the unique simple root which
is not a root of $L$; then $\alpha$ has coefficient $1$ 
in the expansion of the highest root $\Theta$
as a linear combination of simple roots.
We have $\fg^{-\sigma} = \fu_P \oplus \fu_Q$,  where 
$\fu_P$ (resp.~$\fu_Q$) denotes the Lie algebra of $\R_u(P)$ 
(resp.~$\R_u(Q)$).  Moreover,  the representations $\fu_P$, 
$\fu_Q$ of $L$ are irreducible and dual to each other 
(see e.g.~\cite[Section 5.5]{RRS} for these results). 
Their highest weights relative to $L$ are $\Theta$,
$-\alpha$; they are linearly independent unless $G = \PSL_2$. 

We say that the Hermitian symmetric space
$G/\N_G(L)$ is \emph{exceptional}, 
if $P$ and $Q$ are not conjugate in $G$. 
Then $\N_G(L) = L$,  and hence $G/\N_G(L)$ may be 
identified with the open $G$-orbit in $G/P \times G/Q$ 
on which $G$ acts diagonally.  
In the \emph{non-exceptional} case,
where $P$ and $Q$ are conjugate in $G$, 
the group $\N_G(L)/L$ has order $2$ and exchanges 
$P$ and $Q$.  Moreover,  $G/\N_G(L)$ may be identified
with the open $G$-orbit in the symmetric square 
$(G/P)^{(2)}$, the quotient of $G/P \times G/P$
by the involution $(y,z) \mapsto (z,y)$.

In type (3),  $\fg^{-\sigma}$ is irreducible as a representation of 
$H^0$,  with a non-zero highest weight. (We do not know any uniform
proof of this fact, which can be checked on the classification
of symmetric spaces).

\subsection{Highest weight curves}
\label{subsec:hwc}

We still use the notation of Subsection \ref{subsec:ss},
and choose a maximal torus $T_H \subset H$ and a Borel 
subgroup $B_H \subset H$ containing $T_H$.
Recall that $T = \C_G(T_H)$ is a maximal
torus of fixed type of $G$.
We first obtain a generalization of \cite[Lemma 2.2]{BF}:

\begin{lemma}\label{lem:curve}
Let $C$ be an irreducible $B_H$-stable curve in $G/H$ 
through the base point $x$. 

\begin{enumerate}

\item
Either $C$ is contained in $Z^0 \cdot x$, 
or $B_H$ acts non-trivially on $C$. 

\item
In the latter case, $C$ is smooth 
and $B_H$-equivariantly isomorphic to its tangent 
line at $x$,  which is the $T_H$-weight space 
$\fg^{-\sigma}_{\lambda}$ for a unique non-zero highest weight 
$\lambda$ of $\fg^{-\sigma}$.  Moreover, $\lambda$ determines 
$C$ uniquely,  and the stabilizer of $C$ in $H$ equals 
the stabilizer of the weight space $\fg^{-\sigma}_{\lambda}$. 

\end{enumerate}

\end{lemma}

\begin{proof}
(1) Assume that $C$ is fixed pointwise by $B_H$. Then
the orbit $H^0 \cdot y$ is complete for any $y \in C$.
Since $G/H$ is affine, this orbit must be a point,
i.e., $C$ is fixed pointwise by $H^0$. Let $g \in G$ such
that $y = g \cdot x$, then $g^{-1} H^0 g \cdot x = x$,
i.e., $g^{-1} H^0 g \subset H$. So $g \in \N_G(H^0) = N$
(Lemma \ref{lem:nor}).  Thus, $C \subset N \cdot x$.
As $C$ is connected and contains $x$,  it follows that
$C \subset N^0 \cdot x$. 
But $N^0 \cdot x = Z^0 \cdot x$ by Lemma \ref{lem:nor}
again; this yields the assertion.

(2) This is obtained by arguing as in the proof of 
\cite[Lemma 2.2(i)]{BF}. We provide details for the 
reader's convenience.

Since $C$ is not fixed pointwise by $B_H$, it contains
an open orbit $B_H \cdot y$, where $y \neq x$. 
Thus, the isotropy group $B_{H,y}$ has codimension $1$ 
in $B_H$. We thus have
$B_{H,y}^0 = U_{H,y} \rtimes S$ for some subtorus $S$ 
of $B_H$. Replacing $y$ with a $B_H$-translate, 
we may assume that $S \subset T_H$.

If $S = T_H$,  then the orbit 
$B_H \cdot y = U_H \cdot y$ is isomorphic to 
$\bA^1$ and hence is closed in $G/H$,  since the
latter is an affine variety.
But $x \in \overline{B_H \cdot y} \setminus \{ y \}$,
a contradiction. For dimension reasons, it follows
that $S$ is a subtorus of codimension $1$ of $T_H$, 
and $U_H \subset B_{H,y}$. As a consequence, 
$C$ is fixed pointwise by $U_H$, since the latter
is a normal subgroup of $B_H$.  Thus,  $T_H \cdot y$ 
is open in $C$.

In particular, $x \in \overline{H \cdot y}$. So $C$ is 
contained in the fiber at $x$ of the geometric invariant
theory quotient $G/H \to H \backslash\!\!\backslash G/H$ 
of the smooth affine $H$-variety $G/H$. 
By a corollary of Luna's slice theorem 
(see \cite[Sections II.1 and III.1]{Luna}), this fiber is 
$H$-equivariantly isomorphic to the nilcone $\cN$ of $\fg^{-\sigma}$
(the fiber at $0$ of the quotient $\fg^{-\sigma} \to \fg^{-\sigma}/\!\!/H$;
it consists of the points $z \in \fg^{-\sigma}$ such that
$0 \in \overline{H \cdot z}$).
Thus, $C$ is $B_H$-equivariantly isomorphic to 
a $B_H$-stable curve $D$ in $\cN$. Moreover, $C$ and $D$ 
have the same stabilizer in $H$.

As $U_H$ fixes $D$ pointwise, we have 
$D \subset \cN \cap (\fg^{-\sigma})^{U_H}$. Also, 
$\fg^{-\sigma} = (\fg^{-\sigma} \cap \fz) \oplus (\fg^{-\sigma} \cap [\fg,\fg])$
and the projection $\fg^{-\sigma} \to \fg^{-\sigma} \cap \fz$ is 
$H$-invariant, hence sends $\cN$ to $0$.
it follows that
$D \subset (\fg^{-\sigma} \cap [\fg,\fg])^{U_H}$.

So we may assume that $G/H$ is an adjoint symmetric space.
Using the product decomposition of these spaces,  we see that
$D$ is a highest weight line from a unique indecomposable
factor of $G/H$, and is uniquely determined by its weight.
\end{proof}

We say that a curve $C$ as in Lemma \ref{lem:curve}(2)
is a \emph{highest weight curve}.  The corresponding
highest weight $\lambda$ satisfies 
$\lambda = \alpha \vert_{T_H}$ for some root $\alpha$, 
since the non-zero weights of $T_H$ in $\fg^{-\sigma}$ 
are restrictions of non-zero weights of $T$ in $\fg$.
Let $S = \Ker(\lambda)^0 = (\Ker(\alpha) \cap T_H)^0$; 
then $S$ is a subtorus of codimension $1$ of $T_H$,
and fixes $C$ pointwise. 
Thus, the centralizer $\C_G(S)$ is a $\sigma$-stable 
subgroup of $G$ containing $T$;
also,  $C_G(S)$ is connected and reductive by 
\cite[Section 7.6.4]{springer-linear}. 
Moreover,  $\C_H(S)$ is a symmetric subgroup of $\C_G(S)$
containing $T_H$, and $C$ is a highest weight curve
of the symmetric space $\C_G(S)/\C_H(S)$.

Recall the following easy result
(see \cite[Section 2]{Springer} and \cite[Lemma 2.5]{Brion}):

\begin{lemma}\label{lem:rank1}
With the above notation and assumptions, the adjoint 
symmetric space of $\C_G(S)/\C_H(S)$ is one of 
the following:

\begin{itemize}
\item[{\rm ($\A_1$)}] 
$\PSL_2/N$,  where $N$ denotes 
the normalizer of the diagonal torus in $\PSL_2$. 

\item[{\rm ($\A_1 \times \A_1$)}]
$(\PSL_2 \times \PSL_2)/\diag(\PSL_2)$.  
Then $\sigma(\alpha)$ is strongly orthogonal to $\alpha$.

\item[{\rm ($\A_2$)}]
$\PSL_3/\SO_3$, where $\SO_3$ denotes the special 
orthogonal group.  

\end{itemize}

Moreover,  $\alpha$ is non-compact imaginary in types 
$(\A_1)$ and $(\A_2)$.  If the involution $\sigma$ is inner 
(equivalently,  $T_H$ is a maximal torus of $G$),  then 
only type $(\A_1)$ occurs.

\end{lemma}

Still considering a highest weight curve $C$ of weight $\lambda$,
we now obtain a description of $C$ and its tangent line 
$T_x C \subset T_x G/H$ (using the identifications 
$T_x G/H = \fg/\fh = \fg^{-\sigma}$) in the above three types.

\begin{proposition}\label{prop:roots}
In type $(\A_1)$,  there is a unique root $\alpha$
such that $\lambda = \alpha \vert_{T_H}$. 
Moreover,  $C = U_\alpha \cdot x$
and $T_x C = \fg_{\alpha}$.

In type $(\A_1 \times \A_1)$,  there are exactly two roots 
$\alpha$,  $\beta$ such that 
$\lambda = \alpha \vert_{T_H} = \beta \vert_{T_H}$. 
Moreover,  $\alpha$ and $\beta = \sigma(\alpha)$ are
the simple roots of $(\C_G(S),T)$.  We have
$C = U_\alpha \cdot x = U_\beta \cdot x$
and $T_x C = \bC(e_{\alpha} - \sigma(e_{\alpha}))$,
where $e_{\alpha} \in \fg_{\alpha} \setminus \{ 0 \}$.

In type $(\A_2)$,  there is a unique root $\alpha$
such that $\lambda = \alpha \vert_{T_H}$.  Moreover,
$\alpha = \alpha_1 + \alpha_2$, where
$\alpha_1$ and $\alpha_2 = \sigma(\alpha_1)$ are 
the simple roots of $(\C_G(S),T)$.   Also, 
$C = U_\alpha \cdot x$ and $T_x C = \fg_{\alpha}$.
\end{proposition}

\begin{proof}
In type $(\A_1)$,  there are two highest weight 
curves in $\C_G(S)/\C_H(S)$,  namely, 
$U_{\alpha} \cdot x$ and $U_{-\alpha} \cdot x$. 

In type $(\A_1 \times \A_1)$,  recall that the adjoint
symmetric space of $\C_G(S)/\C_H(S)$ is the group
$(\PSL_2 \times \PSL_2)/\diag(\PSL_2) = \PSL_2$.
So the roots of $(\C_G(S),T)$ are $\pm \alpha$, 
$\pm \sigma(\alpha)$.  Moreover,  $U_\alpha \cdot x$
is the unique highest weight curve; it is identified 
with the standard unipotent subgroup $U \subset \PSL_2$, 
and likewise for $U_{\sigma(\alpha)} \cdot x$.

In type $(\A_2)$, the adjoint symmetric space of 
$\C_G(S)/\C_H(S)$ is $\PSL_3/\SO_3$; one checks that 
the highest weight of its isotropy representation is
$(\alpha_1 + \alpha_2)\vert_{T_H}$, where
$\alpha_1$, $\alpha_2$ are the simple roots of
$\PSL_3$. It follows that $\alpha = \alpha_1 + \alpha_2$,
and $\sigma(\alpha_1) = \alpha_2$. Finally,
$U_\alpha \cdot x$ is an irreducible curve in $G/H$,
stable by $T_H$ and fixed by $U_H$ (since the latter 
commutes with $U_\alpha$ and fixes $x$). Thus, 
$U_\alpha \cdot x$ is the highest weight curve 
in $\C_G(S)/\C_H(S)$.

This yields the assertions on roots and highest weight 
curves.  Those on their tangent lines are readily verified.
\end{proof}

\begin{corollary}\label{cor:root}
Let $C$ be a highest weight curve of weight $\lambda$. 
Then there exists $\alpha \in R$ such that 
$\lambda = \alpha \vert_{T_H}$ and 
$C = U_{\alpha} \cdot x$.
\end{corollary}

\begin{corollary}\label{cor:alter}
We have the following alternative for a simple group $G$:

\begin{enumerate}
\item $T_x C =  \fg_\alpha$ for a long root $\alpha$,  or
\item $T_x C = \fg_\alpha$ for a short root $\alpha$,  or
\item $T_x C$ is spanned by $e_{\alpha} - \sigma(e_{\alpha})$,
where $\alpha \in R$ is strongly orthogonal to
$\sigma(\alpha)$.  Moreover,  $G$ is simply laced.
\end{enumerate}

\end{corollary}

\begin{proof}
In view of Proposition \ref{prop:roots}, 
we only have to show that $G$ is simply laced in case (3).
Then $\sigma(\alpha) \neq \alpha$,
and hence $\sigma$ acts non-trivially on $R$.  
As $\sigma$ stabilizes the set $\Delta$ of simple roots,  
it induces a non-trivial automorphism of the Dynkin 
diagram.  But this only occurs for $G$ simply laced.
\end{proof}

\begin{remark} \label{rem:cases}
The three cases in Corollary \ref{cor:alter} do occur
(see Table \ref{table-class} for the notation on types):

\begin{enumerate}
\item In type AI with $G = \SL_n$ and $H = \SO_n$, 
we have $T_x C = \fg_\Theta$ with $\Theta$ the highest root.
\item In type BII with $G = \SO_{2n+1}$ and 
$H = S(\OO_1 \times \OO_{2n})$, 
we have $T_x C = \fg_\theta$ with $\theta$ the highest short root. 
\item In type AII with $G = \SL_{2n}$ and $H = \Sp_{2n}$, 
we have 
$T_x C = \bC (e_{\Theta - \alpha_1} - e_{\Theta - \alpha_{2n-1}})$, 
where $\Theta$ is the highest root and $\alpha_1$ and 
$\alpha_{2n-1}$ are simple roots labeled as in \cite{bourbaki}.
\item Clearly,  case (2) does not occur for simply laced groups.
\item In type $G_2$, none of cases (2) and (3) occurs.
Indeed $\sigma(\alpha) = \alpha$ for any root 
$\alpha \in R$. Therefore, every root is imaginary.
An easy computation shows that the highest 
non-compact root must be long.
\end{enumerate}

\end{remark}

Next, we assume that $G/H$ is indecomposable;
in particular, $G$ is simple or
$G = H \times H$ with $H$ simple.
We show that $T_x C$ is contained in a nilpotent orbit
of a very special type,  defined as follows:

\begin{definition} \label{def:nilp}
Let $G/H$ be an indecomposable symmetric space.
\begin{enumerate}
\item If $G = H \times H$,  then set 
$\cO_{\min} = G \cdot (e,-e) \subset \fg = \fh \oplus \fh$ 
where $e \in \fh$ is a highest weight vector for $H$. 
\item If $G$ is simple,  define a 
\emph{nilpotent orbit} $\cO_{\min}$ 
and a \emph{type of nilpotent orbits}
$\cO_\summ$ in $\fg$ as follows.
\begin{enumerate}
\item $\cO_{\min} = G \cdot e$ where $e$ is a highest weight 
vector in $\fg$.
\item A nilpotent orbit $\cO$ is of type $\cO_{\summ}$ if $\cO = G \cdot (e_1 + e_2)$,
where $e_i \in \fg_{\alpha_i}$ is a root vector with $\alpha_1$ 
and $\alpha_2$ two strongly orthogonal long roots.
\end{enumerate}
\end{enumerate}
\end{definition}

\begin{remark}
There is a unique nilpotent orbit of type $\cO_\summ$ except for $G$ of type ${\rm B}_n$ or ${\rm D}_n$, in which case there are two possible nilpotent orbits.
\end{remark}

\begin{proposition}\label{prop:nilp}
With the above notation, 
$T_x C \setminus \{ 0 \}$ 
is contained in $\cO_{\min}$ or in 
a nilpotent orbit of type $\cO_\summ$.
\end{proposition}

This follows by combining Corollary \ref{cor:alter}, 
Remark \ref{rem:cases}(5) and the next result.

\begin{lemma}\label{lem:orbits}
Assume that $G$ is not simply laced and not of type $G_2$. 
Let $e_\alpha \in \fg_\alpha \setminus \{ 0 \}$ 
with $\alpha$ a short root. Then $e_\alpha$ belongs to a nilpotent orbit 
of type $\cO_\summ$.
\end{lemma}

\begin{proof}
Since all short roots are in the same orbit under the action 
of the Weyl group, we may assume that $\alpha$ is 
the highest short root $\theta$.

There exists a simple root $\beta$ such that 
$\theta + \beta$ is a root; then $\theta+\beta$ must be 
a long root.  We claim that 
$\scal{\beta^\vee,\theta + \beta} = 2$. 
Indeed,  since $\theta$ is a dominant weight,  we have 
$\scal{\beta^\vee,\theta} \geq 0$ and we get 
$\scal{\beta^\vee,\theta + \beta} \geq \scal{\beta^\vee,\beta} = 2$. 
Since $G$ is not of type $G_2$,  we must have equality: 
$\scal{\beta^\vee,\theta+\beta} = 2$.
  
We get $s_\beta(\theta + \beta) = \theta - \beta$; thus, 
$\theta - \beta$ is a long root. Moreover,  we have
$\theta-\beta +\theta+\beta = 2\theta$ and $\theta+\beta-(\theta-\beta) = 2\beta$. Thus $\theta-\beta$ and $\theta+\beta$ are strongly orthogonal.
%Denote by $(-,-)$ a scalar product on $\fX_{\bR}$ 
%which is invariant under $W$ and $\sigma$,  and 
%such that long roots have length $2$.  We have 
%$(\theta - \beta,\theta-\beta) = 2 = (\theta+\beta,\theta+\beta)$. 
%This gives $(\theta,\beta) = 0$ and 
%$(\theta,\theta) = 1 = (\beta,\beta)$, 
%so both $\theta$ and $\beta$ are short roots.  
%We also get 
%$(\theta-\beta,\theta+\beta) = (\theta,\theta) - (\beta,\beta) = 0$,
%so $\theta-\beta$ and $\theta+\beta$ are long orthogonal roots. 
%We check that they are strongly orthogonal: their sum is 
%$\theta-\beta +\theta+\beta = 2\theta$ which is not a root, 
%and their difference $\theta+\beta-(\theta-\beta) = 2\beta$ 
%is not a root either.

We are left to prove that $e_\theta$ and 
$e_{\theta-\beta} + e_{\theta+\beta}$ are in the same 
$G$-orbit in $\fg$. 
The group $G_{\beta}$ (generated by $U_{\pm \beta}$)
acts on $\fg$ and stabilizes the subspace 
$V = \fg_{\theta-\beta}\oplus\fg_\theta\oplus\fg_{\theta+\beta}$ 
on which it acts via the adjoint representation. 
Moreover, $G_{\beta}$ acts with two orbits in the
projective space $\bP(V)$: 
the minimal orbit and its complement. 
The point $[e_\theta]$ is in this last orbit, 
which also contains 
$[e_{\theta-\beta} + e_{\theta+\beta}]$.
This yields the assertion, since nilpotent
orbits are stable under non-trivial homotheties.
\end{proof}

\begin{remark}
In Proposition \ref{prop-hwc} and Corollary \ref{cor-hwc}, we will give a more precise statement describing the nilpotent orbit containing $T_xC$ for $C \in \cK_x$,  where $\cK$ is a family of minimal rational curves.
\end{remark}

\section{Complete symmetric varieties}
\label{sec:css}

In this section, we recall the notions
of wonderful symmetric varieties and complete symmetric varieties and we describe their relations, especially how to compare their respective families of minimal
rational curves. 
We then recall the result of  \cite{BF} about such families on wonderful compactifications of groups. We end the section by a description of the families of minimal rational curves on complete symmetric varieties in the group type, in the Hermitian type, and in some cases of simple type.

\subsection{Wonderful and complete symmetric varieties}
\label{subsec:rwsv}

We use the notation of Subsections \ref{subsec:ss}
and \ref{subsec:nss}. 
In particular,  $G$ denotes a connected reductive group,
$H$ a symmetric subgroup relative to an involution
$\sigma$,  and $N$ the normalizer of $H$ in $G$.
We denote by $x$ (resp.~$x_\ad$) the base point of
the homogeneous space $G/H$ (resp.~$G/N$).
The natural morphism 
\[ \pi : G/H \longrightarrow G/N,
\quad x \longmapsto x_\ad \]
is a principal bundle under $N/H$.  Moreover, $N/H$ is
diagonalizable by Lemma \ref{lem:nors}. 
We have the ``Stein factorization'' of $\pi$ as 
\[ G/H \stackrel{\pi'}{\longrightarrow}
G/N^0 H \stackrel{\eta}{\longrightarrow} G/N, \]
where $\pi'$ is a principal bundle under the torus
$N^0 H/H \simeq Z^0/H \cap Z^0$ (Lemma \ref{lem:nor}),
and $\eta$ is a principal bundle under $N/N^0 H$,
a finite abelian group (Lemma \ref{lem:nors}).

By \cite{decp}, the adjoint symmetric space 
$G/N = G_{\ad}/G_{\ad}^{\sigma}$  
admits a wonderful equivariant embedding that we denote 
by $X_\ad$,  with base point $x_\ad$.  We say that
$X_{\ad}$ is a \emph{wonderful symmetric variety}.

We now recall from \cite[Section 3.3]{LP} how to obtain 
$X_{\ad}$ from the wonderful 
$G_{\ad} \times G_{\ad}$-equivariant embedding
$\overline{G_\ad}$ of $G_\ad 
= (G_{\ad} \times G_{\ad})/\diag(G_{\ad})$. 
We begin with a general construction: the morphism 
\[ G \longrightarrow G, \quad 
g \longmapsto \sigma(g) g^{-1} \]
factors through a closed immersion 
$\iota : G/G^{\sigma} \to G$ 
which sends the base point $x$ to the neutral element 
$e$.  The image of $\iota$ is a connected component of 
the fixed locus $G^{-\sigma}$, where $- \sigma$ 
denotes the involution $g \mapsto \sigma(g^{-1})$ 
of $G$ (viewed as a variety); see
\cite[Lemma 2.4]{Richardson} 
for these results. Note that $\iota$ is 
equivariant for the natural action of $G$ on 
$G/G^{\sigma}$,  and the $G$-action on itself
via twisted conjugation, defined by 
$g_1 \cdot  g_2 := \sigma(g_1 ) g_2 g_1^{-1}$. 
Also, the differential of $\iota$ at $x$
is identified with the inclusion 
$\fg^{-\sigma} \hookrightarrow \fg$.

This construction applies to the involution $\sigma$
of $G_{\ad}$; moreover, $-\sigma$ extends uniquely 
to an involution of $\overline{G_\ad}$ that we still 
denote by $-\sigma$,  and $\iota$ extends uniquely 
to a closed immersion 
\[ \overline{\iota}: X_{\ad} 
\hookrightarrow \overline{G_\ad} \]
which identifies $X_{\ad}$ with a component of 
$(\overline{G_\ad})^{-\sigma}$. 

\begin{definition}
A \emph{complete symmetric variety} is a smooth
projective equivariant embedding $(X,x)$ of $G/H$ 
that is toroidal in the sense of 
\cite[Section 29]{Timashev}), i.e.,
the morphism $\pi : G/H \to G/N$ 
extends to a morphism $X \to X_\ad$. 
We still denote by 
\[ \pi : X \longrightarrow X_\ad \] 
this extension, 
which is of course unique and hence $G$-equivariant. 
(We do not assume that $X$ has a unique closed $G$-orbit).

If $X = X_\ad$ is adjoint we will call $X$ of group, simple, Hermitian or Hermitian exceptional type if $G/H$ is of group, simple, Hermitian or Hermitian exceptional type.
\end{definition}

For any complete symmetric variety $X$, 
the boundary $\partial X = X \setminus X^0$ is a divisor
with simple normal crossings. We will use the 
following relation between the canonical divisors
of $X$ and $X_{\ad}$:

\begin{lemma}\label{lem:can}
With the above notation, we have
the equality of divisor classes
\[ K_X + \partial X = 
\pi^*(K_{X_{\ad}} + \partial X_{\ad}). \]
\end{lemma}

\begin{proof}
Recall from \cite{decp} that $X_{\ad}$ is isomorphic 
to the $G$-orbit closure of $\fh$ in the Grassmannian
of subspaces of $\fg$. Moreover, 
$- K_{X_{\ad}} - \partial X_{\ad}$ is the hyperplane
class $h$ in the corresponding Pl\"ucker embedding
(see \cite[Proposition 30.8]{Timashev}). Also,
$- K_X - \partial X = \pi^*(h)$ by loc.~cit. 
\end{proof}

We will also use the following description of the 
general fibers of $\pi$; by equivariance, it suffices
to describe the fiber at $x$.  In view of the Stein 
factorization,  $\pi$ is the composition of 
a contraction $\pi' : X \to X'$ as discussed after
Remark  \ref{rem:push},  and a finite surjective 
equivariant morphism $\eta : X' \to X_\ad$.
The pair $(X', x' := \pi'(x))$ is a normal projective
equivariant embedding (possibly singular) of 
$G/N^0 H = G/Z^0 H$ (a symmetric space under $G/Z^0$). 
Moreover,  the fiber of $\pi$ at $x$ is isomorphic to 
the associated bundle $N \times^{N^0 H} F$, where 
$F$ denotes the fiber of $\pi'$ at $x$.  The group 
$N^0 H$ acts on $F$ via its quotient torus 
$N^0 H/H \simeq N^0/H \cap N^0 \simeq Z^0/H \cap Z^0$
(where the second isomorphism follows from 
Lemma \ref{lem:nor}), 
and $F$ is a smooth projective toric variety 
under that torus.

\subsection{Relation between families of minimal rational curves}
\label{subsec:mfcsv}

In this subsection,  we consider a complete symmetric 
variety $X$ with base point $x$.  We will reduce somehow
the description of families of minimal rational curves
on $X$ to the cases where $X$ is a smooth projective
toric variety or a wonderful symmetric variety.  
A key notion is that of a \emph{highest weight curve},  
i.e.,  an irreducible curve $C \subset X$ through $x$ 
which is stable and not fixed pointwise by the Borel 
subgroup $B_H$.
Equivalently,  $C \cap X^0$ is a highest weight curve
in the sense of Subsection \ref{subsec:hwc}.

By Corollary \ref{cor:root},  we have 
$C = \overline{U_{\alpha} \cdot x}$ 
for some root $\alpha$.  In view of
\cite[Lemma 2.1 (i),  Lemma 2.4]{BF}, this yields:

\begin{lemma}\label{lem:hwc}
Let $C$ be a highest weight curve.
Then $C$ is an embedded free rational curve.
\end{lemma}

We now obtain an alternative for families of minimal 
rational curves:

\begin{lemma}\label{lem:red}
Let $\cK$ be a family of minimal rational curves on $X$.  

\begin{enumerate}

\item
Either each curve in $\cK$ is contracted by $\pi$,
or $\cK_x$ contains a highest weight curve.

\item
In the former case, $\cK_x$ is 
a family of minimal rational curves on the toric variety $F$.
Moreover,  the tangent map $\tau_x$ is an isomorphism
of $\cK_x$ with a linear subspace of $\bP(\fg^{-\sigma} \cap \fz)$. 

\end{enumerate}

\end{lemma}

\begin{proof}
(1)  In view of Borel's fixed point theorem, 
$\cK_x$ contains a $B_H$-fixed point,  i.e.  
a $B_H$-stable curve $C$.
Then either $C$ is contained in $Z^0 \cdot x$,
or $C$ is a highest weight curve by Lemma 
\ref{lem:curve}(1).

(2) The first assertion is a consequence of Lemma 
\ref{lem:cont}. The second assertion follows from  
\cite[Corollary 2.5]{CFH}. 
\end{proof}

%Recall the notion of Hermitian exceptional symmetric
%space from Subsection \ref{subsec:ass}.

\begin{proposition}\label{prop:min}
Let $X = X_{\ad}$ be a wonderful symmetric variety.

\begin{enumerate}

\item 
A family $\cK$ of rational curves on $X$
consists of minimal curves if and only if 
$\cK_x$ contains a highest weight curve.

\item
If $X$ is indecomposable and not Hermitian,  
then it has  a unique highest weight curve
(and hence a unique family of 
minimal rational curves).

\item
If $X$ is Hermitian, then it has two highest weight curves.
In the non-exceptional case,  these are exchanged
by any element of $\N_G(L) \setminus L$, 
and there is a unique family of minimal rational 
curves.  In the exceptional case,  there are two such
families,  and they are exchanged by an automorphism 
of $X$ fixing $x$. 

\end{enumerate}
\end{proposition}

\begin{proof}
(1) If $\cK$ is a family of minimal rational curves,  
then $\cK_x$ contains a highest weight curve 
by Lemma \ref{lem:red}.

For the converse,  using Lemma \ref{lem:product} and the 
structure of adjoint symmetric spaces (Subsection \ref{subsec:hwc}), 
we may assume that $X$ is indecomposable. 

(2) Since $X$ is not Hermitian, there is a unique highest weight 
curve $C$ by Subsection \ref{subsec:ass}.
Let $\cL$ be a family of minimal rational curves on $X$.
Then $C \in \cL$ by the above step, and hence $\cK = \cL$.

(3) Recall that $\fg^{-\sigma} = \fu_P \oplus \fu_Q$ with
the notation of Subsection \ref{subsec:ass}.
As a consequence, there are two highest weight curves, 
with highest weights $\Theta$ and $- \alpha$.
Consider a Chevalley involution of $(G,T)$,
i.e.,  an involution $\tau$ of $G$ such that
$\tau(t) = t^{-1}$ for all $t \in T$.  Then $\tau$
commutes with $\sigma = \Int(c)$ (see Subsection \ref{subsec:ass} 
again for the definition of $c$),  
since $\tau(c) = c^{-1} = c z$ for some $z \in Z$. 
Thus,  $\tau$ induces an involution $\tau_X$ of $X$ 
fixing $x$.  Also,  $\tau$ sends every root to its opposite;
in particular,  $\tau(\Theta) = - \Theta$.
Choose a representative $g \in \N_L(T)$
of the longest element of the Weyl group of 
$(L,T)$.  Then $g \circ \tau_X$ is an automorphism
of $X$ which fixes $x$ and exchanges the two
highest weight curves (indeed $g\circ\tau_X$ 
maps $\Theta$ to $-\alpha$: 
the involution $\tau_X$ maps the weights of 
$\fu_P$ to the weights of $\fu_Q$ and reverses the order. 
Since $g$ is a representative of 
$w_{0,L}$,  the longest element in $W_L$, 
it reverses the order of the weights in $\fu_P$ and $\fu_Q$; 
thus, $g \circ \tau_X$ maps the highest weight of $\fu_P$ 
to the highest weight of $\fu_Q$).  So each of these curves
is contained in a family of minimal rational curves,  
and hence $\cK$ is such a family.

This completes the proof of the first assertion.
For the second assertion,  it only remains to show 
that there are two families of minimal rational curves
if $X$ is Hermitian exceptional. In this case,  
the map $G/L \to G/P$ extends to a morphism 
$f_P : X \to G/P$. (Indeed, we may view $G/P$ as the closed
$G$-orbit in the projectivization of a simple $G$-module 
$V$ with fundamental highest weight. Let $D$ be the pull-back 
of the $B$-stable hyperplane in $\bP(V)$ under the resulting 
morphism $G/L \to G/P \to \bP(V)$. Then the closure
$\overline{D} \subset X$ is a prime $B$-stable divisor,
and hence the translates $g \cdot \overline{D}$,
where $g \in G$, have no common point as $X$ is toroidal.
The corresponding base-point-free linear system
yields a morphism $X \to \bP(V)$ extending
$G/L \to \bP(V)$).
Moreover, the morphism $f_P$ contracts 
the highest weight curve with weight 
$\Theta$ but not the other one.
So these two curves cannot be in the
same family. In the non-exceptional case, 
they are exchanged by any element of
$\N_G(L) \setminus L$.
\end{proof}

\begin{proposition}\label{prop:red}
Let $\cK$ be a family of minimal rational curves 
on $X$ containing a highest weight curve $C$.
 
\begin{enumerate} 
 
\item
$\cK_x$ consists of embedded free curves
(in particular, of smooth curves). 
Moreover,  $\cK_x$ is smooth and equidimensional, 
of dimension $- K_X \cdot C - 2$. 

\item
There is a unique family of minimal rational curves
$\cL$ on $X_{\ad}$ and a commutative diagram of 
$H$-equivariant rational maps
\[ \xymatrix{
\cK_x \ar[r]^{\pi_{*,x}} \ar[d]_{\tau_x} 
& \cL_{x_{\ad}} \ar[d]^{\tau_{x_{\ad}} }\\
\bP(T_x X) \ar@{-->}[r]^-{d \pi_x} 
& \bP(T_{x_{\ad}} X_{\ad})\\
}\]
where $\tau_x$ and $\tau_{x_{\ad}}$ are finite and
birational onto their image, 
and $\pi_{*,x}$ is a finite morphism. 
If $\pi$ is birational, then $\pi_{*,x}$ is finite and 
birational onto its image as well.

\item
We have
$1 \leq \partial X \cdot C \leq 
\partial X_{\ad} \cdot \pi(C)$. 
Moreover, 
$\partial X \cdot C = \partial X_{\ad} \cdot \pi(C)$
if and only if the image of $\pi_{*,x}$ is a union of
components of $\cL_{x_{\ad}}$.

\item
If each connected component of $\cL_{x_{\ad}}$ is 
a unique $N^0$-orbit,  then
$\pi_{*,x}$ sends each component of $\cK_x$
isomorphically to a component of $\cL_{x_{\ad}}$.
\end{enumerate}

\end{proposition}

\begin{proof}
(1) The open subset $\cK_{\emfr,x}$ is $B_H$-stable, 
and contains every $B_H$-fixed point by Lemma 
\ref{lem:curve}(2) and Lemma \ref{lem:hwc}.
Using Borel's fixed point theorem, 
it follows that $\cK_{\emfr,x}$ is the whole $\cK_x$. 
Thus, $\cK_x$ is smooth; it is equidimensional by 
Lemma \ref{lem:cov}. The assertion on its dimension 
follows from \cite[Section II.3.2]{Kollar}.

(2) By Lemma \ref{lem:curve}(2) again,
$\pi \vert_C$ is birational to its image $D$.
In view of Lemma \ref{lem:dir},
this yields a commutative diagram of rational maps
\[ \xymatrix{
\cK_x \ar@{-->}[r]^{\pi_{*,x}} \ar@{-->}[d]_{\tau_x} 
& \cL_{x_{\ad}} \ar@{-->}[d]^{\tau_{x_{\ad}}} \\
\bP(T_x X) \ar@{-->}[r]^{d \pi_x} & \bP(T_{x_{\ad}} X_{\ad}) \\
}\]
for a unique covering family of rational maps $\cL$ on
$X_{\ad}$.  Since $\cK_x$ is smooth,  $\tau_x$ is a
finite morphism,  birational onto its image
(see \cite{Kebekus} and \cite{ HM}).  Also,
$\pi_{*,x}$ is a morphism since it is a $B_H$-equivariant 
rational map,  defined at every $B_H$-fixed point.  
Moreover, $\cL$ contains the highest weight curve $D$, 
and hence is a
family of minimal rational curves
by Proposition \ref{prop:min}.  Thus,  $\tau_{x_{\ad}}$ is also 
a morphism,  and is finite and birational onto its image
as well.

The rational map 
$d\pi_x : \bP(T_x X) \dasharrow \bP(T_{x_{\ad}} X_{\ad})$
is a linear projection, and hence yields an affine 
morphism on its domain of definition. As a consequence,
the fibers of $\pi_{*,x}$ are affine; thus, $\pi_{*,x}$ 
is a finite morphism.

The final assertion follows from Remark \ref{rem:push}.

(3) Since $X^0$ is affine, $C$ intersects $\partial X$
and hence $\partial X \cdot C \geq 1$. 

By (1), we have
$\dim(\cK_x) = - K_X \cdot C - 2$ and
$\dim(\cL_{x_{\ad}}) = - K_{X_{\ad}} \cdot \pi(C) - 2$. 
Since $\pi_{*,x} : \cK_x \to \cL_{x_{\ad}}$ is finite,
it follows that 
\[ K_X \cdot C \geq K_{X_{\ad}} \cdot \pi(C), \]
with equality if and only if the image of $\pi_{*,x}$ is 
a union of components of $\cL_{x_{\ad}}$. Moreover,
\[ (K_X + \partial X) \cdot C  = 
(K_{X_{\ad}} + \partial X_{\ad}) \cdot \pi(C) \]
by Lemma \ref{lem:can} and the projection formula. 
This yields the remaining statements.

(4) By assumption, each component of $\cL_{x_{\ad}}$
is homogeneous under $N^0$, and hence under $H^0$
(Lemma \ref{lem:nor}). The corresponding isotropy group
is a parabolic subgroup of $H^0$; thus, it is connected.
As $\pi_{*,x}$ is finite and $H^0$-equivariant, 
this yields the assertion.
\end{proof}

\begin{example}\label{ex:quadric}
Let $G = \SO_n$, where $n \geq 3$, and let $\sigma$ be 
the conjugation by $c = \diag(1, \ldots, 1,-1) \in \OO_n$. 
Let $H = G^{\sigma,0} = \SO_{n-1}$. Then 
$N = G^{\sigma} = \OO_{n-1}$ embedded in $\SO_n$
via $g \mapsto (g, \det(g))$, and $\fg^{-\sigma} = \bC^{n-1}$
on which $\OO_{n-1}$ acts via its standard representation.
Also, $G/H$ has a unique smooth projective equivariant 
embedding: the quadric 
$\Q_{n-1} \subset \bP^n = \bP(\bC^n \oplus \bC)$, 
where $\SO_n$ acts on $\bP(\bC^n \oplus \bC)$ via its
standard representation on $\bC^n$. 
Moreover, $X_{\ad} = \bP^{n-1} = \bP(\bC^n)$ and
$\pi : X \to X_{\ad}$ is a ramified double cover induced 
by the linear projection
$\bP(\bC^n \oplus \bC) \dasharrow \bP(\bC^n)$.

If $n \geq 4$ then $X$ has a unique family of 
minimal rational curves $\cK$; it consists of the lines 
in $\Q_{n-1}$.  Moreover,  $\pi_*$ sends $\cK$
to the family $\cL$ of lines in $\bP^{n-1}$, and
$\pi_{*;x} : \cK_x \to \cL_{x_{\ad}}$ is identified
with the inclusion 
$\Q_{n-3} \subset \bP^{n-2} = \bP(\fg^{-\sigma})$,
compatibly with the action of $\OO_{n-1} = N$.

If $n = 3$ then $X = \Q_2 \simeq \bP^1 \times \bP^1$
has two families of minimal rational curves, the fibers
of the two projections to $\bP^1$.  For both families,
$\pi_{*,x}$ identifies $\cK_x$ with a point in $\bP^1$.
\end{example}

\subsection{Minimal rational curves for group and Hermitian types}
\label{subsec:mfght}

We still consider a complete symmetric variety
$X$ with base point $x$, and a family of minimal rational
curves $\cK$ on $X$; we assume that $\cK$
contains a highest weight curve $C$.  
By Lemma \ref{lem:product},  Lemma \ref{lem:red}
and Proposition \ref{prop:min},
 there is a unique 
indecomposable factor $X_C$ of $X_{\ad}$ such that the composition of 
$\pi: X \to X_{\ad}$ with the projection $X_{\ad} \to X_C$ sends $C$ 
isomorphically to its image.  In this subsection,  we will handle in details 
the cases where $X_C$ is of group or Hermitian types.

We first handle the group type, where $X_C$ is the 
wonderful completion of an adjoint simple group $H_C$.
The Lie algebra of $H_C$ is denoted by $\fh_C$, and we 
still denote by $C$ the highest weight curve in $X_C$.
By the main result of \cite{BF}, 
$X_C$ has a unique family of minimal rational
curves $\cL$.  Moreover,  the tangent map
$\tau_{x_C} : \cL_{x_C} \to \bP(\fh_C)$
is an $H_C$-equivariant isomorphism to its image
$\cC_{x_C}$.  
If $H_C$ is of type $\A_r$, 
i.e.,  $H_C \simeq \PGL(V)$ where $V$ is a vector 
space of dimension $r + 1$, 
then $\fh_C \simeq \End(V)/\bC \, \id$. 
When $r = 1$, we have $X_C = \bP(\End(V))$
and hence $\cC_{x_C} = \bP(\fh_C)$.  On the
other hand, when $r \geq 2$, the VMRT
$\cC_{x_C}$ is isomorphic to 
$\bP(V) \times \bP(V^\vee)$ embedded in $\bP(\fh_C)$
via the Segre embedding 
\[ \bP(V) \times \bP(V^\vee) \hookrightarrow 
\bP(V \otimes V^\vee) = \bP(\End(V)) \]
followed by the linear projection
$\bP(\End(V)) \dasharrow \bP(\End(V)/\bC \, \id)$.
In all other types, we have 
$\cC_{x_C} = \bP(\cO_{C,\min})$, 
the projectivization of the minimal nilpotent 
orbit in $\bP(\fh_C)$.

\begin{proposition}\label{prop:group}
If $X_C$ is of group type,  then
$\partial X \cdot C$ equals $1$ or $2$. 
In the former case,  every component of $\cK_x$ is 
isomorphic to $\bP(\cO_{C,\min})$. 
In the latter case, we have $H_C \simeq \PGL(V)$;
moreover, every component of $\cK_x$ is isomorphic 
to $\bP(V) \times \bP(V^\vee)$ if $\dim(V) \geq 3$,
and to $\bP(\fh_C) \simeq \bP^2$ if 
$\dim(V) = 2$.
\end{proposition}

\begin{proof}
Recall from Proposition \ref{prop:red} that
$1 \leq \partial X \cdot C \leq 
\partial X_{\ad} \cdot \pi(C)$.
Moreover, the line bundle on $X_{\ad}$ associated with
the divisor $\partial X_{\ad}$ equals 
$\cL_{X_{\ad}}(\alpha_1 + \cdots + \alpha_r)$
with the notation of \cite[Section 3]{BF}. By using
\cite[Lemmas 3.3 and 3.4]{BF}, it follows that
$\partial X_{\ad} \cdot \pi(C) = 2$ if $H_C$ is 
of type $\A_r$; otherwise, 
$\partial X_{\ad} \cdot \pi(C) = 1$.

In the latter case, we must have 
$\partial X \cdot C = \partial X_{\ad} \cdot \pi(C)$.
So every component of $\cK_x$ is isomorphic to 
the orbit $H_C \cdot C$, by Proposition \ref{prop:red}
again. Moreover, $H_C \cdot C = \bP(\cO_{C,\min})$.

In the former case, if $r \geq 2$ then 
$\cL_{x_C} = \bP(V) \times \bP(V^\vee)$
consists of two orbits of $H_C = \PGL(V)$: 
a closed orbit of codimension $1$ (the incidence variety,
isomorphic to $\bP(\cO_{C,\min})$), 
and an open orbit isomorphic to 
$\SL_{r + 1}/\GL_r$, and hence simply 
connected. If $r = 1$ then $\cL_{x_C} \simeq \bP^2$
is simply connected as well.

If $\partial X \cdot C = 2$,  then by Proposition
\ref{prop:red} again, we get a finite
surjective $H$-equivariant morphism
$\pi_{*,x}: \cK_x \to 
\bP(V) \times \bP(V^\vee)$.
Since the open orbit in the right-hand side is
simply connected,  it follows that
$\pi_{*,x}$ is birational on each component,
and hence an isomorphism in view of Zariski's main
theorem.

On the other hand, if $\partial X \cdot C = 1$,
then the image of $\pi_{*,x}$ is the closed orbit
and we conclude as above.
\end{proof}

\begin{example}\label{ex:group}
Assume that $\pi: X \to X_{\ad}$ is birational and
$X_{\ad}$ is the wonderful completion of $\PGL(V)$, 
where $\dim(V) \geq 3$.  Then the highest weight
curve $C_{\ad} \subset X_{\ad}$ intersects a unique 
$\PGL(V) \times \PGL(V)$-orbit $\cO_{1,r}$ 
of codimension $2$ in $X_{\ad}$ (see 
\cite[Lemma 3.4]{BF}). 

If $\pi$ is an isomorphism over $\cO_{1,r}$, 
then the family of minimal rational curves $\cK$ on $X$ 
satisfies $\partial X \cdot C = 2$ and 
$\cK_x = \bP(V) \times \bP(V^\vee)$.  Indeed, 
$\pi$ is an isomorphism over an open neighborhood
of $\cO_{1,r}$ in $X_{\ad}$,  stable by
$\PGL(V) \times \PGL(V)$, and every curve
in $\cL_{x_{\ad}}$ intersects such a neighborhood.

On the other hand, if $\pi$ is not an isomorphism
over $\cO_{1,r}$, then $\partial X \cdot C = 1$
and $\cK_x$ is the incidence variety $\bP(\cO_{\min})$;
moreover,  we have
$\partial X_{\ad} \cdot C_{\ad} = 2$ and 
$\cL_{x_{\ad}} = \bP(V) \times \bP(V^\vee)$.  
Indeed, $\pi$ factors  through the blow-up 
$\varphi : X' \to X_{\ad}$ of $\overline{\cO_{1,r}}$
in $X_{\ad}$.  Using Proposition \ref{prop:red}, 
we may thus assume that $X = X'$.  Then
$K_X = \pi^*(K_{X_{\ad}}) + E$, 
where $E$ denotes the exceptional divisor.  
Thus,
\[ K_X \cdot C = 
K_{X_{\ad}} \cdot C_{\ad} + E \cdot C
> K_{X_{\ad}} \cdot C_{\ad}, \]
since $C$ intersects $E$. 
It follows that $\dim(\cK_x) < \dim(\cL_{x_{\ad}})$,
and we conclude by Proposition \ref{prop:red} again.
\end{example}

Next, we handle the Hermitian type,  where 
$X_C$ is the wonderful completion of 
the symmetric space $G_C/\N_{G_C}(L_C)$,
where $G_C$ is a simple factor of $G_{\ad}$
and $L_C$ is a Levi subgroup of  $G_C$.

\begin{proposition}\label{prop:hermitian}
If $X_C$ is of Hermitian type but not of type
$\PGL_2/N$,  then $\partial X \cdot C = 1$ and every 
component of $\cK_x$ is isomorphic to the orbit
$L_C \cdot C$.
\end{proposition}

\begin{proof}
By Proposition \ref{prop:red}, we may assume that
$X = X_C$,  and hence $G = G_C$ and 
$L = L_C=P \cap Q$,  where $P,Q$ are
opposite parabolic subgroups and $P \supset B$.
We now view $X$ as a subvariety
of  $ \overline{G}$ (as recalled in Subsection \ref{subsec:rwsv}),
and use the description of minimal rational curves 
in $\overline{G}$,  as in the proof of Proposition 
\ref{prop:group}.

With the notation of Subsection \ref{subsec:hwc}, 
the highest weight curves are 
$C_{\Theta} := \overline{U_{\Theta} \cdot x}$ and 
$C_{-\alpha} := \overline{U_{-\alpha} \cdot x}$
(indeed,  these curves are irreducible,  stable by 
$B_L$ and distinct).  By Proposition \ref{prop:min},  
these curves are exchanged 
by an automorphism of $X$ fixing $x$; thus, we may 
assume that $C = C_{\Theta}$.

By \cite[Section 5.5]{RRS}, $\sigma$ is the inner 
involution $\Int(c)$, where $c \in T$ satisfies 
$\alpha(c) = -1$ and $\beta(c) = 1$ for all simple 
roots $\beta \neq \alpha$. 
In particular, the roots $\Theta$ and $-\alpha$ are 
non-compact imaginary. Thus, the closed immersion 
\[ \iota : G/\N_G(L) \longrightarrow G,  \quad
g \N_G(L) \longmapsto \sigma(g) g^{-1} \]
induces isomorphisms 
\[ U_{\Theta} \cdot x 
\stackrel{\sim}{\longrightarrow} U_{\Theta},
\quad U_{-\alpha} \cdot x 
\stackrel{\sim}{\longrightarrow} U_{-\alpha}. \]
So $\overline{\iota} : X \to \overline{G}$  
sends $C_{\Theta}$, $C_{-\alpha}$ isomorphically 
to the corresponding root curves considered in 
\cite[Section 3]{BF}. 
Since $\Theta$ and $-\alpha$ are long roots, these root 
curves are minimal; hence $\overline{\iota}$ sends 
$\cK_x$ to the unique family $\cL_{\overline{G}, e}$ 
of minimal rational curves through $e$ in 
$\overline{G}$. 
Moreover, $\overline{\iota}(\cK_x)$ is contained in 
the fixed locus $\cL_{\overline{G}, e}^{-\sigma}$.

If $G$ is not of type $\A_r$, where $r \geq 2$, 
then the tangent map $\tau_e$ identifies 
$\cL_{\overline{G}, e}$ with $\bP(\cO_{\min})$.  
Since $d \iota_x$ identifies $T_x X$ with $\fg^{-\sigma}$,
we see that 
\[ \overline{\iota}(\cK_x) \subset 
\bP(\cO_{\min} \cap \fg^{-\sigma}) \subset 
\bP(\cO_{\min})^{\sigma}. \]
By \cite[Theorem A]{Richardson}, the right-hand side
is a finite union of closed orbits of 
$G^{\sigma,0} = L$. We conclude that the component
of $C$ in $\cK_x$ is $L \cdot C$.

Otherwise, $G = \PGL(V)$ where $\dim(V)  = r + 1$,
and $\tau_e$ yields an isomorphism
\[ 
\cL_{\overline{G}, e} 
\simeq \bP(V) \times \bP(V^\vee) \subset 
\bP(\End(V)/\bC \,  \id) = \bP(\fg), \]
equivariantly for the action of $-\sigma$. 
Consider  the $\sigma$-eigenspace decomposition 
$V = V_1 \oplus V_{-1}$. Then $L$ is the image of
$\GL(V)^{\sigma} = \GL(V_1) \times \GL(V_{-1})$ 
in $\PGL(V)$;  also,  we have
$\bP(V)^\sigma = \bP(V_1) \sqcup \bP(V_{-1})$
and likewise for $\bP(V^\vee)^\sigma$. 
Moreover, the image of 
$\overline{\iota}(C)$ under $\tau_e$ lies  in 
\[ \bP(\fg^{-\sigma}) = 
\bP(\Hom(V_1,V_{-1}) \oplus \Hom(V_{-1}, V_1)). \]
It follows that 
$\overline{\iota}(\cK_x)$ is contained in 
$(\bP(V_1^\vee) \times \bP(V_{-1})) \sqcup  
(\bP(V_{-1}^\vee) \times \bP(V_1))$. 
As a consequence, 
the component of $C$ in $\cK_x$ is $L \cdot C$ 
in this case,  too.

We now show that $\partial X \cdot C = 1$.
Consider first the exceptional case,  where $H = L = P \cap Q$.
Then we have a $G$-equivariant birational morphism
$\varphi: X \to G/P \times G/Q =: Y$ which sends
$x$ to the base point $y = (P,Q)$.
Since $P$ is a maximal parabolic subgroup of $G$ 
associated with a long root, $G/P$ has 
a unique family of minimal rational curves $\cL$.
Moreover,  denoting by $P$ the base point of
the homogeneous space $G/P$ and by $D$ 
the Schubert line in that space (i.e., the unique 
irreducible $B$-stable curve),  we have that 
$\cL_P = L \cdot D$ (see e.g. \cite[Proposition 3.3]{BK}). 
The projection $p : X \to G/P$ sends $C$ isomorphically 
to $D$, and yields an isomorphism 
$p_* : L \cdot C \to L \cdot D$ 
which identifies $L \cdot C$ with the variety 
of lines in $G/P$ through its base point. Since 
$\dim(L \cdot C) = \dim(\cK_x) = - K_X \cdot C - 2$
and $\dim(L \cdot D) = -K_{G/P} \cdot D - 2$, 
we obtain
\[ K_X \cdot C = K_{G/P} \cdot D = 
K_Y \cdot \varphi(C) = \varphi^*(K_Y) \cdot C \]
by using the projection formula.
On the other hand, we have 
$K_X = \varphi^*(K_Y) + \sum_i a_i E_i$, where the $E_i$
are the exceptional divisors of $\varphi$ and the $a_i$
are positive integers. Since $C$ is not contained in any
$E_i$, it follows that $E_i \cdot C = 0$ for all $i$.
Also, the boundary of $Y$ is an irreducible divisor $E$,
and $\partial X = E' + \sum_i E_i$, where $E'$ denotes the
strict transform of $E$. This yields
$\partial X \cdot C = E' \cdot C = E \cdot D$
by the projection formula again. Since 
$D \subset G/P \times \{ Q \}$, where we still denote by $Q$ 
the base point of $G/Q$, and $E \cap (G/P \times \{ Q \})$ 
is identified with the Schubert divisor in $G/P$, we obtain 
$E \cdot D = 1$. This yields the assertion in that case.   

Next,  we consider the non-exceptional case, 
where $H = \N_G(L)$ contains $L$ as 
a subgroup of index $2$. 
By \cite[Sections 29.1 and 29.2]{Timashev}, 
there exists a smooth 
toroidal equivariant embedding $X'$ of $G/L$ such that the
natural map $G/L \to G/H$ extends to a morphism 
$\psi: X' \to X$. By Lemma \ref{lem:can}, we have
$K_{X'} + \partial X' = \psi^*(K_X + \partial X)$.
Moreover, $C$ lifts uniquely to a highest weight curve
$C' \subset X'$, and the corresponding families of
minimal rational curves
have isomorphic components by Proposition \ref{prop:red}.
Taking dimensions, we obtain 
$K_{X'} \cdot C' = K_X \cdot C = \psi^*(K_X) \cdot C'$. 
As a consequence, we have
$\partial X \cdot C = \psi^*(\partial X) \cdot C' 
=  \partial X' \cdot C' = 1$,
where the latter equality is proved as in the exceptional
case.
\end{proof}

Proposition \ref{prop:hermitian} leaves out the case of 
type $\PGL_2/N$,  which is easily treated:

\begin{lemma}\label{lem:plane}
If $X_C$ is of type $\PGL_2/N$,  then $\partial X \cdot C$ 
equals $1$ or $2$.  In the former case, $\cK_x$ is finite. 
In the latter case,  every component of $\cK_x$ is 
a projective line.
\end{lemma}

\begin{proof}
Note that $\PGL_2/N$ has a unique projective 
equivariant embedding, namely, $\bP^2$ on which 
$\PGL_2$ acts via the projectivization of its 
adjoint representation. Thus, $\partial X_C$ 
is a conic,  with $C$ as a tangent line so that
$\partial X_C \cdot C = 2$. Also, the minimal rational
curves on $X_C$ are just lines, and those through
a given point form a $\bP^1$. This yields the statement
by using Proposition \ref{prop:red} as in the proof
of the above proposition.
\end{proof}

\begin{corollary}\label{cor:hermitian}
Let $G$ be a simple adjoint group,  and $X$ 
the wonderful embedding of a Hermitian 
symmetric space $G/\N_G(L)$. 
Denote by $C_{\Theta}$ and $C_{-\alpha}$
the highest weight curves in $X$, 
indexed by their weight.

\begin{enumerate}

\item
If $X$ is exceptional,
then it has two families of minimal rational curves
$\cK^+$, $\cK^-$.  Moreover,
$\cK_x^+ = L \cdot C_{\Theta}$ and
$\cK_x^- = L \cdot C_{-\alpha}$.

\item
If $X$ is non-exceptional,  
then it has a unique family of minimal rational curves $\cK$. 
Moreover,  $\cK_x = \N_G(L) \cdot C_{\Theta} 
= L \cdot C_{\Theta} \sqcup L \cdot C_{-\alpha}$
unless $G/\N_G(L) = \PGL_2/N$.
\end{enumerate}
\end{corollary}

\begin{proof}
The two highest weight curves $C_{\Theta}$
and $C_{- \alpha}$ are exchanged by an element 
of $\N_G(L)$ in the non-exceptional case (see 
Proposition \ref{prop:min}) while in the exceptional case, 
they lie in two distinct families by Proposition \ref{prop:min} again. 
We conclude by using Proposition \ref{prop:hermitian} and
Lemma \ref{lem:plane}.
\end{proof}

\subsection{Some cases of simple type}
\label{subsec:some-cases}

In this subsection we use previous techniques and the results 
in the group case to briefly describe the unique family of
minimal rational curves
in some of the cases of simple type. 
We refer to Table \ref{table-class} for the different cases 
of the classification. 
A different approach, working in all cases is developed in
Section \ref{sec:cas-adjoint}.

We assume that $G$ is simple throughout this subsection.
Consider a highest weight curve $C$ of simple type,
and denote by $\lambda$ its weight relative to $B_H$.
Then $\lambda$ is the highest weight of the representation
of $H^0$ in $\fg^{-\sigma}$, and hence is the restriction to $T_H$ of
some $\alpha \in R^+$ (not necessarily unique). Moreover,
$S := \ker(\lambda)^0$ is a subtorus of codimension $1$
in $T_H$, and $C$ is an irreducible curve in $(G/H)^S$ 
through $x$,  stable by the Borel subgroup $\C_{B_H}(S)$ 
of $\C_H(S)$.  So $C$ is a highest weight curve of 
the symmetric space
$\C_G(S)/\C_H(S) = (\C_G(S)/S)/(\C_H(S)/S)$, 
where $\C_H(S)/S$ has rank $1$.
We now apply Lemma \ref{lem:rank1}: the adjoint symmetric 
space of $\C_G(S)/S$ is of type $(\textrm{A}_1)$, 
$(\textrm{A}_1 \times \A_1)$ or $(\textrm{A}_2)$. 
If $\sigma$ is inner,  then only type $(\textrm{A}_1)$ may occur.

\begin{lemma}\label{lem:redu}
Assume that $\lambda = \alpha \vert_{T_H}$ 
for a unique root $\alpha$ (i.e.,  type $(\A_1 \times \A_1)$ 
is excluded),  and $\alpha$ is long.  Then the component 
of $\cK_x$ containing $C$ admits a finite
equivariant morphism to 
$\cK_{\overline{G_\ad},\id}^{-\sigma}$.
\end{lemma}

\begin{proof}
By assumption,  $C = \overline{U_\alpha \cdot x}$, 
where $\alpha$ is non-compact imaginary.
Thus,  the image of $C$ under the
morphism $\psi : X \to \overline{G_\ad}$, obtained by
composing $\pi: X \to X_\ad$ with 
$\overline{\iota} : X_ \ad \to \overline{G_\ad}$, 
is just the closure 
of $U_\alpha$; since $\alpha$ is long, this is a minimal 
rational curve on $\overline{G_\ad}$. This yields the
assertion by arguing as in the proof of Proposition
\ref{prop:red}.
\end{proof}

The assumptions of the lemma hold if and only if 
$T_xC \setminus \{ 0 \}$ is contained in $\cO_{\min}$
(as follows by combining 
Proposition \ref{prop:roots}, Corollary \ref{cor:alter}
and Lemma \ref{lem:orbits}).

\begin{proposition}\label{prop:cases}
Assume that 
$T_x C \setminus \{ 0 \} \subset \cO_{\min}$.
\begin{enumerate}
\item If $\partial X \cdot C = 1$, then the component of $\cK_x$ containing $C$ is $H^0 \cdot C$.
\item If $\partial X \cdot C = 2$, then $X$ is of type {\rm AI} with $G = \PGL_{r+1}$ and $\cK_x \simeq \bP^r$.
\end{enumerate}
\end{proposition}

\begin{proof}
If $G$ is not of type $\A_r$ or if $X$ is Hermitian (but not of type $\PGL_2/N$), we may argue as in the proof of Proposition \ref{prop:hermitian} proving that the component of $\cK_x$ containing $C$ is $H^0 \cdot C$ and that $\partial X \cdot C = 1$. 

If $G$ is of type $\A_r$ and $X$ is not Hermitian, then $X$ is of type AI, with $G = \PGL(V)$ such that $\dim V= r+1$, $r \geq 2$ and $\sigma(g) = (g^t)^{-1}$. In this case, the family of minimal rational curves
$\cK_{\overline{G_\ad},\id}$ identifies with $\bP(V) \times \bP(V^\vee)$ and the involution $-\sigma$ acts via $(-\sigma)([v],[H]) = ([H^\perp],[v^\perp])$ where the orthogonality is taken with respect to the standard scalar product. We thus have $\cK_{\overline{G_\ad},\id}^{-\sigma} \simeq \bP(V)$. 
If $\partial X \cdot C = 1$, then $\dim \cK_x = \dim \bP(V) - 1 = \dim H \cdot C$ and the result follows as above. If $\partial X \cdot C = 2$, then $\dim \cK_x = \dim \bP(V)$ proving the result.
\end{proof}

\begin{remark}
In Table \ref{table-class}, we list the nilpotent orbits containing $T_x C \setminus \{ 0 \}$ (see the column``$\sigma(\Theta)  = -\Theta$", the condition $T_x C \setminus \{ 0 \} \subset \cO_{\min}$ being equivalent to $\sigma(\Theta)  = -\Theta$ by Corollary \ref{cor-hwc}).
In particular, the above proposition settles all cases except the following symmetric spaces: AII, BII, CII, DII, EIV and FII. We will deal with all cases in the next section via a different approach.
\end{remark}

\section{Minimal rational curves on wonderful symmetric varieties}
\label{sec:cas-adjoint}

In this section, we deal with wonderful embeddings 
of adjoint indecomposable symmetric spaces.
These are of the form $G_{\ad}/H$, where $G_{\ad}$ is an 
adjoint semisimple group and $H$ is the fixed point
subgroup of an involution $\sigma$. Then $\sigma$ lifts 
to a unique involution of the universal cover
$G = G_{\rm sc}$ of $G_{\ad}$, that we still denote by $\sigma$
for simplicity; moreover, $G^\sigma$ is connected 
(see \cite{Steinberg} for these results). 
As seen in Subsection \ref{subsec:ass}, we have
$G_{\ad}/H = G/N$, where
$N = \N_G(G^\sigma)$; moreover, $G/N$ is of group,  
Hermitian, or simple type.  We will consider its 
wonderful embedding $X = X_\ad$.

We start with reminders on restricted root systems 
(Subsection \ref{subsection:rrs}) and their connection 
to curves and divisors on $X$ (Subsection 
\ref{subsec:div-rrs}). 
Many results on these topics are well known but we could not find a good reference, so we included proofs for the convenience of the reader.  From this we obtain an explicit description of the classes of 
minimal rational curves in $X$
(Subsection \ref{subsection:classes-vmrt}). 
We then compute the dimension of these families of minimal
rational curves $\cK$ using the contact structures on projectivised 
nilpotent orbits (Subsection \ref{sec:contact}). It turns out that in all cases except for $X$ of restricted type $\textrm{A}_r$,  the family $\cK_x$ has the same dimension as the orbit $N \cdot C$ where $C \in \cK_x$ is a highest weight curve, which in turn implies that $\cK_x = N \cdot C$. 
We deal with $X$ of restricted type $\textrm{A}_r$ separately (Subsection \ref{subsection-typea}). 
We conclude with a full description of $\cK_x$ 
(Subsection \ref{subsection-summary}).

\subsection{Restricted root system}
\label{subsection:rrs}

Let us first recall a few facts on the restricted root system;
we refer to \cite{Vust2} and \cite{Timashev} for details. 
Let $T_{\s}$ be a maximal torus of split type,  and 
$S \subset T_{\s}$ its maximal split subtorus: 
$S = \{t \in T_{\s} \ | \ \sigma(t) = t^{-1} \}^0$. 
Let $R$ be the root system associated to the pair 
$(G,T_{\s})$.  Then $\sigma$ acts on $R$. 
Set $\Sb = S/S^\sigma$,  $\Xb = \fX(\Sb)$ and 
$\chib = \chi - \sigma(\chi)$ for $\chi \in \fX(T_{\s})$. 
We have an identification 
$\Xb = \{\chib \ | \ \chi \in \fX(T_{\s}) \}$. 
Define the subset $\Rb \subset\Xb$ via
$$\Rb = \{ \alphab \ | \ \alpha \in R \}.$$
Then $\Rb$ is an irreducible root system called the
\emph{restricted root system}.  It may be non-reduced 
(see Remark \ref{remark : excep-reduce} below).

Recall from Subsection \ref{subsec:ss} 
that $L = \C_G(S)$ is the Levi subgroup 
containing $T_{\s}$ of a parabolic subgroup $P \subset G$ 
and that $\sigma(P)$ is the opposite parabolic subgroup 
to $P$ with common Levi subgroup $L$. 
Let $B_{\s} \subset P$ be a Borel subgroup and let 
$\Delta \subset R^+ \subset R$ be the sets of simple roots 
and positive roots defined by $B_{\s}$. 
Then for $\alpha \in R^+$, 
we have $\sigma(\alpha) = \alpha$ if and only if $\alpha$ 
is a root of $L$; moreover, if 
$\sigma(\alpha) \neq \alpha$,  then 
$\sigma(\alpha) < 0$.  Set 
$\Delta_1 = \{ \alpha \in \Delta \ | \ \sigma(\alpha) < 0 \}$ 
and $\Delta_0 = \Delta \setminus \Delta_1$. 
Then $\sigma(\alpha) = \alpha$ for any 
$\alpha \in \Delta_0$. Define $\Db \subset\Xb$ via
$$\Db = \{ \alphab \ | \ \alpha \in \Delta_1 \}.$$
Then $\Db$ is a basis of $\Rb \subset\Xb$. In particular 
$|\Db| = \rk(\Xb) = \dim T_{\s} = r$ is the rank of $X$. 
Furthermore,  there exists a length-preserving involution 
$\sigmab$ on $\Delta$, preserving $\Delta_1$ and acting 
as $-w_{0,L}$ on $\Delta_0$ 
(where $w_{0,L}$ is the element of maximal length in $W_L$, the Weyl group of $L$), such that for any $\alpha \in \Delta_1$, we have
$$\sigma(\alpha) + \sigmab(\alpha) 
= - \sum_{\beta \in \Delta_0} c_\beta\beta$$
with $c_\beta \in \bZ_{\geq0}$ (see \cite[Section 1.5]{decs} for these facts).  In particular,  if 
$\sigmab(\alpha) \neq \alpha$,  then
$\scal{\alpha^\vee,\sigma(\alpha)} \geq 0$. 
Note that for $\alpha, \beta \in \Delta_1$,  
we have 
$\alphab = \betab \Leftrightarrow (\beta 
= \alpha \textrm{ or } \beta = \sigmab(\alpha))$.

A root $\alpha \in \Delta_1$ such that 
$\sigmab(\alpha) \neq \alpha$ and 
$\scal{\alpha^\vee,\sigma(\alpha)} \neq 0$ 
is called 
\emph{exceptional}. 
If $\alpha$ is exceptional,  then 
$\sigma(\alpha) \neq - \alpha$.  Moreover, 
$\beta = \sigmab(\alpha) \neq \alpha$ is also exceptional
and one of the following two conditions is satisfied: 
either $\sigma(\alpha) \neq - \beta$ and 
$\scal{\alpha^\vee,\beta} = 0$, 
or $\sigma(\alpha) = -\beta$ and 
$\scal{\alpha^\vee,\beta} \neq 0$ (see \cite[Lemma 4.3]{decs}). 
If there exists an exceptional root,  then $\Rb$ and $X$ are called \emph{exceptional}.  This definition is equivalent
to the one given in Subsection \ref{subsec:ass}, 
see for example \cite[Lemma 4.7]{decs}. 
Note that by loc.~cit., 
there are at most two exceptional roots 
(thus, of the form 
$\alpha$ and $\sigmab(\alpha)$). 

\begin{remark}\label{remark : excep-reduce}
If $\Rb$ is exceptional,  it 
is non-reduced. 
In fact, for $\alpha$ exceptional, we have
$\scal{\alpha^\vee,\sigma(\alpha)} \neq 0$. 
As $\alpha$ and $\sigmab(\alpha)$ are different 
but of the same length,  we have 
$\scal{\alpha^\vee,\sigma(\alpha)} = 1$. 
Thus, $\gamma = \alpha - \sigma(\alpha) \in R$ and 
$\gammab = 2\alphab$.  In particular, 
$\alphab, 2\alphab \in \Rb$ and $\Rb$ is non-reduced.
\end{remark}

\begin{example}
There are non-reduced restricted root systems which are 
non-exceptional (actually only two families: types CII and FII, 
see Appendix).  For example,  if $G = \Sp_{6}$, 
there exists an involution $\sigma$ such that 
$G^\sigma = \Sp_2 \times \Sp_4$ and with the labeling 
of the simple roots as in Bourbaki \cite{bourbaki}, 
we have $\Delta_0 = \{\alpha_1,\alpha_3 \}$ and 
 $\Delta_1 = \{\alpha_2 \}$. Set $\alpha = \alpha_2$,  then 
 $\sigmab(\alpha) = \alpha$, $\sigma(\alpha) 
 = - (\alpha_1 + \alpha_2 + \alpha_3)$ 
 and 
 $\alphab = \alpha_1 + 2\alpha_2 + \alpha_3 = \gamma \in R$. 
 We have $\gammab = 2\alphab$; thus,  
 $\Rb = \{-2\alphab, -\alphab, \alphab, 2\alphab\}$ 
 but $G/G^\sigma$ is not exceptional.
\end{example}

Let $\alpha \in R^+$ such that $\sigma(\alpha) < 0$. 
The roots $\alpha$ and $\sigma(\alpha)$ have the same length. 
As explained in \cite[Lemme 2.3]{Vust2},  three cases occur 
and the coroot $\alphab^\vee$ is defined accordingly:
  \begin{enumerate}
  \item If $\sigma(\alpha) = -\alpha$,  then 
  $\alphab^\vee = \frac{1}{2} \alpha^\vee$.
  \item If $\scal{\alpha^\vee,\sigma(\alpha)} = 0$,  then  
  $\alphab^\vee = \frac{1}{2} (\alpha^\vee - \sigma(\alpha)^\vee)$.
    \item If $\scal{\alpha^\vee,\sigma(\alpha)} = 1$,  then 
    $\alphab^\vee = \alpha^\vee - \sigma(\alpha)^\vee$.
  \end{enumerate}

Case (3) above actually occurs if and only if $\Rb$ is non-reduced, see Proposition \ref{prop:rrs-summary}(5) below. In the next proposition, we summarise the results on restricted root systems needed for the study of curves and divisors on $X$. These results might be well known to the experts but we could not find a good reference, so we included a proof and further results on restricted root systems in Subsection \ref{ap:rrs} in the Appendix.

\begin{proposition}
\label{prop:rrs-summary}
Let $\Theta$ be the highest root of $R$ and $w_0 \in W$ be the longest element.
\begin{enumerate}
\item $\Thetab$ is the highest root of $\Rb$, the actions of $\sigma$ and $w_0$ on roots commute, and we have 
$w_0(\Thetab) = -\Thetab$.
\item If $\sigma(\Theta) \neq -\Theta$, then $\Theta$ and $\sigma(\Theta)$ are strongly orthogonal long roots. 
\item If $\alpha \in \Delta_1$ is exceptional, its coefficient in the expansion of $\Theta$ in simple roots is $1$.
\item If $\Rb$ is not of type $\A_1$, then there exists $\alphab \in \Db$ with $\langle \Thetab^\vee,\alphab \rangle = 1$ and $\Thetab^\vee$ is an indivisible cocharacter of $\Sb$.
\item For $\alpha \in \Delta_1$, we have the equivalence: $\alphab, 2\alphab \in \Rb$ 
    $\Leftrightarrow$ $\scal{\alpha^\vee,\sigma(\alpha)} = 1$.
\end{enumerate}
\end{proposition}

\begin{proof}
(1) This is Lemma \ref{lem:highest-root-ap}, 
Lemma \ref{lemm:w_0-et-sigma} and Corollary \ref{cor-w_0-Theta-ap}. 

(2) This is Proposition \ref{prop-primitif-ap}(4).

(3) This is the last statement in Corollary \ref{coro-coef1-ap}. 

(4) This is Proposition \ref{prop-primitif-ap}(2)-(3). 

(5) This is Proposition \ref{char-non-red-ap}.
\end{proof}

We end this subsection with a piece of notation. For $\alphab \in \Db$,  we denote by $\alphah^\vee$ 
the simple root of $\Rb^\vee$ colinear to $\alphab^\vee$. 
Note that if $2\alphab \not\in \Rb$,  then 
$\alphah^\vee = \alphab^\vee$ but if $2\alphab \in \Rb$, 
we have 
$\alphah^\vee = \frac{1}{2}\alphab^\vee = (2\alphab)^\vee$. 
In particular,  for $\alpha \in \Delta_1$, 
Proposition \ref{prop:rrs-summary}(5) implies that 
$$\alphah^\vee = \left\{
\begin{array}{ll}
  \alphab^\vee & \textrm{ if $\scal{\alpha^\vee,\sigma(\alpha)} \neq 1$}, \\
  \frac{1}{2}\alphab^\vee & \textrm{ if $\scal{\alpha^\vee,\sigma(\alpha)} = 1$.} \\
\end{array}
\right.$$

We will also need the following result proved in Lemma \ref{lem-coef-rac-ap}.

\begin{lemma}
  \label{lem-coef-rac}
Assume that $\alpha \in \Delta_1$ is an exceptional root. 
Then the coefficient of $\alphah^\vee$ in the expansion of 
$\Thetab^\vee$ in terms of simple coroots of $\Rb$ is equal to $1$.
\end{lemma}

\subsection{Divisors and restricted root system}
\label{subsec:div-rrs}

We relate the Picard group of $X$ (viewed as the group of divisors up to linear equivalence), to the restricted root system $\Rb$. We will need some definitions from the theory of spherical varieties, we refer to \cite{survey} 
or \cite[Section 17]{Timashev}
for further details. The variety $X$ is spherical: it is a normal $G$-variety such that $B_{\s}$ has a dense orbit. This implies that $B_{\s}$ acts on $X$ with finitely many orbits. In particular there are finitely many prime $B_{\s}$-stable divisors in $X$. The boundary $\partial X = X_1 \cup \cdots \cup X_r$ with $r$ the rank of $X$ is the union of the prime $G$-stable divisors. The prime $B_{\s}$-stable divisors which are not $G$-stable are called colors. We denote by $\cD_X$ the set of colors and by $\cV_X = \{ X_1, \ldots, X_r \}$ 
the set of prime $G$-stable divisors.

We start with a description of prime $G$-stable divisors. 
Let $j: Y \to X$ be the inclusion of the closed $G$-orbit in $X$.  
Let $B^-_{\s}$ be the Borel subgroup containing $T_{\s}$ and
opposite to $B_{\s}$, and let $z \in Y$ be the unique 
$B^-_{\s}$-fixed point in $Y$. Then the stabilizer of $z$ in $G$
is $\sigma(P)$, and this identifies $Y$ with $G/\sigma(P)$.  
For any character $\lambda$ of $\sigma(P)$, 
we have a homogeneous line bundle
$\cL_Y(\lambda) = G \times^{\sigma(P)} \bC_{\lambda}$ 
on $G/\sigma(P)$, where $\bC_{\lambda}$ is the $1$-dimensional 
$\sigma(P)$-representation of weight $\lambda$. We may now
state the following result (see 
\cite[Proposition 8.1 and Corollary 8.2]{decp}): 

\begin{proposition}
  \begin{enumerate}
    \item The map $j^* : \Pic(X) \to \Pic(Y)$ is injective.
  \item For any $i \in [1,r]$, the torus $T_{\s}$ acts on $T_zX/T_zX_i$ with weight some $\alphab_i \in \Db$ and the resulting map 
  $\cV_X \to \Db,  X_i \mapsto \alphab_i$ is bijective.
  \item We have $\cO_X(X_i)\vert_Y = \mL_Y(\alphab_i)$ for any such $i$.
  \end{enumerate}
\end{proposition}

\begin{remark}
  \label{def-bord-root}
We set $X_{\alphab_i} := X_i$ for $i \in [1,r]$ so that $X_\betab$ is well defined for $\betab \in \Db$.
\end{remark}

Next we want to relate colors and restricted roots. This is more difficult, since there may be more colors than restricted roots as we will see next.  Recall the description of the
Picard group of $X$ (see \cite[Theorem 3.2.4]{survey}):

\begin{proposition}
  \label{pic-colors}
  We have $\Pic(X) = \bigoplus_{D \in \cD_X}\bZ[D]$.
\end{proposition}

%The rank of the Picard group is determined as follows (see \cite[Theorem 7.6]{decp}).
%
%\begin{proposition}
%  \label{pic-rk}
%  We have $\Pic(X) = \bZ^{r + s}$ where $r$ is the rank of $X$ and $s$ is the number of restricted simple roots $\gammab \in \Db$ such that there exists a pair of exceptional simple roots $\alpha,\beta = \sigmab(\alpha)$ with $\alphab = \gammab = \betab$. 
%\end{proposition}
%
%
%In particular,  we have the following formula:
%$$\rk(\Pic(X)) = \left\{ \begin{tabular}{ll}
%$r$ & if $X$ is not exceptional,\\
%$r+1$ & if $X$ is exceptional.\\
%\end{tabular}\right. $$

There is a correspondence between colors and restricted roots that we describe now.
For $\alpha \in \Delta$, recall that $G_\alpha$ denotes
the subgroup of $G$ generated by $U_{\alpha}$
and $U_{- \alpha}$, and 
set $\cD_X(\alpha) = \{ D \in \cD_X \ | \ G_\alpha \cdot D \neq D \}$. Note that if $\sigma(\alpha) = \alpha$, then $\cD_X(\alpha) = \emptyset$ (see \cite[Section 1.4]{lunaA}).

\begin{proposition}
  \label{prop-lunaA}
  For $\alpha \in \Delta_1$, the set $\cD_X(\alpha)$ consists of a unique element $D_\alpha$.
Moreover, for any distinct $\alpha,\beta \in \Delta_1$, 
we have $D_\alpha = D_\beta$ if and only if $\scal{\alpha^\vee,\beta} = 0$ and $\sigma(\alpha) = -\beta$.
\end{proposition}

\begin{proof}
The assertion that  $\cD_X(\alpha)$ has a unique element
follows from a result of Luna which holds true for any wonderful variety (see \cite[Section 1.4]{lunaA}).
Luna proves that three cases, called $(a)$, $(a')$ and $(b)$, occur. In cases $(a')$ and $(b)$, the set $\cD_X(\alpha)$ consists of a unique element, while in case $(a)$ the set $\cD_X(\alpha)$ consists of two elements. We prove that case $(a)$ does not occur:  in this case, by \cite[Section 1.4.(2)]{lunaA}, we have $\alpha = \gammab = \gamma - \sigma(\gamma)$ for some $\gamma \in \Delta$. Thus, $\sigma(\alpha) = -\alpha$ and $\alphab = 2\alpha = 2\gammab \in \Db$. In particular, $\gammab,2\gammab \in \Db$, which contradicts the fact that $\Db$ is a basis of $\Rb$.

For $\alpha,\beta \in \Delta_1$, there are, according to \cite[Proposition 3.2]{lunaA}, the following possibilities to have $\cD_X(\alpha) \cap \cD_X(\beta) \neq \emptyset$:
\begin{itemize}
\item Both $\alpha$ and $\beta$ are in $\Db$, in which case it may happen that $|\cD_X(\alpha) \cup \cD_X(\beta)| = 3$.
\item $\scal{\alpha^\vee,\beta} = 0$ and $\alpha + \beta \in \Db$ or $\frac{1}{2}(\alpha + \beta) \in \Db$.
\end{itemize}
The first case does not occur by the above argument. 
If $\scal{\alpha^\vee,\beta} = 0$ and 
$\frac{1}{2}(\alpha + \beta) \in \Db$ or 
$\alpha + \beta \in \Db$, then there exists $\gamma \in \Delta_1$ such that $\gammab = \gamma - \sigma(\gamma) = \frac{1}{2}(\alpha + \beta)$ or $\gammab = \gamma - \sigma(\gamma) = \alpha + \beta$. Write
  $$\sigma(\gamma) + \sigmab(\gamma) = - \sum_{\delta \in \Delta_0} c_\delta\delta.$$
  Then, we have
  $$\frac{1}{2}(\alpha + \beta) = \gamma + \sigmab(\gamma) + \sum_{\delta \in \Delta_0} c_\delta\delta \ \ \ \textrm{ or }\ \ \  \alpha + \beta = \gamma + \sigmab(\gamma) + \sum_{\delta \in \Delta_0} c_\delta\delta.$$
In the former case, 
this implies $\alpha = \beta$ and 
$\gamma + \sigmab(\gamma) \leq \alpha$, 
which is impossible. 
In the latter case, 
we get that $\gamma$ equals $\alpha$ or $\beta$. Assume for example that $\gamma = \alpha$, then we have $\sigmab(\alpha) = \sigmab(\gamma) = \beta$ and $c_\delta = 0$ for all $\delta \in \Delta_0$. We thus have $\sigma(\alpha) = \sigma(\gamma) = -\sigmab(\gamma) = -\beta$. Note that $\alphab = \betab$.

Conversely, if $\beta = -\sigma(\alpha)$ and $\scal{\alpha^\vee,\beta} = 0$, then $\alphab = \alpha - \sigma(\alpha) = \alpha + \beta$ and by \cite[Proposition 3.2]{lunaA} again, we have $D_\alpha = D_\beta$.
\end{proof}

\begin{remark}
If $X$ is the wonderful compactification of a nonadjoint symmetric space, then there may be some simple roots $\alpha \in \Delta_1$ with $|\cD_X(\alpha)| =2$. A typical example is the case $G = \SL_2$ and $H = T$ a maximal torus. Then $N = \N_G(T)$ is the normalizer of the torus. There is a unique projective compactification of $G/H$ given by $\bP^1 \times \bP^1$ with two $B$-stable divisors $D^+$ and $D^-$ and both are such that $\SL_2 \cdot D^\pm = X$. We thus have $\cD_X(\alpha) = \{ D^+,D^-\}$, where $\alpha$ is the unique simple root of $G$. The wonderful compactification $X_\ad$ of $G/N$ is the quotient of $X$ by the involution exchanging the two factors and is isomorphic to $\bP^2$ with a unique $B$-stable divisor $D$, so that $\cD_{X_\ad}(\alpha) = \{D\}$.

  The restricted root system is $\Rb = \{ -\alphab , \alphab \}$ in both cases. But for a maximal split torus $T_{\s}$, we have $\alpha \in \fX(T_{\s}/H \cap T_{\s})$ while $\alpha \not\in \fX(T_{\s}/N \cap T_{\s})$ and $\alphab = 2\alpha \in \fX(T_{\s}/N \cap T_{\s})$.
\end{remark}

In view of Proposition \ref{prop-lunaA}, we may define a map $\zeta : \Delta_1 \to \cD_X$ by $\alpha \mapsto D_\alpha$.%with 
%$\cD_X(\alpha)= \{ D_\alpha \}$. 
%$\zeta(\alpha) = D$ with $D \in \cD_X(\alpha)$. 

\begin{lemma}
The map $\zeta : \Delta_1 \to \cD_X$ is surjective. 
\end{lemma}

\begin{proof}
For $D \in \cD_X$, there exists $\alpha \in \Delta_1$ such that $D \in \cD_X(\alpha)$. Indeed, since $G$ is generated by the $G_\alpha$ for $\alpha \in \Delta$ and since $D$ is not $G$-stable, there exists at least one $\alpha \in \Delta$ with $G_\alpha \cdot D \neq D$. Furthermore, by \cite[Section 1.4.(1)]{lunaA}, we have $\alpha \not\in \Delta_0$; thus, $\alpha \in \Delta_1$.
\end{proof}

\begin{lemma}
\label{lemma:alphab=betab}
Let $\alpha, \beta \in \Delta_1$ with $\alpha \neq \beta$. Then $\alphab = \betab$ if and only if $(\alpha,\beta)$ is a pair of exceptional roots or ($\beta = - \sigma(\alpha)$ and $\scal{\alpha^\vee,\beta} = 0$).
\end{lemma}

\begin{proof}
Write $\alphab = \alpha - \sigma(\alpha) = \alpha + \sigmab(\alpha) + \sum_{\gamma \in \Delta_0} c_\gamma \gamma$ with $c_\gamma \geq 0$. Since $\alpha$, $\beta$, $\sigmab(\alpha)$ and $\sigmab(\beta)$ lie in $\Delta_1$, the equality $\alphab = \betab$ together with the condition $\alpha \neq \beta$ implies that $\beta = \sigmab(\alpha)$. If $\scal{\alpha^\vee,\sigma(\alpha)} \neq 0$, then $(\alpha,\beta)$ is a pair of exceptional roots. Otherwise, we have 
$0 = \scal{\alpha^\vee,\sigma(\alpha)} = - \scal{\alpha^\vee,\beta} 
- \sum_{\gamma \in \Delta_0} c_\gamma \scal{\alpha^\vee,\gamma}$. 
This implies the vanishings $\scal{\alpha^\vee,\beta} = 0$ and $c_\gamma \scal{\alpha^\vee,\gamma} = 0$ for all $\gamma$. Since $\beta = \sigmab(\alpha)$, we have $\sigma(\beta) = -\alpha -  \sum_{\gamma \in \Delta_0} c_\gamma \gamma$. This implies the equality $\sigma(\beta) = -\alpha$ since otherwise, the support of the root $\sigma(\beta)$ being connected, there would exist a $\gamma$ with $c_\gamma \scal{\alpha^\vee,\gamma} \neq 0$. 
\end{proof}

\begin{proposition}
\label{prop:fibres-tau}
The map $\tau : \cD_X \to \Db, D \mapsto \alphab$, with $\alpha \in \Delta_1$ such that $D \in \cD_X(\alpha)$, is well defined and surjective. Moreover, this map is injective except if $X$ is exceptional, in which case the only non-trivial fiber is $\tau^{-1}(\alphab) = \{D_\alpha,D_\beta\}$, where $\alpha,\beta$ are the exceptional roots.
\end{proposition}

\begin{proof}
The restricted root $\alphab$ does not depend on the choice of $\alpha$ with $D \in \cD_X(\alpha)$: if $D \in \cD_X(\alpha) \cap \cD_X(\beta)$, then $\beta = - \sigma(\alpha)$ and $\betab = \alphab$ by Proposition \ref{prop-lunaA}. For the surjectivity, note that the composition $\tau \circ \zeta$ is surjective, since $\tau \circ \zeta(\alpha) = \alphab$.

  Assume first that $\alpha$ is exceptional and set $\beta = \sigmab(\alpha)$. Then $\tau(D_\alpha) = \alphab = \betab = \tau(D_\beta)$. If $D_\alpha = D_\beta$, then $\cD_X(\alpha) \cap \cD_X(\beta) \neq \emptyset$. We thus have $\beta = - \sigma(\alpha)$ and $\scal{\alpha^\vee,\sigma(\alpha)} = - \scal{\alpha^\vee,\beta} = 0$, a contradiction with the fact that $\alpha$ is exceptional. Therefore, $D_\alpha \neq D_\beta$ and $\tau$ is not injective. Conversely, if $\alphab =\betab$, then by Lemma \ref{lemma:alphab=betab}, we have that $(\alpha,\beta)$ is a pair of exceptional roots or ($\beta = - \sigma(\alpha)$ and $\scal{\alpha^\vee,\beta} = 0$). In the second case, we have $D_\alpha = D_\beta$ so that $\tau$ is not injective exactly when $X$ is exceptional and the only non-trivial fiber is $\tau^{-1}(\alphab) = \{D_\alpha,D_\beta\}$, where $\alpha,\beta$ are the exceptional roots.
\end{proof}

%\begin{corollary}
%   \label{coro-fibre-zeta}
%Let $\alpha, \beta \in \Delta_1$ with $\alpha \neq \beta$. Then $\zeta(\alpha) = \zeta(\beta)$ if and only if
% $\beta = - \sigma(\alpha)$ and $\scal{\alpha^\vee,\beta} = 0$.
%\end{corollary}
%
%\begin{proof}
%Assume that $\zeta(\alpha) = \zeta(\beta)$, then $\cD_X(\alpha) \cap \cD_X(\beta) \neq \emptyset$ and the result follows from Proposition \ref{prop-lunaA}. Conversely, if $\beta = - \sigma(\alpha)$ and $\scal{\alpha^\vee,\beta} = 0$, then $\alphab = \betab$. If $\zeta(\alpha) \neq \zeta(\beta)$, then $(\alpha,\beta)$ is a pair of exceptional roots. In particular, we  would have $0 \neq \scal{\alpha^\vee,\sigma(\alpha)} = - \scal{\alpha^\vee,\beta} = 0$, a contradiction. Thus $\zeta(\alpha) = \zeta(\beta)$.
%\end{proof}

We recover a classical result (see \cite[Theorem 7.6]{decp}).

\begin{proposition}
  \label{pic-rk}
  We have $\Pic(X) = \bZ^{r + s}$ where $r$ is the rank of $X$ and $s$ is the number of restricted simple roots $\gammab \in \Db$ such that there exists a pair of exceptional simple roots $\alpha,\beta = \sigmab(\alpha)$ with $\alphab = \gammab = \betab$. Moreover, $s = 1$ if $X$ is exceptional, and $s = 0$ otherwise.
\end{proposition}

\begin{proof}
By Proposition \ref{pic-colors}, we have $\Pic(X) = \bigoplus_{D \in \cD_X}\bZ[D]$. Proposition \ref{prop:fibres-tau} shows that $|\cD_X| = |\Db|$ except for $X$ exceptional in which case $\cD_X$ has one more element.
\end{proof}

We now compute the restrictions $j^*\cO_X(D)$ for $D \in \cD_X$.
The next proposition is a direct application of results 
in \cite{luna}. For $\alpha \in \Delta$, let $\varpi_\alpha$ be the fundamental weight associated to $\alpha$.

\begin{proposition}
  \label{prop-lambda-alpha}
  Let $\alpha \in \Delta_1$ and let $\lambda_\alpha \in \fX(T_{\s})$ be such that $j^*\cO_X(D_\alpha) = \mL_Y(\lambda_\alpha)$. Then we have
  $$\lambda_\alpha = \left\{\begin{array}{ll}
  2 \varpi_\alpha & \textrm{ if $\sigma(\alpha) = -\alpha$,} \\
    \varpi_\alpha + \varpi_{\sigmab(\alpha)} & \textrm{ if $\sigma(\alpha) = -\sigmab(\alpha)$ and $\scal{\alpha^\vee,\sigma(\alpha)}  =0$,} \\
     \varpi_\alpha & \textrm{ otherwise.} \\
  \end{array}\right.$$
\end{proposition}

\begin{proof}
This result is \cite[Lemma 30.24]{Timashev}, due to A.~Foschi. 
We provide details for the reader's convenience.

 For $\beta \in \Delta_1$, let $Y_\beta \subset Y$ be the Schubert curve dual to $\cO_Y(\varpi_\beta)$. Then by \cite[Lemma 3.1.1 and Lemma 3.1.2]{luna}, we have
  $$Y_\beta \cdot j^*D_\alpha = \left\{\begin{array}{ll}
  2 \delta_{\alpha,\beta} & \textrm{ if $\sigma(\alpha) = -\alpha$,} \\
  \delta_{D_\alpha,D_\beta} & \textrm{ otherwise.} \\
  \end{array}\right.$$
The result follows from this and the facts that if $|\zeta^{-1}(D_\alpha)| > 1$, then $\zeta^{-1}(D_\alpha) = \{\alpha,\sigmab(\alpha)\}$ and that this occurs if and only if $\sigma(\alpha) = -\sigmab(\alpha)$ and $\scal{\alpha^\vee,\sigma(\alpha)}  =0$ (Proposition \ref{prop-lunaA}).
\end{proof}

The above proof implicitly uses the fact 
that there exists a family of irreducible
$B$-stable curves $(C_D)_{D \in \cD_X}$ such that 
the classes $[C_D]$ in the Chow group $A_1(X)$ form the dual 
basis to the basis $([D])_{D \in \cD_X}$ of $\Pic(X)$ (see \cite[Lemma 3.1.2]{luna}). Recall the definition of $\alphah$ for $\alphab \in \Db$ and the notation $X_\betab$ for $\betab \in \Db$ from Remark \ref{def-bord-root}.
Also, denote 
$C_{D_\alpha}$ 
by $C_{\alpha}$ for simplicity.

\begin{corollary}
\label{coro:intersection}
  We have $X_\betab  \cdot C_{\alpha} = \scal{\alphah^\vee,\betab}$ for all $\alpha,\beta \in \Delta_1$.
\end{corollary}

\begin{proof}
Recall that $j^*\cO_X(X_\betab) = \mL_Y(\betab)$. 
Note that we have $\betab = \sum_{\gamma \in \Delta} \scal{\gamma^\vee,\betab} \varpi_\gamma = \sum_{\gamma \in \Delta_1} \scal{\gamma^\vee,\betab} \varpi_\gamma$ since $\scal{\gamma,\betab} = 0$ for $\gamma \in \Delta_0$. Define the following subsets of $\Delta_1$: 
$$ A = \{\gamma \in \Delta_1 \ | \ \sigma(\gamma) = -\gamma\},$$
$$ B = \{ \gamma \in \Delta_1\ | 
\ \sigma(\gamma) = -\sigmab(\gamma) \textrm{ and } 
\scal{\gamma^\vee,\sigma(\gamma)} = 0 \}, $$
$$\text{ and } C = \Delta_1 \setminus (A \cup B).$$ 
Since $\sigmab$ induces a fixed point free involution on $B$ 
and since 
$\scal{\gamma^\vee,\betab} = \scal{\sigmab(\gamma^\vee),\betab}$ 
for $\gamma \in B$, Proposition \ref{prop-lambda-alpha} implies
$$\betab = \frac{1}{2} \sum_{\gamma \in A} \scal{\gamma^\vee,\betab} \lambda_\gamma + \frac{1}{2} \sum_{\gamma \in B} \scal{\gamma^\vee,\betab} \lambda_\gamma + \sum_{\gamma \in C} \scal{\gamma^\vee,\betab} \lambda_\gamma.$$
By Proposition \ref{prop-lunaA}, $D_{\sigmab(\gamma)} = D_\gamma$ for $\beta \in B$, giving the following equality on the level of divisors:
$$X_\betab =   \frac{1}{2} \sum_{D = D_\gamma, \; \gamma \in A} 
\scal{\gamma^\vee,\betab} D_\gamma 
+ \sum_{D = D_\gamma, \; \gamma \in B} 
\scal{\gamma^\vee,\betab}  D_\gamma 
+ \sum_{D = D_\gamma, \; \gamma \in C} 
\scal{\gamma^\vee,\betab} D_\gamma.$$ 
We get
  $$X_\betab \cdot C_{\alpha}  
  = \left\{\begin{array}{ll}
  \frac{1}{2} \scal{\alpha^\vee,\betab} & \textrm{ if $\sigma(\alpha) = -\alpha$,} \\
     \scal{\alpha^\vee,\betab} & \textrm{ otherwise.} \\
  \end{array}\right.$$
%  Note that $\scal{\alpha^\vee,\betab} = \scal{-\sigma(\alpha)^\vee,\betab}$ so that if $\sigma(\alpha) = -\sigmab(\alpha)$ and $\scal{\alpha^\vee,\sigma(\alpha)}  = 0$, we have $\scal{\alpha^\vee,\betab} \varpi_\alpha + \scal{-\sigma(\alpha)^\vee,\betab} \varpi_{-\sigma(\alpha)} = \scal{\alpha^\vee,\betab} \lambda_\alpha$.
%
We now compare the above values to $\scal{\alphah^\vee,\betab}$. If $\sigma(\alpha) = - \alpha$, then $\alphah^\vee = \alphab^\vee = \frac{1}{2}\alpha^\vee$ proving the first case.
If $\sigma(\alpha) \neq - \alpha$, we have two possibilities: either $\scal{\alpha^\vee,\sigma(\alpha)} = 0$ or $\scal{\alpha^\vee,\sigma(\alpha)} = 1$. In the former case, we have $\alphah^\vee = \alphab^\vee = \frac{1}{2}(\alpha^\vee - \sigma(\alpha)^\vee)$; thus, $\scal{\alphah^\vee,\betab} = \frac{1}{2}(\scal{\alpha ^\vee,\betab} + \scal{-\sigma(\alpha)^\vee,\betab}) = \scal{\alpha^\vee,\betab}$. Finally, if 
$\scal{\alpha^\vee,\sigma(\alpha)} = 1$, 
then $\alphah^\vee = \frac{1}{2}\alphab^\vee = \frac{1}{2}(\alpha^\vee - \sigma(\alpha)^\vee)$ and the result follows as before.
\end{proof}

Recall from Proposition \ref{pic-colors}
that $\Pic(X) = \bigoplus_{D \in \cD_X} \bZ [D]$. 
Furthermore, by \cite[Theorem 3.2.9]{survey}
the monoid of numerically effective divisor classes is given by 
$\Nef(X) = \bigoplus_{D \in \cD_X}\bZ_{\geq 0} [D]$. 
It coincides with the monoid of globally generated divisor
classes. 
On the curve side,  we have 
$A_1(X) = \bigoplus_{D \in \cD_X} \bZ[C_D]$ and 
rational equivalence coincides with numerical equivalence;
moreover, 
the monoid of effective classes (generated  by the classes
of irreducible reduced curves) is given by 
$\NE(X) = \bigoplus_{D\in \cD_X} \bZ_{\geq 0} [C_D]$, 
see \cite{sanya} for more on curves on spherical varieties. Furthermore, we have a $\bZ$-linear map 
$$\psi : A_1(X) \to \bZ \Db^\vee$$ 
defined by $\psi(C_D) = \alphah^\vee$, for $D \in \cD_X(\alpha)$; where $\bZ \Db^\vee$ is the coroot lattice of the restricted root system $\Rb$.
By Corollary \ref{coro:intersection}, we have 
$X_\betab \cdot C = \scal{\psi(C),\betab}$ 
for all $[C] \in \NE(X)$.

\begin{proposition}
\label{prop-psi}
 The map $\psi: A_1(X) \to \bZ\Db^\vee$ is surjective.
  \begin{enumerate}
  \item The image of the monoid of effective curves is the monoid generated by the positive coroots.
  \item The image of the monoid of curves having non-negative intersection with any component of $\partial X_\ad$ is the  intersection of $\bZ\Db^\vee$ with the monoid of dominant cocharacters.
  \end{enumerate}
\end{proposition}

\begin{proof}
The surjectivity follows from the surjectivity of $\tau$. The monoid of effective curves is generated by the set $\{C_D \ | \ D \in \cD_X\}$ whose image by $\psi$ is $\Db^\vee$; this proves (1). Part (2) follows from Corollary \ref{coro:intersection}.
\end{proof}

Recall that a curve class $\gamma \in A_1(X)$ is covering 
if there exists an irreducible and reduced
curve $C$ of class $\gamma$ passing through a general point $x \in X$. Note that this implies that 
$X_\betab \cdot \gamma \geq 0$ for all $\betab \in \Db$. We call a class $\gamma \in \NE(X)$ 
\emph{virtually covering} if 
$X_\betab \cdot \gamma \geq 0$ for all $\betab \in \Db$.

\begin{corollary} 
\label{coro:vir-cov}
  \begin{enumerate}
    \item If $X$ is non-exceptional, then there is a unique virtually covering curve class 
    $\gamma_0 \in \NE(X)$ which is minimal in this monoid. Moreover, we have 
    $\psi(\gamma) = 
    \Thetab^\vee$.
    \item If $X$ is exceptional, then there are exactly two minimal virtually covering curve classes 
    $\gamma^+_0, \gamma^-_0 \in \NE(X)$ and we have 
    $\psi(\gamma^+_0) = 
     \Thetab^\vee = \psi(\gamma^-_0)$.
  \end{enumerate}
\end{corollary}

\begin{proof}
The image by $\psi$ of an effective and virtually covering curve class is in the intersection of the monoid generated by coroots in $\Rb$ and the dominant chamber. There is a unique minimal such element: the coroot of the highest root of $\Rb$. Since $\Thetab$ is the highest root of $\Rb$ by Proposition \ref{prop:rrs-summary}(1), 
the element $\Thetab^\vee$ is the smallest possible image by $\psi$ of an effective and virtually covering curve class (see Lemma \ref{lemm:minimal}).

  If $X$ is non-exceptional, then $\psi$ is injective and this proves (1). To prove (2), we are left to prove that if $X$ is exceptional, there are exactly two classes 
  $\gamma^+$ and $\gamma^-$ in $\NE(X)$ such that 
  $\psi(\gamma^+) = \Thetab^\vee = \psi(\gamma^-)$. But the kernel of $\psi$ is $\bZ([C_{D_\alpha}] - [C_{D_{\sigmab(\alpha)}}])$ with $\alpha$ an exceptional root. Since the coefficient of $\alphah^\vee$ in $\Thetab^\vee$ is $1$ by Lemma \ref{lem-coef-rac}, there are exactly two classes in $\NE(X)$ that are mapped to $\Thetab^\vee$ via $\psi$, namely, 
  $\gamma^+$ with coefficient $1$ in $C_\alpha$ and $0$ on $C_{\sigmab(\alpha)}$, and $\gamma^-$ with coefficient $0$ in $C_\alpha$ and $1$ on $C_{\sigmab(\alpha)}$.
\end{proof}

\subsection{Classes of the minimal rational curves}
\label{subsection:classes-vmrt}

In this subsection, we prove that the curve classes 
$\gamma_0$, $\gamma_0^+$ and $\gamma_0^-$ are covering and are therefore the classes of minimal rational curves on $X$.

We will need a few more results on $X$. Recall that $x \in X^0$ is our base point and that $r$ is the rank of $X$. The $G$-orbits in $X$ are indexed by the subsets $I \subset [1,r]$ via
$\cO_I = \{ x' \in X \ | \ x' \in X_i \Leftrightarrow  i \in I \}$.

The local structure theorem associated to the closed orbit $Y$ 
gives the following: there exists an affine $P$-stable open subset 
$X_{Y,B} \subset X$ containing $x$ with $X_{Y,B} \cap Y \neq \emptyset$ 
and a $P$-equivariant isomorphism $X_{Y,B} \simeq \R_u(P) \times \bA^r$, 
where $P = \R_u(P) \rtimes L$ acts on $\R_u(P) \times \bA^r$ via
$(u,l) \cdot (g,z_1,\ldots,z_r) = 
(ul g l^{-1}, \alphab_1(l)z_1,\ldots,\alphab_r(l)z_r)$.
Here each simple restricted root $\alphab_i \in \fX(\Sb)$ is viewed as
a character of $L$ via the isomorphism $\Sb \simeq L/L \cap H$.
In particular, the closure of $T_{\s} \cdot x = S \cdot x$ in $X_{Y,B}$ 
is $T_{\s}$-equivariantly isomorphic to $\bA^r$, where the torus $T_{\s}$ 
acts linearly with weights $\Db = (\alphab_i)_{i \in [1,r]}$. 
The prime $G$-divisor $X_{\alphab_i}$ is defined in $X_{Y,B}$ 
by the vanishing of the coordinate with weight $\alphab_i$ in $\bA^r$
(see \cite[Proposition 2.3]{decp} and 
\cite[Section 30.3]{Timashev} 
for the results of this paragraph).

Recall from Subsection \ref{subsec:div-rrs} that $j : Y \to X$ denotes the inclusion of the closed $G$-orbit and that the map $j^* : \Pic(X) \to \Pic(Y)$ is injective. Let 
$$\fX_X(T_{\s}) = \{ \lambda \in \fX(T_{\s}) \ | \ 
\mL_Y(\lambda) \in \,  j^*\Pic(X) \subset \Pic(Y)\}.$$ 
For $\lambda \in \fX_X(T_{\s})$, we write $\mL_X(\lambda)$ for the line bundle such that $j^*\mL_X(\lambda) = \mL_Y(\lambda)$ (see \cite[End of 8.1]{decp}).

Given a cocharacter $\eta : \bG_m \to \Sb = S/S^\sigma$, we say that $\eta$ is dominant if $\scal{\eta,\alphab} \geq 0$ for all $\alpha \in \Delta$. A cocharacter $\eta$ defines a map 
$\bC^\times \to X, t \mapsto \eta(t) \cdot x$.
This map extends to a morphism $\eta : \bP^1 \to X$.

The following lemma generalizes \cite[Lemma 3.1]{BF} to the case of wonderful compactifications of adjoint indecomposable symmetric spaces. 

\begin{lemma}\label{lem:mult}
Let $\eta : \bG_m \to \Sb$ be a dominant cocharacter, $\eta : \bP^1 \to X$ the corresponding morphism, and $C_{\eta}$ its image.
\begin{enumerate}
\item We have $\eta(0) \in \cO_I$, where 
$I := \{ i \in [1,r] \ | \ \scal{\eta,\alphab_i} \neq 0 \}$.
\item We have $\eta(\infty) \in \cO_J$, where $J := \{ j \in [1,r] \ | \ \scal{\eta,w_0(\alphab_j)} \neq 0 \}$.
\item The morphism $\eta : \bP^1 \to C_{\eta}$ is an isomorphism if and only if there exists $i \in [1,r]$ such that $\scal{\eta,\alphab_i} = 1$. 
\item For $\lambda \in \fX_X(T_{\s})$, we have $\deg (\eta)^* \mL_X(\lambda) = \scal{\eta, \lambda - w_0\lambda}$.
\end{enumerate}
\end{lemma}

\begin{proof}
(1) Since $\eta$ is dominant, it extends to a morphism 
$\bA^1 \to X_{Y,B} \cap \overline{T_s \cdot x}$ 
defined by $t \mapsto (t^{\scal{\eta,\alphab_i}})_{i \in [1,r]}$, where $X_{Y,B} \cap \overline{T_s \cdot x}$ 
is identified with $\bA^r$ as above. In particular, $\eta (0) \in X_{Y,B}$ and vanishes on the coodinates with 
  indices in $\{ i \in [1,r] \ | \ \scal{\eta,\alphab_i} \neq 0 \}$. Moreover, the morphism $\eta : \bP^1 \to C_{\eta}$ is a local isomorphism at $\eta (0)$ if and only if there exists $i$ such that $\scal{\eta,\alphab_i} = 1$.

(2) Consider the open affine subset
$w_0 \cdot X_{Y,B}$ of $X$. It is isomorphic to 
$\R_u(P)^{w_0} \times \bA^r$ with a linear action of $T_{\s}$ on $\bA^r$ with weights $w_0(\Db)$. All these weights are non-negative linear combinations of negative roots. In particular, the one-parameter subgroup $-\eta$ acts with non-negative weights on $\bA^r$, and hence extends to a morphism $\bA^1 \to w_0 \cdot X_{Y,B}$, $t \mapsto (t^{\scal{\eta, -w_0(\alphab_i)}})_{i \in[1,r]}$. It follows that $\eta (\infty) = (-\eta )(0) \in w_0 \cdot X_{Y,B}$ and as above, $\eta(\infty) \in \cO_J$. Moreover, the morphism $\eta : \bP^1 \to C_{\eta}$ is a local isomorphism at $\eta (\infty)$ if and only if there exists $i$ such that $\scal{\eta,-w_0(\alphab_i)} = 1$. Note that $-w_0(\alphab_i) = -w_0(\alpha_i) + w_0(\sigma(\alpha_i)) = - \overline{w_0(\alpha_i)}$ and since $-w_0$ permutes $\Delta_1$, by Proposition \ref{prop:rrs-summary}(1), 
the previous condition is true if and only if there exists $i$ such that $\scal{\eta,\alphab_i} = 1$.

(3)  This follows from the above conditions at $\eta(0)$ and $\eta(\infty)$.

(4) The pull-back of $\mL_X(\lambda)$ to 
$X_{Y,B}$ (resp.~$w_0 \cdot X_{Y,B}$) 
has a trivializing section of weight $\lambda$
(resp.~$w_0(\lambda)$). As a consequence, 
the line bundle $(\eta)^* \mL_X(\lambda)$ 
is a $\bG_m$-linearized line bundle on $\bP^1$ with weights $\scal{\eta,\lambda}$ at $0$ and $\scal{\eta, w_0 \lambda}$ at $\infty$. Since the degree of such a line bundle is the difference of its weights, this yields our assertion. 
\end{proof}  

We now apply the above result to $\eta = \Thetab^\vee$.

\begin{corollary}
\label{cor:class-of-1-pm-theta}
  Consider the morphism $\Thetab^\vee : \bP^1 \to X$.
  \begin{enumerate}
  \item If $\Rb$ is not of type ${\rm A}_1$, then $\Thetab^\vee$ is an isomorphism onto its image.
    \item If $\Rb$ is of type ${\rm A}_1$, then $\Thetab^\vee$ has degree $2$ over its image.
    \item The push-forward class is given as follows:
      $$\Thetab^\vee_*[\bP^1] = \left\{\begin{array}{ll}
      2 \gamma_0 & \textrm{ if $X$ is non-exceptional,} \\
      \ \!\! \gamma_0^+ + \gamma_0^- & \textrm{ if $X$ is exceptional.} \\
    \end{array}\right.$$
  \end{enumerate}
\end{corollary}

\begin{proof}
(1) If $\Rb$ is not of type $\textrm{A}_1$, then Proposition \ref{prop:rrs-summary}(4) 
implies that there exists a simple root $\alphab \in \Db$ such that $\scal{\Thetab^\vee,\alphab} = 1$. Therefore, $\Thetab^\vee : \bP^1 \to X$ is an isomorphism onto its image.

(2) If $\Rb$ is of type $\textrm{A}_1$, then $\Thetab^\vee$ induces a map 
$\bG_m \to X_{Y,B} \cap \overline{T_s \cdot x} = \bA^1,
 t \mapsto t^2$ 
which is of degree $2$ onto its image. 
  
(3) For $\lambda \in \fX_X(T_{\s})$, we have 
$[\mL_X(\lambda)] \cdot \Thetab^\vee_*[\bP^1]  = \deg(\Thetab^\vee)^*\mL_X(\lambda) = \scal{\Thetab^\vee,\lambda - w_0(\lambda)} = \scal{\Thetab^\vee,\lambda} - \scal{w_0(\Thetab)^\vee,\lambda} = 2\scal{\Thetab^\vee,\lambda}$ by Proposition \ref{prop:rrs-summary}(1).
This proves the result for $X$ non-exceptional, since 
$X_\betab \cdot \gamma_0 = \scal{\Thetab^\vee,\betab}$, 
$\cO_X(X_\betab) = \mL_X(\betab)$, and $(\mL_X(\betab))_{\betab \in \Db}$ 
generates $\Pic(X) \otimes_\bZ \bQ$
%\fX_X(T_{\s}) \otimes_\bZ \bQ$ 
by Propositions \ref{pic-rk} and \ref{prop-lambda-alpha}.

If $X$ is exceptional, then by the same argument, 
we have that the class of the image and the class 
$\gamma_0^+ + \gamma_0^-$ agree on all boundary divisors. 
Since  $\Pic(X) \otimes_\bZ \bQ$
is spanned by $(X_{\betab})_{\betab \in \Db}$ 
and the class of $D_{\alpha}$ for an exceptional
simple root $\alpha$, 
we therefore only need to check that the class of 
the image and the class 
$\gamma_0^+ + \gamma_0^-$ agree on $D_\alpha$. 
Assume that 
  $\gamma_0^+$ is dual to $D_\alpha$ while 
  $\gamma_0^-$ is dual to $D_{\sigmab(\alpha)}$. Since $\Thetab = 2\Theta$ by Corollary \ref{coro-coef1-ap}, we get $2\scal{\Thetab^\vee,\lambda_\alpha} = \scal{\Theta^\vee,\lambda_\alpha} = 1 
  = D_\alpha \cdot (\gamma_0^+ + \gamma_0^-)$ by Proposition \ref{prop:rrs-summary}(3). 
Similarly, we have $2\scal{\Thetab^\vee,\lambda_{\sigmab(\alpha)}} = \scal{\Theta^\vee,\lambda_{\sigmab(\alpha)}} = 1 
= D_{\sigmab(\alpha)} \cdot (\gamma_0^+ + \gamma_0^-)$. 
\end{proof}

Recall the definitions of the nilpotent orbits $\cO_{\min}$ 
and of type $\cO_\summ$ from Definition \ref{def:nilp}. 
For $G$ simple with maximal torus $T_s$ of split type 
such that $\sigma(\Theta) \neq - \Theta$, 
define the nilpotent orbit $\cO_{\summ,\sigma}$ by 
$\cO_{\summ,\sigma} 
= G \cdot  (e_{\Theta} - \sigma(e_{\Theta}))$ 
with $e_{\Theta} \in \fg_{\Theta} \setminus \{0\}$. 
Note that by Proposition \ref{prop:rrs-summary}(2), 
the nilpotent orbit $\cO_{\summ,\sigma}$ is indeed of type 
$\cO_\summ$.

\begin{proposition}
  \label{prop-hwc}
  \begin{enumerate}
\item In the non-exceptional case, there exists a smooth rational curve $C$ in $X$ such that 
  $x \in C$ 
  and $[C] = \gamma_0$. 
  \item In the exceptional case, for any $\gamma \in \{\gamma_0^+,\gamma_0^-\}$ there exists smooth rational curve $C$ in $X$ such that  $x \in C$ and $[C] = \gamma$.
    \end{enumerate}
   Furthermore, we may choose $C$ such that
  $T_xC \setminus \{ 0 \}\subset \cO_{\min}$ 
  if $\sigma(\Theta) = - \Theta$, and  
  $T_x C \setminus \{ 0 \} \subset 
  \cO_{\summ,\sigma}$ otherwise.
\end{proposition}

\begin{proof}
If $\sigma(\Theta) = - \Theta$, pick 
$e = e_{\Theta} \in \fg_{\Theta} \setminus \{0\}$. If $\sigma(\Theta) \neq - \Theta$, pick 
$e = e_{\Theta} + e_{-\sigma(\Theta)}$ 
with $e_{\Theta}$ as above and $e_{-\sigma(\Theta)} \in \fg_{-\sigma(\Theta)} \setminus \{0\}$. 
Note that $e \in \cO_{\min}$ for 
$\sigma(\Theta) = - \Theta$ and that, by Proposition \ref{prop:rrs-summary}(2), 
$e$ is in a nilpotent orbit of type $\cO_\summ$ for 
$\sigma(\Theta) \neq - \Theta$. Set $f = \sigma(e)$. We may choose $e$ so that $h = [e,f] = 2\Thetab^\vee$. Then $(e,h,f)$ is an $\fsl_2$-triple. The cocharacter $h$ induces a morphism $h : \bP^1 \to X$ which factors through $\Thetab^\vee : \bP^1 \to X$
  $$\xymatrix{\bP^1 \ar[rd]^-{h} \ar[d]_-{2:1} & \\
    \bP^1 \ar[r]^-{\Thetab^\vee} & X\\}$$
  so that the vertical map is a double cover. Note in particular that both maps $h$ and $\Thetab^\vee$ have the same image $C'$ in $X$.

Denote by $G(h)$ the closed subgroup of $G$ with Lie algebra $\scal{e,h,f}$. Then $G(h)$ is isomorphic to $\SL_2$ or $\PGL_2$, and $\sigma$ acts non-trivially on $G(h)$. In particular, we have an isogeny $\SL_2 \to G(h)$ and $\sigma$ lifts to a unique involution on $\SL_2$. Let $T'$ be a maximal torus of $\SL_2$ fixed pointwise by $\sigma$ and let $X'$ be the closure of $G(h) \cdot x$ in $X$ and $\tilde X'$ its normalisation. The theory of spherical embeddings
implies that $\tilde X'$ is either isomorphic to 
$\bP^1 \times \bP^1$ (the unique projective embedding 
of $\SL_2/T'$) or to $\bP^2$ (the unique projective embedding of $\SL_2/\N_{\SL_2}(T')$). 
Note that the normalization map is bijective in both
cases.  Also, 
note that $x \in X'$ and is a smooth point; thus, $X' \cap X^0$ is non-empty. Let $\tilde x$ be the preimage of $x$ in $\tilde X'$ and $\tilde C'$ be the preimage of $C'$ in $\tilde X'$. In the first case, the curve $\tilde C'$ is linearly equivalent to the diagonal curve $D$. In the second case, the curve $\tilde C'$ is a line in $\bP^2$. In both cases, $\tilde X'$ contains a line $\tilde L$ through $\tilde x$ (i.e.,
  either a line in one of the two rulings of $\bP^1 \times \bP^1$
  or a line in $\bP^2$) such that the image of $T_{\tilde x} \tilde L$ is equal to  $\scal{h + e - f}$ (and also to $\scal{h - e + f}$ in the $\bP^1 \times \bP^1$ case), 
  as a subset of $T_{\tilde x}\tilde X' = T_xX' \subset T_x X$ identified to $\fg^{-\sigma}$. Note also that  ${h + e - f}$, ${h - e + f}$ and  ${e}$ are in the same $G(h)$-orbit.

  If $\Rb$ is of type $\textrm{A}_1$, then $X$ is 
  non-exceptional and $\Thetab^\vee$ has degree $2$
  onto its image $C'$. In particular, 
  $[C'] = \frac{1}{2}\Thetab^\vee_*[\bP^1] = \gamma_0$ and by minimality, $\gamma_0$  has to be the push-forward of the class of a line in $\tilde X'$. Note that this implies that we are in the case $\tilde X' = \bP^2$.
  
Assume that $\Rb$ is not of type $\textrm{A}_1$. Then $\Thetab^\vee : \bP^1 \to X$ is an isomorphism onto its image $C'$. Note also that $h \in \fX^\vee$ is indivisible as a cocharacter of $T_{\s}$ by Proposition \ref{prop:rrs-summary}(4);
thus, $h(-1)$ is non-central in $G$. Since $h(-1)$ is central in $G(h)$, this group is isomorphic to $\SL_2$ and thus we have $G(h)^\sigma = T'$ and $\tilde X' = \bP^1 \times \bP^1$. Therefore, $[\tilde C'] = [D]$. Let $\tilde C^+$ and $\tilde C^-$ be the two rulings passing through $\tilde x$ in $\tilde X' = \bP^1 \times \bP^1$ and let $C^+$ and $C^-$ be their images in $X' \subset X$. %Define $[C^+]$ and $[C^-]$ to be the classes of the two rulings in $\bP^1 \times \bP^1 \subset X$. 
Then $[C^+] + [C^-] = [C'] = \Thetab^\vee_*[\bP^1]$. Since $C^+$ and $C^-$ pass through $x$, %any ruling meets $X' \cap X^0 = X' \setminus D$, 
the classes $[C^+]$ and $[C^-]$ are virtually covering. By Corollary \ref{cor:class-of-1-pm-theta}, this implies that $[C^+] = [C^-] = \gamma_0$ if $X$ is non-exceptional and (up to exchanging the two rulings) 
$[C^+] = \gamma_0^+$ and $[C^-] = \gamma_0^-$ if $X$ is exceptional.

Note that by the above discussion, the curve $C'$, if $\bar R$ is of type $\A_1$, and the curves $C^+$ and $C^-$, otherwise, are images of lines in $\tilde X'$ passing though $\tilde x$. Furthermore, their tangent space at $x$ (without the origin)
lies in the nilpotent orbit of $e$. Moreover, these curves are of class $\gamma_0$, $\gamma_0^+$ or $\gamma_0^-$ and therefore minimal. Let $\mathcal{K}$ be the corresponding minimal family, letting $B_H$ a Borel subgroup of $H$ act on the points representing these curves, we get a family of minimal curves whose limit is a highest weight curve in  $\mathcal{K}_x$. Since highest weight curves are smooth so are the curves $C'$, $C^+$ and $C^-$, finishing the proof.
\end{proof}

We  now compute the classes of minimal rational curves appearing in Proposition \ref{prop:min}.

\begin{corollary}\label{cor:1or2}
\begin{enumerate}
\item If $X$ is non-exceptional, there exists a unique family $\cK$ of minimal rational curves and for $C \in \cK$, we have $[C] = \gamma_0$.
\item  If $X$ is exceptional, there exists two families $\cK^+$ and $\cK^-$ of minimal rational curves and for $C \in \cK^\pm$, we have 
$[C] = \gamma_0^\pm$.
\end{enumerate}
\end{corollary}

\begin{proof}
Proposition \ref{prop:min} implies that there is a unique family of minimal rational curves in the non-exceptional case and there are two families in the exceptional case. 
By Proposition \ref{prop-hwc}, there exist curves 
of class $\gamma_0$ (resp. $\gamma_0^+$ and $\gamma_0^-$ in the exceptional case) passing through $x$ and therefore belong to a
covering family. Corollary \ref{coro:vir-cov} implies that this family has to be minimal
in the sense of monoids, that is, indecomposable. But then this family is 
minimal in the sense of families of rational curves: otherwise, the subfamily of curves 
through a general point contains reducible curves. 
%This follows from Proposition \ref{prop:min}, Corollary \ref{coro:vir-cov} and Proposition \ref{prop-hwc}.
\end{proof}

\begin{corollary}
  \label{cor-hwc}
Let $\cK$ be a family of minimal rational curves, and $C \in \cK_x$. 
If $\sigma(\Theta) = - \Theta$, then we have
$T_x C \setminus \{ 0 \}\subset \cO_{\min}$. 
Otherwise, $T_x C \setminus \{ 0 \}\subset \cO_{\summ,\sigma}$.
\end{corollary}

We now compute the dimensions of the families
of minimal rational curves $\cK$ and $\cK^\pm$ 
and of the nilpotent orbits $\cO_{\min}$ and 
$\cO_{\summ,\sigma}$. Let $\rho$ be the half-sum of positive roots in $G$ and $\rho_L$ be the half-sum of positive roots in $L$. Set $\kappa = 2\rho - 2\rho_L$. We have $\kappa = \sum_{\alpha \in R^+, \sigma(\alpha) < 0} \alpha$. Let $\Sigma = \sum_{\alphab \in \Db}\alphab$ be the sum of all restricted simple roots.

\begin{theorem}
  \label{theo-dim}
  Let $\cK$ be a family of minimal rational curves and
  let $C \in \cK_x$ and $m \in T_xC \setminus \{ 0 \}$.
  We have 
  $$\dim \cK_x = \scal{\Thetab^\vee,\kappa + \Sigma} - 2 \textrm{ and } \dim G \cdot m = 2 \scal{\Thetab^\vee,\kappa}.$$
  In particular, $\dim \cK_x = \frac{1}{2}\dim G \cdot m - 1 + ( \scal{\Thetab^\vee, \Sigma} - 1)$.
\end{theorem}

\begin{remark}
  \label{rem-2-ou-1}
  Note that the value of $\scal{\Thetab^\vee,\Sigma}$ depends on the type of $\Rb$ as follows:
  $$\partial X \cdot C = \scal{\Thetab^\vee,\Sigma} = \left\{\begin{array}{cl}
  2& \textrm{if $\Rb$ is of type $\A_r$ with $r \geq 1$,} \\
  1 & \textrm{otherwise.} \\
  \end{array}\right.$$
\end{remark}

\begin{proof}
Recall from Proposition  \ref{prop:red} that 
$\dim \cK_x = - K_X \cdot C - 2$. 
By the adjunction formula, we have 
$(-K_X)\vert_Y = -K_Y + \partial X\vert_Y$. Since $\cO_Y(-K_Y) = \mL_Y(\kappa)$ and $j^*\cO_X(\partial X) = \mL_Y(\Sigma)$, we get $j^*\cO_X(-K_X) = \mL_Y(\kappa + \Sigma)$ and $\dim \cK_x = \scal{\Thetab^\vee,\kappa + \Sigma} - 2$. 

We compute the dimension of $G \cdot m$ (we thank an anonymous referee for the uniform proof presented here). Since $e := m$ is a nilpotent element, there exist $f,h \in \fg$ such that $(e,h,f)$ is an $\fsl_2$-triple. Then $h$ induces a grading $\fg = \oplus_{k = -a}^a\fg(k)$ with $e \in \fg(2)$ and $f \in \fg(-2)$. The stabiliser of $e$ is the space of highest weight vectors of $\fsl_2 = \scal{e,h,f}$ acting on $\fg$. Its dimension is equal to the number of simple $\fsl_2$-modules thus equals $\dim\fg(0) + \dim\fg(1)$, since every simple $\fsl_2$-module intersects either $\fg(0)$ or $\fg(1)$ in dimension $1$ (depending on the parity of the highest weight). We thus have $\dim G \cdot m = \dim \fg - \dim \fg(0) - \dim \fg(1) = \dim\fg(1) + 2\sum_{k \geq 2}\dim\fg(k)$, since $\dim\fg(-k) = \dim\fg(k)$. Since $e$ lies in $\cO_{\min}$ 
or $\cO_{\summ,\sigma}$ by Proposition \ref{prop-hwc}, we have $a = 2$, thus $\dim G \cdot m = \dim \fg(1) + 2 \dim \fg(2)$.

In case $e \in \cO_{\min}$ \emph{i.e.} $\sigma(\Theta) = - \Theta$, we may choose $h = \Theta^\vee$ while if $e \in \cO_{\summ,\sigma}$ \emph{i.e.} $\sigma(\Theta) \neq - \Theta$, we may chose $h = \Theta^\vee - \sigma(\Theta^\vee)$. In any case, we get $h = 2\Thetab^\vee$ and in particular $\sigma(h) = -h$. If $\sigma(\alpha) = \alpha$, then $\scal{h,\alpha} = \scal{h,\sigma(\alpha)} = \scal{\sigma(h),\alpha} = - \scal{h,\alpha} = 0$ thus 
$\scal{h,\rho_L} = 0$. If $\alpha \in R^+$ is such that $\fg_\alpha \subset \fg(i)$ for $i\geq 0$, then $\scal{h,\alpha} = i$, thus
  $$\scal{h,\kappa} = \scal{h,2\rho} = \sum_{i = 0}^2\sum_{\alpha \in R^+, \ \fg_\alpha \subset \fg(i)} i = \dim \fg(1) + 2 \dim \fg(2).$$
We get $\dim G \cdot m = 2 \scal{\Thetab^\vee,\kappa}$, proving the result.
\end{proof}

\subsection{Contact structure}
\label{sec:contact}

In this subsection, we compute the dimension of $H \cdot m$ for $m \in T_xC \setminus \{0\}$, $C \in \cK_x$ and $\cK$ a family of minimal rational curves. 
We first gather some facts on orbits associated with symmetric spaces, and in particular prove that orbits of symmetric subgroups of $G$ are Lagrangian subvarieties in nilpotent $G$-orbits. Recall the following general definitions.

\begin{definition} Let $\widehat{M}$ be a smooth complex variety of dimension $2n + 2$ and let $M$ be a smooth complex variety of dimension $2n+1$.
  \begin{enumerate}
  \item {\it A symplectic structure} on $\widehat{M}$ is 
  a closed skew form 
 $\omega : T\widehat{M} \times_{\widehat{M}} T\widehat{M} 
 \to \widehat{M} \times \bC$
 which is everywhere nondegenerate.
  \item {\it A contact structure on $M$} is an everywhere non-vanishing map $\eta : TM \to \mL$, where $\mL$
  is a line bundle, such that the bilinear form $\theta_\eta : D \times D \rightarrow TM/D$ defined by $(u,v) \mapsto [u,v] \ (\text{mod}\ D)$ on $D := \Ker \, \eta$ is non-degenerate for all $m \in M$.
  \end{enumerate}
\end{definition}

If $\eta : TM \to \mL$ is a contact structure on $M$, then there is a natural symplectic structure $\omega$ defined by $\omega = d(p^*\eta)$ on $\widehat{M} = \mL^\times$, where $\mL^\times$ is the $\bC^\times$-bundle over $M$ with structure map $p : \widehat{M} \to M$, associated to $\mL$ and we identify $p^*\mL$ to the trivial line bundle over $\widehat{M}$.

\begin{definition}
We say that a symplectic structure $\omega$ on $\widehat{M}$ is induced by a contact structure 
$\eta : TM \to \mL$ on $M$ if 
$\widehat{M} = \mL^\times$ and $\omega = d(p^*\eta)$.
\end{definition}

The most famous examples of the above structures are given by coadjoint orbits in the dual $\fg^\vee$ of the Lie algebra $\fg$ of a connected reductive group $G$. For later purposes, we present a (non-canonical) version of \emph{Kostant-Kirillov form} which takes place in $\fg$ the Lie algebra and not $\fg^\vee$. If $\fg$ is semisimple, the Killing form $K$ identifies $\fg$ with $\fg^\vee$ and the construction is canonical.

\begin{example} 
\label{ex:coad-orb} 
Choose an invariant non-degenerate bilinear form $K$ on $\fg$
(e.g., the Killing form if $\fg$ is semisimple).
   Let $m$ be a non-zero element in $\fg$ and let $\widehat{M}_m = G \cdot m$ and $M_m = G \cdot [m] \subset \bP(\fg)$ be the orbits of $m \in \fg$ and of $[m] \in \bP(\fg)$ under the adjoint action.
Let $G_{m}$ be the isotropy subgroup of $G$ at $m$, with Lie algebra $\fg_{m}$. Define the anti-symmetric bilinear form $K_m$ on $\fg$ by $K_m(y,z) = K(m,[y,z])$. The non-degeneracy of $K$ and the equality $K_m(x,y) = K([m,y],x)$ implies that we have $\Ker \, K_m = \fg_m$. Thus, $K_m$ descends to a symplectic form
$\omega_m \colon  \fg /\fg_{m} \times \fg /\fg_{m} \rightarrow \mathbb C$ at $m \in \fg$. By the Jacobi identity, the form $\omega_m$ is closed.
\end{example}

If $m$ is such that the orbit $\widehat M_m = G.m$ is the cone in $\fg$ over $M_m = G \cdot [m] \subset \bP(\fg)$
(i.e., the affine cone minus the origin), then the arguments in 
\cite[Proposition 2.2]{Be98} adapt verbatim and yield a contact structure $\eta$ on $M_m$ which induces the symplectic form $\omega_m$. In particular, if $m$ is a nilpotent element in $\fg$, then the existence of an $\fsl_2$-triple containing $m$ ensures that $\widehat{M}_m$ is the cone over $M_m$ (see \cite[Paragraph (2.4)]{Be98}).

\medskip

Given a symplectic structure on a variety $\widehat{M}$ or a contact structure on a variety $M$, it is natural to ask for Lagrangian or Legendrian subvarieties; we recall their definitions. A Lagrangian subspace in a symplectic vector space $V$ of dimension $2m$ is an isotropic subspace of maximal dimension, \emph{i.e.}, of dimension $m$.

\begin{definition} Let $\widehat{M}$ have a symplectic structure $\omega$. A smooth subvariety $\widehat{L} \subset \widehat{M}$ is called {\it Lagrangian} if, for all $m \in \widehat{L}$, 
the subspace $T_m\widehat{L} \subset T_m\widehat{M}$ is Lagrangian for the symplectic form $\omega_m$ on $T_m\widehat{M}$.
\end{definition}

\begin{definition} Let $M$ have a contact structure $\eta$ and let $p : \widehat{M} \to M$ be the $\bC^\times$-bundle $\mL^\times$ associated to the line bundle $\mL$ with symplectic form $\omega = d(p^*\eta)$. A smooth subvariety $L \subset M$ is called {\it Legendrian} if $\widehat{L} = p^{-1}(L)$ is Lagrangian in $\widehat{M}$.
\end{definition}

\begin{example}\label{adjoint}
  Let $G$ be simple and $\fg$ its Lie algebra. Let  $m \in \fg$ be a highest weight vector. Set $\cO_{\min} = G \cdot m$ and $\bP(\cO_{\min}) = G \cdot [m] \subset \bP(\fg)$. The latter is called the {\it adjoint variety} of $G$. It is the unique closed orbit of $G$ in $\bP(\fg)$ under the adjoint action. 
Consider the Grassmannian variety $\Gr(2,\fg)$ of lines in $\bP(\fg)$, and let
  $\bL_G \subset \Gr(2,\fg)$ be the subset of lines contained in $\bP(\cO_{\min})$ 
  and passing through a given point of that variety. Then $\bL_G$ is a smooth Legendrian variety 
  in its linear span (in the Pl\"ucker embedding of $\Gr(2,\fg)$)
  and is homogeneous under the isotropy subgroup $G_{m}$ (see \cite[Theorem 1]{LM07}). 
  This Legendrian variety $\bL_G$ is called {\it the subadjoint variety}. Note that in type $\C_n$ we have $\bL_G = \emptyset$: the subadjoint variety is empty, since $\bP(\cO_{\min}) $ is the second Veronese embedding of $\bP^{2n-1}$ and hence contains no line. We will see in Subsection  \ref{para:subadjoint} 
  that $\bL_G$ (viewed as a subvariety of its linear span) 
  can be recovered as the \VMRT~ of a specific wonderful adjoint 
  symmetric variety for $G$. 
\end{example}

Let $H \subset G$ be a symmetric subgroup with group involution $\sigma$.
The following result is well known, we include a proof for the convenience of the reader.

\begin{proposition}[{\cite[Proposition 5]{ko-ra}}] 
\label{Legendrian} Let $m \in \fg^{-\sigma}$ be nonzero. 
Set $\widehat{L}_m := H \cdot m$ and $L_m := H \cdot [m]$.
  Then the variety $\widehat{L}_m$ is Lagrangian in $\widehat{M}_m$; in particular, $\dim G \cdot m = 2 \dim H \cdot m$.

  If $\widehat{L}_m$ is the cone over $L_m$ (or equivalently is stable by non-trivial homotheties), then the variety $L_m$ is Legendrian in $M_m$, in particular $\dim G \cdot [m] = 2 \dim H \cdot [m] + 1$. 
\end{proposition}

\begin{proof}
  Since $\sigma(m) = -m$, the action of $\sigma$ on $\fg$ restricts to an action on $\fg_m$. Since $\sigma$ is semisimple, we thus have $\fg_m = \fh_m \oplus \fg^{-\sigma}_m$ and $\fg/\fg_m = \fh/\fh_m \oplus \fg^{-\sigma}/\fg^{-\sigma}_m$ with $\fh_m = \fh \cap \fg_m$ and $\fg^{-\sigma}_m = \fg^{-\sigma} \cap \fg_m$. Let $u,v \in \fh/\fh_m$ (resp.~$u,v \in \fg^{-\sigma}/\fg^{-\sigma}_m$) and let $y$ and $z$ in $\fh$ (resp.~in $\fg^{-\sigma}$) be representatives of $u$ and $v$. 
  In Example \ref{ex:coad-orb}, the form $K$ can be chosen to be $\sigma$-invariant. Using the fact that 
  $\sigma(u) = \varepsilon u$, $\sigma(v) = \varepsilon v$
  for the same sign $\varepsilon$, and likewise for 
  $\sigma(y)$, $\sigma(z)$, we get 
$$\begin{array}{rcl}
\omega_m (u, v)&=&\omega_m (\sigma (u), \sigma (v)) \\
                        &=&B(m,[\sigma (y), \sigma (z)]) \\
                        &=&- B(\sigma(m),[\sigma (y), d\sigma (z)]) \\
                        &=&- B(\sigma(m), \sigma [y, z]) \\
                        &=&- B( m, [y, z]) \\
&=&- \omega_m (u, v).
  \end{array}$$
  Hence $\omega_m(u,v) = 0$ and both $\fh/\fh_m$ and $\fg^{-\sigma}/\fg^{-\sigma}_m$ are isotropic and therefore Lagrangian in $\fg/\fg_m$. This proves the first part.
  If $H \cdot m$ is stable under non-trivial homotheties, then the same holds true for $G \cdot m$.
  The result follows from this and the first part.
\end{proof}

Also, recall from \cite[Proposition 4,  Proposition 11]{ko-ra} that
the condition that $\widehat{L}_m$ is the cone over $L_m$ 
is equivalent to $m$ being nilpotent. Together with 
Proposition \ref{prop:nilp}, this readily yields:

\begin{corollary}
  \label{cor-Legendrian}
  Let $\cK$ be a family of minimal rational curves, 
  let $C \in \cK_x$ and let 
  $m \in T_xC \setminus \{ 0 \}$. Then $\widehat{L}_m$ is the cone over $L_m$, in particular $\dim H \cdot [m] = \frac{1}{2} \dim G \cdot m - 1$. 
\end{corollary}

Next, assume that $\Rb$ is not of type $\A_r$ and let $C \in \cK_x$ 
with $\cK$ a family of minimal rational curves. We obtain the following description of $\cK_x$.

\begin{theorem}
\label{theo-non-type-A}
We have $\cK_x = N \cdot C$. 
Furthermore, if $X$ is Hermitian non-exceptional, then $\cK_x$ has two components. Otherwise, $\cK_x$ is irreducible.
   \end{theorem}

\begin{proof}
 Since $\Rb$ is not of type $\A_r$, we have 
 $\partial X \cdot C = 1$ (see Remark \ref{rem-2-ou-1}). By Theorem \ref{theo-dim} and Corollary \ref{cor-Legendrian}, we have $\dim N \cdot C = \dim \cK_x$. If $X$ is non-Hermitian or Hermitian exceptional, then there exists a unique highest weight curve and $\cK_x$ is irreducible, proving the result. If $X$ is Hermitian non-exceptional, then $\cK_x$ contains two 
 highest weight curves which are exchanged by $N$; the result follows.
\end{proof}

We conclude this subsection by the following related result,
which follows from \cite[Theorem A]{Richardson}; 
we provide a direct proof for the reader's convenience.

\begin{lemma}
  The orbit $H \cdot m$ (resp. $H \cdot [m]$) is open in $(G \cdot m)^{-\sigma}$ (resp. $(G \cdot [m])^{\sigma}$). 
  In particular, $H \cdot m$ (resp.~$H \cdot [m]$) is a union of connected components of $(G \cdot m)^{-\sigma}$ (resp.~$(G \cdot [m])^\sigma$).
  Moreover, $\dim(H \cdot m) = \dim(G \cdot m)^{-\sigma}$ and $\dim(H \cdot [m]) = \dim(G \cdot [m])^\sigma$.
\end{lemma}

\begin{proof}
Note that $m$ is fixed for $-\sigma$; therefore, 
$\sigma([m]) = [m]$. We thus have inclusions 
$H \cdot m \subset (G \cdot m)^{-\sigma}$ and 
$H \cdot [m] \subset (G \cdot [m])^\sigma$. 
To prove the openness,  we only need to check that 
the tangent spaces at their respective base points
$m$ and $[m]$ agree. 
We deal with $H \cdot [m]$, the other case works in a similar way. 
The natural map 
$\fg \to T_{[m]} (G \cdot [m]),  \xi \mapsto \xi \cdot [m]$
is surjective,  and $\sigma$-equivariant as $[m]$
is $\sigma$-fixed.  Since $\sigma$ is semisimple,
this map induces a surjection on $\sigma$-fixed
subspaces, i.e. 
$ (T_{[m]}(G \cdot [m]))^{\sigma} = 
(\fg \cdot [m])^\sigma = \fg^\sigma \cdot [m] 
= \fh \cdot [m] = T_{[m]}(H \cdot [m])$.
Moreover,  
$(T_{[m]}(G \cdot [m]))^{\sigma} =
T_{[m]}((G \cdot [m])^{\sigma}) $
in view of the semisimplicity of $\sigma$ again;
this gives the desired equality 
$T_{[m]}((G \cdot [m])^{\sigma}) = T_{[m]}(H \cdot [m])$.

Since every $H$-orbit in $(G \cdot m)^{-\sigma}$ is open, 
there are only finitely many such orbits and these orbits are 
also closed, proving the last statement.
\end{proof}

\subsection{Wonderful symmetric varieties of type $\A_r$}
\label{subsection-typea}

As the above discussion shows, the case of symmetric spaces 
whose restricted root system is of type $A_r$ with $r \geq 1$ 
will present a different feature: the family $\cK_x$ has 
dimension one more than the orbit $H \cdot C$ for 
$C \in \cK_x$. In this section we prove that $\cK_x$ 
is a rational projective homogeneous variety.

Assume that the restricted root system $\Rb$ is of type 
$\A_r$ and let $(\alphab_i)_{i \in [1,r]}$ be the simple roots 
of $\Rb$ (labeled as in Bourbaki \cite{bourbaki}). For $i \in [1,r]$, 
let $\alpha_i \in R$ be a simple root such that 
$\alpha_i - \sigma(\alpha_i) = \alphab_i$. 
For $\beta \in \Delta$, let $\varpi_\beta$ be the associated 
fundamental weight of $R$. 
For each $i \in [1,r]$, recall the definition of the dominant weight 
$\lambda_i := \lambda_{\alpha_i}$ from Proposition \ref{prop-lambda-alpha}:
  $$\lambda_i = \left\{\begin{array}{ll}
  2 \varpi_{\alpha_i} & \textrm{ if $\sigma(\alpha_i) = -\alpha_i$,} \\
    \varpi_{\alpha_i} + \varpi_{\sigmab(\alpha_i)} & \textrm{ if $\sigma(\alpha_i) = -\sigmab(\alpha_i)$ and $\scal{\alpha_i^\vee,\sigma(\alpha_i)}  =0$,} \\
     \varpi_{\alpha_i} & \textrm{ otherwise.} \\
  \end{array}\right.$$

 We now list the different symmetric spaces (up to finite cover) 
    whose restricted root system is of type $\A_r$,
    the corresponding dominant weights $\lambda_1$, the irreducible 
    $G$-representations $V_{\lambda_1}$ (which will feature 
    prominently in the rest of this section), and the 
    corresponding $H$-representations $\fg^{-\sigma}$.

    \begin{center}
      \begin{tabular}{c|c|c|c|c}
        $G/H$ & Rank & $\lambda_1$ & $V_{\lambda_1}$ & $\fg^{-\sigma}$ \\
        \hline
        $\SL_{r+1}\times\SL_{r+1}/\SL_{r+1}$ & $r$ & $(\varpi_1,0) + (0,\varpi_r)$ 
        & ${\rm End}(\bC^{r+1})$ & $\fsl_{r+1}$ \\
        $\SL_{r+1}/{\rm SO}_{r+1}$ & $r$ & $2\varpi_1$ 
        & $S^2(\bC^{r+1})$ & $S^2(\bC^{r+1})_0$ \\
        $\SL_{2r+2}/\Sp_{2r+2}$ & $r$ & $\varpi_2$ 
        & $\Lambda^2(\bC^{2r+2})$ & $\Lambda^2(\bC^{2r+2})_0$ \\
        ${\rm SO}_{n}/{\rm S}(\OO_1 \times \OO_{n-1})$ & $1$ & $\varpi_1$ 
        & $\bC^n$ & $\bC^{n-1}$ \\
        $E_6/F_4$ & $2$ & $\varpi_1$ & $\bC^{27}$ & $\bC^{26}$ \\
      \end{tabular}
    \end{center}
    
Here $S^2(\bC^{r+1})_0$ denotes the $\SO_{r+1}$-stable complement of 
$\bC q$ in $S^2(\bC^{r+1})$ with $q$ being the standard quadratic form, 
and $\Lambda^2(\bC^{2r+2})_0$ denotes the $\Sp_{2r+2}$-stable
complement of $\bC \omega$ in $\Lambda^2(\bC^{2r+2})$ with 
$\omega$ being the standard symplectic form. 
As a consequence of this classification, we see that 
$V_{\lambda_1} = \fg^{-\sigma} \oplus \bC$ as $H$-representations.

Note that since $\Rb$ is reduced, none of the symmetric spaces we consider is exceptional.
Let $\Rb^\vee$ be the dual root system and let 
$(\overline{\alpha}_i^\vee)_{i \in [1,r]}$ be the simple coroots.
Then the weights $(\lambda_i)_{i \in [1,r]}$ are the fundamental weights 
of $\Rb$. 
The dominant chamber is thus the cone generated by 
the fundamental coweights $(\lambda_i^\vee)_{i \in [1,r]}$.

Next,  using the above list,  we obtain a geometric
construction of the wonderful compactification:

\begin{proposition}\label{prop:eclt}
Let $G/H$ be an adjoint indecomposable symmetric space with restricted 
root system of type $\A_r$ and let $\lambda_1$ be as above.

\begin{enumerate}

\item The group $G$ acts on $\bP(V_{\lambda_1})$ with $r + 1$ orbits 
whose closures $(Z_i)_{i \in [1,r+1]}$ satisfy the following inclusions: 
$Z_1 \subsetneq Z_2 \subsetneq \cdots \subsetneq Z_{r+1} = \bP(V_{\lambda_1})$.  The open orbit is isomorphic to $G/H$.

\item The join $J(Z_1,Z_i)$ (i.e., the union of lines joining 
$Z_1$ and $Z_i$) equals $Z_{i+1}$ for all $i$.
    
\item The wonderful compactification $X$ 
is equipped with a birational $G$-equivariant morphism
$f : X \to \bP(V_{\lambda_1})$. 
If $r = 1$, then $f$ is an isomorphism. 
If $r \geq 2$, then $f$ is the composition of the blow-ups 
of  the strict transforms of the orbit 
closures $Z_1$, $\ldots$, $Z_r$ in this order.
Moreover, these strict transforms are smooth.
\end{enumerate}

\end{proposition}

\begin{proof}
(1) and (2) In all cases except the last one, the $G$-orbits are given 
by the rank of matrices (plain  matrices, symmetric matrices or 
skew-symmetric matrices) and the assertions follows from this. 
The case of $E_6/F_4$ is a classical result (see 
e.g.~\cite[Proposition 4.1]{LM01}).

(3) This is again a classical result in the first three cases,
see \cite[Theorem 1]{Vainsencher} for 
$\SL_{r + 1} \times \SL_{r+1}/\SL_{r+1}$
(then $X$ is the moduli space of complete collineations)
and \cite[Theorems 10.1, 11.1]{Thaddeus} for 
$\SL_{r+1}/\SO_{r+1}$ and $\SL_{2r+2}/\Sp_{2r+2}$
(complete quadrics and complete skew forms).
The next case of $\SO_n/{\rm S}(\OO_1 \times \OO_{n-1})$
is easy, as we then have $X = \bP^{n-1} = \bP(V_{\lambda_1})$.

It remains to treat the case of $E_6/F_4$.
This is a symmetric space of rank $2$ and 
its equivariant compactification $\bP(V_{\lambda_1})$ 
has a unique closed orbit $Z_1$, the projectivization 
of the orbit of highest weight vectors in $V_{\lambda_1}$.
We now use the theory of spherical embeddings for which 
we refer to \cite{Timashev}, especially Section 15.1
(the classification of spherical embeddings in terms of
colored fans) and Section 26.8 (the colored equipment of
symmetric spaces). By Theorem 26.25 in loc.~cit., 
the valuation cone $\cV$ is the opposite of 
the dominant cone of type $\A_2$; moreover, there are 
two colors and these are mapped to positive
multiples of the two simple roots. Since 
$\bP(V_{\lambda_1})$ is a simple complete embedding,
its colored fan consists of a single colored
cone which contains the valuation cone. Also,
this colored cone does have a color, since
$\lambda_1$ is not in the interior of the cone 
generated by the weight monoid (recall that the weight monoid of a spherical homogeneous space
$G/H$ consists of the dominant weights $\lambda$ such that $V_\lambda$
contains nonzero $H$-fixed vectors; see \cite[Proposition 26.24]{Timashev} for its description in the symmetric case). This implies that
the colored cone of $\bP(V_{\lambda_1})$ is 
generated by $\cV$ and one simple root corresponding
to the color. By the classification of orbits
in spherical embeddings (see Theorem 15.4 in loc.~cit.),
it follows that the $G$-orbit closures in 
$\bP(V_{\lambda_1})$ are exactly
$Z_1 \subsetneq Z_2 \subsetneq Z_3 = \bP(V_{\lambda_1})$.
Moreover, the boundary $Z_2$ is a divisor, since its 
complement $G/H$ is affine. 

Denote by $\varphi : X' \to \bP(V_{\lambda_1})$ 
the blow-up along $Z_1$. Then $X'$ is a smooth
projective equivariant compactification of $G/H$,
and its boundary is the union of two prime divisors:
the exceptional divisor $X'_1$ and the strict transform
$X'_2$ of $Z_2$.  Moreover,  
$X'_2 \setminus X'_1 = Z_2 \setminus Z_1$ 
is a unique $G$-orbit. We now claim that 
$X'_1 \setminus X'_2$ and $X'_1 \cap X'_2$ 
are $G$-orbits as well. 

To check this, we identify $Z_1$ to $G/P_1$, 
where $P_1$ is the maximal parabolic subgroup 
of $G$ associated with the fundamental weight 
$\lambda_1 = \varpi_1$.  Denote by $M$
the normal space to $Z_1$ in $\bP(V_{\lambda_1})$
at the base point of $G/P_1$. 
Then $M$ is a representation of $P_1$, 
and the $G$-variety $X'_1$ is isomorphic to
the projectivization of the normal bundle
to $Z_1$ in $Z_3$, that is, 
the homogeneous projective bundle
$G \times^{P_1} \bP(M)$.  
Thus, the $G$-orbits in $\bP(V_{\lambda_1})$ 
correspond bijectively to the $P_1$-orbits in $\bP(M)$. 
So it suffices in turn to show that $P_1$ acts on
$\bP(M)$ with two orbits.  But the Levi subgroup
$L_1$ of $P_1$ is isomorphic to 
$\SO_{10} \times \bC^*$ up to finite cover,
and $M = \bC^{10}$ where $\SO_{10}$ acts
via its standard representation and $\bC^*$
acts by scalar multiplication. Therefore,
$L_1$ acts on $\bP(M)$ with two orbits:
a quadric and its complement.  As
$P_1$ does not act transitively on 
$\bP(M)$, it acts with two orbits as well,
proving the claim.

By that claim, $X'$ is a smooth projective 
embedding of $G/H$ and its boundary is the
union of two prime divisors $X'_1,X'_2$ 
intersecting properly along the unique closed orbit.
It follows that the colored fan of $X'$ consists
of a unique cone: the valuation cone $\cV$.
Thus, $X'$ is isomorphic to $X$ by the classification 
of embeddings of $G/H$ again.
As a consequence, the boundary divisor $X'_2$ is smooth.
\end{proof}

\begin{remark}
The above statements (1) and (2) can be proved in a uniform way 
using Jordan algebras: the representation $V_{\lambda_1}$ has 
the structure of a Jordan algebra with structure group $G$ 
and the stabiliser of the unit element is $H$. The above orbit structure 
is then explained by the notion of rank for elements in a Jordan algebra. 
We refer to \cite{springer-jordan} and \cite{BP} for more on Jordan 
algebras. We were however not able to fully relate symmetric spaces 
with restricted root systems of type $\A_r$ to Jordan algebras 
without using a case by case check, so we refrained from using them.

Likewise, the above statement (3) can be deduced in a 
uniform way from the equality $\dim G/N = \dim V_{\lambda_1} - 1$
by using embedding theory of spherical homogeneous spaces,
as suggested by an anonymous referee.
But the only proof of this equality that we know goes via a 
case-by-case checking.
\end{remark}

\begin{theorem}
\label{them: vmrt of type A}
Let $X$ be the wonderful compactification of an adjoint 
indecomposable symmetric space $G/H$ with restricted root system 
of type ${\rm A}_r$.
\begin{enumerate}
\item There is a unique family of minimal rational curves $\cK$.
\item The tangent map $\cK_x \to \cC_x$ is an isomorphism.
\item If $r = 1$, then $\cC_x = \bP(\fg^{-\sigma})$.
\item If $r \geq 2$, then $\cK_x$ is isomorphic to 
the closed $G$-orbit in $\bP(V_{\lambda_1})$.
\item In both cases, $H \cdot C$ is a prime divisor in $\cK_x$,
where $C$ denotes the unique highest weight curve.
\end{enumerate}
\end{theorem}

\begin{proof}
 If $r = 1$, then $X = \bP(V_{\lambda_1})$ with $V_{\lambda_1} = \fg^{-\sigma} \oplus \bC$ and $x = [\bC]$. Thus there is a unique minimal family $\cK$ and it consists
 of lines in $X$. The result follows in this case.
 
 If $r \geq 2$, then there is also a unique minimal family $\cK$,
 as follows from Proposition \ref{prop:min} and Corollary 
 \ref{cor:1or2}. To determine this family, we use the description of $X$ 
 as an iterated blow-up in Proposition \ref{prop:eclt}, together with
 Proposition 9.7 in \cite{Hwang-Fu}. Arguing as in 
 \cite[Proposition 5.1]{BF} for the group case of type $\A_r$, 
 we see that $\cK_x$ is the set of strict transforms of the 
lines in $\bP(V_{\lambda_1})$ that pass through a general point and meet the closed orbit. The tangent map $\cK_x \to \cC_x$ is an isomorphism 
and the \VMRT~ is therefore isomorphic to the closed $G$-orbit in 
$\bP(V_{\lambda_1})$.

This proves all the assertions except for (5), which follows from 
(3), (4) and the isomorphism of $H$-representations 
$V_{\lambda_1} \simeq \fg^{-\sigma} \oplus \bC$.

\end{proof}

\subsection{Minimal rational curves on wonderful compactifications}
\label{subsection-summary}

We summarise our results.   Let $X$ be the wonderful compactification of an adjoint indecomposable symmetric space $G/H$ with base point $x$ and let $\cK$ be a family of minimal rational curves.

\begin{theorem}
  \label{theo-adj}
  
  \begin{enumerate}
    \item Every irreducible component of $\cK_x$ contains a unique highest weight curve $C$. Moreover, $\cK_x$ is smooth and
    equidimensional of dimension
    $ \dim H \cdot C + \partial X \cdot C - 1$, 
    with $\dim H \cdot C  = \scal{\Thetab^\vee,\kappa}$
   and $\partial X \cdot C =\scal{\Thetab^\vee,\Sigma}$.
    \item We have $\partial X \cdot C \in \{ 1,2 \}$. Moreover, 
    $\partial X \cdot C = 2$ if and only if
    the restricted root system is of type $\A_r$.

    \item Assume that $\partial X \cdot C = 1$. 
    Then $\cK_x = H \cdot C$. Furthermore, 
    if $X$ is Hermitian non-exceptional, then $\cK_x$ 
    has two components. Otherwise, $\cK_x$ is irreducible.

    \item Assume that $\partial X \cdot C = 2$, so that 
    the restricted root system of $X$ is of type $\A_r$.
      \begin{enumerate}
      \item If $r = 1$, then $\cK_x \simeq \bP(\fg^{-\sigma})$.
      \item If $r \geq 2$, then there exists a $G$-equivariant birational morphism $X \to \bP(V)$ for some irreducible $G$-representation $V$ and $\cK_x$ is isomorphic to the closed $G$-orbit in $\bP(V)$. The orbit $H \cdot C$ is a prime divisor in $\cK_x$.
      \end{enumerate}
      \item The orbits $H \cdot C$ are described in Table \ref{table-class}.
      \item The tangent map $\cK_x \to \cC_x$ is an isomorphism. 
      
  \end{enumerate}
\end{theorem}

\begin{proof}
(1) This follows from Proposition \ref{prop:red}
and Theorem \ref{theo-dim}.

(2) This follows from Remark \ref{rem-2-ou-1}.

(3) This follows from Theorem \ref{theo-non-type-A}.

(4) This follows from (2) and Theorem 
\ref{them: vmrt of type A}.

(5) This is proved in the Appendix.

(6) This follows from Theorem \ref{them: vmrt of type A}
again for $\Rb$ of type $\A_r$. 
Assume that $\Rb$ is not of type $\A_r$. By Proposition \ref{prop:red}, 
the tangent map $\tau_x : \cK_x \to \cC_x$ is $H$-equivariant, finite 
and birational. Furthermore by (3), the variety $\cK_x$ is $H$-homogeneous. 
Thus, $\cC_x$ is $H$-homogeneous as well, and $\tau_x$ is bijective.
\end{proof}

\section{Minimal rational curves on complete symmetric varieties}
\label{sec:mfcss}

We are now in a position to prove our main results. 
Let $X$ be a complete symmetric variety and 
let $\cK$ be a family of minimal rational curves on $X$. 
Let $\pi : X \to X_\ad$ be the map from $X$ to the wonderful 
compactification of the associated adjoint symmetric space, 
and let $C \in \cK_x$.

\begin{theorem}
\label{thm:VMRT=orbit}
 
 \begin{enumerate}
  \item If $C$ is contracted by $\pi$, then $\cK_x$ is isomorphic to a linear subspace of $\bP(\fg^{-\sigma} \cap \fz)$.\\
\par\vbox{\parbox[t]{\linewidth}{Assume that $\pi$ does not contract $C$.}}

 \item The map $\pi$ induces an isomorphism between $C$ and its image $D := \pi(C)$ and there exists a unique indecomposable factor $X_C$ of $X_\ad$ such that the composition map $\pi_C : X \to X_\ad \to X_C$ does not contract $C$.

    \item We have $\partial X \cdot C \in \{ 1,2 \}$.

    \item If $\partial X \cdot C = 1$, then  
    $\cK_x = H \cdot C$. Moreover, the components of 
    $\cK_x$ are isomorphic to the components of $H \cdot D$. 

    \item If $\partial X \cdot C = 2$, then the restricted root system of $X_C$ is of type $\A_r$.
      \begin{enumerate}
      \item If $r = 1$, then each component of 
      $\cK_x$ is isomorphic to $\bP(\fg^{-\sigma}_C)$.
      \item If $r \geq 2$, then there is a $G$-equivariant birational morphism $X_C \to \bP(V)$ for some irreducible $G$-representation $V$ and 
      each component of $\cK_x$ is isomorphic to the closed $G$-orbit in $\bP(V)$. The orbit $H \cdot C$ is a prime divisor in $\cK_x$.
      \end{enumerate}
      \item The orbits $N \cdot C$ are described in Table \ref{table-class}.
      \item The variety $\cK_x$ is smooth and the tangent map 
      $\cK_x \to \cC_x$ is an isomorphism.
  \end{enumerate}
\end{theorem}

\begin{proof}
(1) This follows from Lemma \ref{lem:red}.

(2) This follows from Proposition \ref{prop:red}. 

(3) This follows from Proposition \ref{prop:red} and Theorem \ref{theo-adj}.

(4) If $\partial X \cdot C = 1$ and $X_C$ is not of type 
$\A_r$, then 
$\partial X \cdot C = \partial X_C \cdot D$.
Moreover, there is a unique family of minimal rational curves 
$\cL$ containing $D = \pi(C)$ and a finite $H$-equivariant map $\pi_{*,x} : \cK_x \to \cL_{\pi_C(x)}$
(Proposition \ref{prop:red} again).
By Theorem \ref{theo-adj}, we have 
$\cL_{\pi_C(x)} = H \cdot D$; in particular, every component of $\cL_{\pi_C(x)}$ is 
homogeneous under $H^0$. Using  
Proposition \ref{prop:red} once more,
it follows that $\pi_{*,x}$ induces an isomorphism on components. 

If $X_C$ is of type ${\rm A}_r$, then 
$\partial X_C \cdot D = 2$. Thus, the image of each component of
$\cK_x$ in $\cL_{\pi_C(x)}$ has codimension $1$ and must be equal to $H \cdot D$. The result follows from this by a similar argument as in the previous case.

  (5) If $\partial X \cdot C = 2$, then the restricted root system of $X_C$ is of type 
  $\A_r$ and 
  $\partial X \cdot C = \partial X_C \cdot D$. 
  Again, there is a unique family of minimal rational curves $\cL$ containing 
  $D$, and a finite $H$-equivariant map 
  $\pi_{*,x} : \cK_x \to \cL_{\pi_C(x)}$. 
  Moreover, 
  $\cL_{\pi_C(x)}$ is irreducible and has the same dimension as $\cK_x$; thus, $\pi_{*,x}$ induces an isomorphism on each component of $\cK_x$.
  
  (7) By Proposition \ref{prop:red}, $\cK_x$ is smooth and the tangent map 
  $\cK_x \to \cC_x$ is finite, birational and $H$-equivariant, 
  therefore an isomorphism if $\cK_x = H \cdot C$. 
  If $\cK_x$ is not $H$-homogeneous, then $X_C$ is of type ${\rm A}_r$ 
  and the result follows from Theorem \ref{thm:VMRT=orbit}(5).
\end{proof}

\section{Appendix}
\label{sec:app}

The goal of this appendix is twofold: we first prove basic results on restricted root systems used to describe curves and divisors on wonderful compactifications. We also obtain characterisations of exceptional wonderful varieties useful to establish Table \ref{table-class}. We then give an easy way to describe, using the Kac diagram of the symmetric space, the components of the $H$-orbit $H \cdot C$ in $\cK_x$, where $C$ is a highest weight curve.
Finally, in Table \ref{table-class}, we give a list, based on  the classification of symmetric spaces, of families of minimal rational curves and \VMRT~ of wonderful symmetric varieties.

\subsection{Restricted root systems}
\label{ap:rrs}

In this subsection we prove useful results on restricted root systems that might be well known to experts, but for which we could not find a good reference. 

\begin{lemma}
\label{lem:highest-root-ap}
The restricted root $\Thetab$ is the highest root of $\Rb$.
\end{lemma}

\begin{proof}
For $\alpha \in R$,  write 
$\alpha = \sum_{\beta \in \Delta} c_\beta \beta$ 
with all the $c_\beta$ of the same sign.  We have 
$\alphab = \alpha - \sigma(\alpha) 
= \sum_{\beta \in \Delta_1} c_\beta (\beta - \sigma(\beta)) 
= \sum_{\beta \in \Delta_1} c_\beta \betab$ 
and the result follows from this.
\end{proof}

\begin{lemma}
\label{lemm:w_0-et-sigma}
Let $w_0 \in W$ be the longest element, then the actions of $\sigma$ and $w_0$ on roots commute. In particular, $\sigma(w_0) = w_0$.
\end{lemma}

\begin{proof}
It follows from \cite[Section 5.2]{Springer}, that there exists an involution $\tau$ of $\Delta$ such that $\tau(\Delta_0) = \Delta_0$ with  $\sigma = w_{0,L} w_0\tau$ and $\sigmab = -w_0\tau$ (see also \cite[Section 1.5]{decs}). Note that $-w_0$, $\sigmab$ and $\tau$ are involutions preserving $R^+$ and $\Delta$. Furthermore there is at most one non-trivial such involution except in type $D_4$ (in type $D_4$, we have $-w_0 = \id$). 
% $\sigmab$ is either equal to $\id$ or to $-w_0$. 
 %Therefore %, in any case, 
Thus, $-w_0$ is trivial or equal to $\sigmab$ or $\tau$. In particular $w_0(\Delta_0) = - \Delta_0$ and $w_0(L) = L$ (this can also be easily checked using the classification).

Recall that $\sigma = -w_{0,L} \sigmab$ where $w_{0,L}$ is the longest element in the Weyl group $W_L$ of the pair $(L,T_{\s})$, for this see \cite[Page 149]{Timashev}. Therefore $w_0\sigma = -w_0w_{0,L}\sigmab = -w_{w_0(L)}w_0\sigmab = -w_{0,L}\sigmab w_0 = \sigma w_0$. The result follows.
\end{proof}

\begin{corollary}
\label{cor-w_0-Theta-ap}
We have $w_0(\Thetab) = - \Thetab$.
\end{corollary}

\begin{proof}
We have $w_0(\Thetab) = w_0(\Theta) - w_0(\sigma(\Theta)) = w_0(\Theta) - \sigma(w_0(\Theta)) = -\Theta + \sigma(\Theta) = -\Thetab$. 
\end{proof}

We now prove a characterisation of non-reduced restricted root systems.

\begin{proposition}
    \label{char-non-red-ap}
    Let $\alpha \in R \setminus R^\sigma$. We have the equivalence: $\alphab, 2\alphab \in \Rb$ 
    $\Leftrightarrow$ $\scal{\alpha^\vee,\sigma(\alpha)} = 1$.
%        Let $\alpha \in \Delta_1$. We have the equivalence: $\alphab, 2\alphab \in \Rb$  $\Leftrightarrow$ $\scal{\alpha^\vee,\sigma(\alpha)} = 1$.
  \end{proposition}

    \begin{proof}
If $\scal{\alpha^\vee,\sigma(\alpha)} = 1$,  then 
$\alphab = \alpha - \sigma(\alpha) =  \beta \in R$ 
and $\sigma(\beta) = - \beta$.  Thus, 
$\betab = %2\beta = 
2\alphab \in \Rb$. 

Conversely, assume that $\alphab, 2\alphab \in \Rb$.  Recall that for $\gamma \in R$,  there are three possibilities:  
%$\sigma(\gamma) = -\gamma$, $\scal{\gamma^\vee,\sigma(\gamma)} = 0$ or $\scal{\gamma^\vee,\sigma(\gamma)} = 1$. We get
   $$(\gammab,\gammab) = \left\{\begin{array}{ll}
    4(\gamma,\gamma) & \textrm{ if $\sigma(\gamma) = -\gamma$}, \\
    2(\gamma,\gamma) & \textrm{ if $\scal{\gamma^\vee,\sigma(\gamma)} = 0$}, \\
        (\gamma,\gamma) & \textrm{ if $\scal{\gamma^\vee,\sigma(\gamma)} = 1$}. \\
\end{array}\right.$$
Let $\gamma \in R$ with $\gammab = 2\alphab$, then $(\gammab,\gammab) = 4(\alphab,\alphab)$ is equal to $16(\alpha,\alpha)$, $8(\alpha,\alpha)$ or $4(\alpha,\alpha)$ if $\sigma(\alpha) = -\alpha$, $\scal{\alpha^\vee,\sigma(\alpha)} = 0$ or $\scal{\alpha^\vee,\sigma(\alpha)} = 1$. In the first case, we would have $(\gamma,\gamma) \geq 4(\alpha,\alpha)$ which is impossible for a reduced root system. In the second case, we have $(\alpha,\sigma(\alpha)) = 0$ and $\gamma = \alpha - \sigma(\alpha)$ is a root. This implies that $\alpha + \sigma(\alpha)$ is a root but the corresponding root vector is $[e_\alpha,e_{\sigma(\alpha)}]$ which lies in $\fg^{-\sigma}$. This means that $\alpha + \sigma(\alpha)$ is non-compact imaginary which is impossible since we work with a split torus, see \cite[Page 149]{Timashev}. We thus have $\scal{\alpha^\vee,\sigma(\alpha)} = 1$.
\end{proof}

    \begin{corollary}
    \label{cor-non-red-ap}
     We have the equivalences:
     $$\textrm{$\Rb$ is non-reduced } \Leftrightarrow \exists \alpha \in \Delta_1, \scal{\alpha^\vee,\sigma(\alpha)} = 1 \Leftrightarrow \exists \alpha \in R, \scal{\alpha^\vee,\sigma(\alpha)} = 1.$$
  \end{corollary}

\begin{proof}
If $\Rb$ is non-reduced then it is of type $\textrm{BC}_r$ and there exists $\alpha \in \Delta_1$ with $\alphab,2\alphab \in \Rb$. The second implication from left to right is clear. If $\alpha \in R$ is such that $\scal{\alpha^\vee,\sigma(\alpha)} = 1$, then $\alphab = \alpha - \sigma(\alpha) = \gamma$ is a root and $\gammab = 2\gamma = 2\alphab \in \Rb$; thus, $\Rb$ is non-reduced.
\end{proof}

 \begin{corollary}
    \label{cor-char-exc-ap}
    Let $\alpha \in \Delta_1$ such that $\alphab,2\alphab \in \Rb$.
    \begin{enumerate}
    \item The root $\alphab$ is the unique root of $\Db$ with $\alphab, 2\alphab \in \Rb$.
      \item The variety $X$ is exceptional if and only if $\sigmab(\alpha) \neq \alpha$.
    \end{enumerate}
  \end{corollary}

\begin{proof}
(1) Assume that $\alpha \in \Delta_1$ is such that $\alphab ,2\alphab \in \Rb$. Then $\alphab \in \Db$ is the unique simple root whose double is a root in the root system $\textrm{BC}_r$ and is therefore unique.

(2) If $X$ is exceptional then for $\alpha \in \Delta_1$ an exceptional root, we have $\sigmab(\alpha) \neq \alpha$ by definition and $\alphab,2\alphab \in \Rb$ by Remark \ref{remark : excep-reduce}. 
For the converse, 
if $\alpha \in \Delta_1$ is such that $\alphab,2\alphab \in \Rb$, then $\scal{\alpha^\vee,\sigma(\alpha)} = 1$. 
If furthermore $\sigmab(\alpha) \neq \alpha$ then $\alpha$ is exceptional by definition .
\end{proof}

For $\alpha \in \Delta_1$,  write 
$\alphab = \alpha - \sigma(\alpha) 
= \alpha + \sigmab(\alpha) 
+ \sum_{\beta \in \Delta_0} c_\beta \beta$ 
and define the support of $\alphab$ by 
$\supp(\alphab) = \{ \alpha , \sigmab(\alpha) , \beta 
\ | \ \beta \in \Delta_0 \textrm{ with } c_\beta > 0\}$.
    
    \begin{lemma}
      \label{lem-dom-ap}
Let $\alpha \in \Delta_1$,  then $\alphab$ is dominant on 
$\supp(\alphab)$.  More precisely,  we have
$$\scal{\alpha^\vee,\alphab} > 0, 
\scal{\sigmab(\alpha)^\vee,\alphab} > 0 
\textrm{ and } \scal{\beta^\vee,\alphab} = 0 
\textrm{ for $\beta \in \supp(\alphab)$} \cap \Delta_0.$$
    \end{lemma}

    \begin{proof}
For $\beta \in \Delta_0$,  we have 
$\scal{\beta^\vee,\alphab} 
= \scal{\sigma(\beta)^\vee,\sigma(\alphab)} 
= \scal{\beta^\vee,-\alphab}$; 
thus, $\scal{\beta^\vee,\alphab} = 0$. 
We have $\scal{\alpha^\vee,\alphab} 
= \scal{\alpha^\vee,\alpha - \sigma(\alpha)} 
= 2  - \scal{\alpha^\vee,\sigma(\alpha)}$ 
and since $\alpha$ and $\sigma(\alpha)$ 
have the same length,  we have 
$\scal{\alpha^\vee,\sigma(\alpha)} \leq 1$ 
proving the result. 
    \end{proof}

Recall that $\Theta$ denotes the highest root of $R$, 
and $\theta$ the highest short root (if $R$ is simply laced, 
then $\Theta = \theta$ and all roots are long and short).
    
\begin{lemma}
      We have the following equivalences
      \begin{enumerate}
      \item $\sigma(\Theta) = - \Theta$ $\Leftrightarrow$ 
      there exists a long root $\alpha$ with $\sigma(\alpha) = -\alpha$.
      \item $\sigma(\theta) = - \theta$ $\Leftrightarrow$ 
      there exists a short root $\alpha$ with $\sigma(\alpha) = -\alpha$.
      \end{enumerate}
\end{lemma}

\begin{proof}
The implications from left to right in (1) and (2) are clear. We prove the converse implications. The proofs in both cases are similar.

 We first prove that the converse implications in (1) and (2) are implied by the  following claim: if $\sigma(\alpha) = -\alpha$ and there exists $\beta \in \Delta$ with $s_\beta(\alpha) > \alpha$, then there exists $w \in W$ with $\sigma(w) = w$ such that $w(\alpha) > \alpha$ 
 (recall that $\sigma$ acts on $W$ by conjugaison).

Assume that the above claim is true. We prove the converse implications in (1) and (2) by induction on roots for their natural order: if $\alpha$ is a positive root such that $\sigma(\alpha) = -\alpha$ and $\alpha$ is not maximal, we produce a root $\alpha' > \alpha$ with the same length as $\alpha$ and such that $\sigma(\alpha') = -\alpha'$. Indeed, if $\alpha$ is not maximal, then there exists $\beta \in \Delta$ with $s_\beta(\alpha) > \alpha$. By the above claim, there exists $w \in W$ with $\sigma(w) = w$ and $w(\alpha) > \alpha$. We thus have $\sigma(w(\alpha)) = \sigma(w)(\sigma(\alpha)) = w(-\alpha) = - w(\alpha)$ with $w(\alpha) > \alpha$. Therefore, if the claim is true, we get by induction that $\sigma(\alpha') = -\alpha'$ for $\alpha'$ the highest root with the same length as $\alpha$, proving the implication from right to left of (1) and (2).

We now prove our claim,  so let $\alpha \in R$ such that 
$\sigma(\alpha) = - \alpha$ and $\beta \in \Delta$ with 
$s_\beta(\alpha) > \alpha$.  In particular 
$\scal{\beta^\vee,\alpha} < 0$. 
We have four possible cases: 
$\sigma(\beta) = \beta$, 
$\sigma(\beta) = - \beta$, 
$\scal{\beta^\vee,\sigma(\beta)} = 0$ 
or $\scal{\beta^\vee,\sigma(\beta)} = 1$.

If $\sigma(\beta) = \beta$,  then 
$\scal{\beta^\vee,\alpha} 
= \scal{\sigma(\beta)^\vee,\sigma(\alpha)} 
= - \scal{\beta^\vee,\alpha}$; thus, 
$\scal{\beta^\vee,\alpha} = 0$ a contradiction, 
so this case does not occur.

If $\sigma(\beta) = -\beta$,  then $w = s_\beta$ 
works since $\sigma(w) = w$.

If $\scal{\beta^\vee,\sigma(\beta)} = 0$, 
then set $w = s_\beta s_{\sigma(\beta)}$. 
Since $s_\beta$ and $s_{\sigma(\beta)}$ commute, 
we have $\sigma(w) = w$.  Furthermore,  we have  
$w(\alpha) = \alpha - \scal{\beta^\vee,\alpha} \beta 
- \scal{\sigma(\beta)^\vee,\alpha} \sigma(\beta) 
=   \alpha - \scal{\beta^\vee,\alpha} \beta + \scal{\beta^\vee,\alpha} \sigma(\beta)$. 
But since $\sigma(\beta) \neq \beta$,  we have 
$\sigma(\beta) < 0$ and $w(\alpha) > \alpha$.

Finally, if $\scal{\beta^\vee,\sigma(\beta)} = 1$,  define  
$\delta = s_{\sigma(\beta)}(\beta)  
= \beta - \sigma(\beta) = \betab$; 
then $\delta$ is a root. Let $w = s_\delta$. 
We have $\sigma(\delta) = -\delta$; thus, $\sigma(w) = w$. 
Furthermore, we have 
$w(\alpha) = \alpha - \scal{\delta^\vee,\alpha} \delta 
= \alpha 
- \scal{s_{\sigma(\beta)}(\beta)^\vee,\alpha}(\beta - \sigma(\beta))$
and 
$\scal{s_{\sigma(\beta)}(\beta)^\vee,\alpha} = \scal{\beta^\vee,s_{\sigma(\beta)}(\alpha)} = \scal{\beta^\vee,\alpha - \scal{\sigma(\beta)^\vee,\alpha} \sigma(\beta)} = \scal{\beta^\vee,\alpha}  - \scal{\sigma(\beta)^\vee,\alpha} = 2\scal{\beta^\vee,\alpha}$. Thus, $w(\alpha) = \alpha - 2\scal{\beta^\vee,\alpha} (\beta - \sigma(\beta))$ and since $\sigma(\beta) < 0$, we get $w(\alpha) > \alpha$.
\end{proof}
    
\begin{corollary} \label{coro-coef1-ap}
Assume that $\Rb$ is non-reduced. 
We have the equivalences:
$$X \textrm{ is exceptional } \Leftrightarrow \ 
\sigma(\Theta) = -\Theta \ \Leftrightarrow 
\textrm{ $R$ is simply laced.}$$
Furthermore,  if $X$ is exceptional and $\alpha$ is 
an exceptional root,  then its coefficient in 
the expansion of $\Theta$ as a linear combination of
simple roots is equal to $1$.
\end{corollary}

\begin{proof}
Assume that $\Rb$ is non-reduced and let $\alpha \in \Delta_1$ 
such that $\alphab,2\alphab \in \Rb$.  Then 
$\scal{\alpha^\vee,\sigma(\alpha)} = 1$ and 
$\gamma = \alpha - \sigma(\alpha) = \alphab$ 
is a root such that $\sigma(\gamma) = -\gamma$. 
Therefore we either have $\sigma(\Theta) = - \Theta$ 
or $\sigma(\theta) = - \theta$.

If $X$ is exceptional, then 
$\sigmab$ is a non-trivial involution of the Dynkin diagram 
and this implies that $R$ is simply laced. 
In particular $\sigma(\Theta) = - \Theta$ (since $\Theta = \theta$). 
If $\alpha \in \Delta_1$ is exceptional,  then 
$\sigmab(\alpha) \neq \alpha$ and the coefficients of such roots 
in $\Theta$ are always equal to $1$.
    
On the other hand if $X$ is non-exceptional, 
then $\sigmab(\alpha) = \alpha$ and 
$\gamma = \alpha - \sigma(\alpha)$ is dominant on its support 
and bigger than $2\alpha$.  If $\gamma$ is long then 
it is the highest root of $\Supp(\gamma)$, but this is impossible 
by the discussion on pages 150--151 in \cite{Timashev}. 
Therefore, $\gamma$ is short and $R$ is not simply laced. 
Assume that $\sigma(\Theta) = - \Theta$ 
and let $( - , - )$ be a $(W,\sigma)$-invariant scalar product on 
$\fX_{\bR}$ such that long roots have length $2$. 
We have 
$(\Thetab,\Thetab) = 4(\Theta,\Theta) = 8$ and 
$\gammab = 2\gamma  = 2\alphab$ is such that 
$(\gammab,\gammab) = 4(\gamma,\gamma) = 4$. 
A contradiction since in $\Rb$ all roots which are the double 
of another root have the same length. 
Therefore, $\sigma(\Theta) \neq - \Theta$. 
  \end{proof}

\begin{corollary}
$G/H$ is exceptional if and only if $G$ is simply laced and $\Rb$ is non-reduced.
\end{corollary}

Recall the definition of $\alphah$ from the end of Subsection \ref{subsection:rrs}.

\begin{lemma}
  \label{lem-coef-rac-ap}
Assume that $\alpha \in \Delta_1$ is an exceptional root, 
then the coefficient of $\alphah^\vee$ in the expansion of 
$\Thetab^\vee$ in terms of simple coroots of $\Rb$ is equal to $1$.
\end{lemma}

\begin{proof}
Since $\alpha$ is exceptional,  we have 
$\alphab,2\alphab \in \Rb$ and 
$\alphah^\vee = \frac{1}{2}\alphab^\vee 
= \frac{1}{2}(\alpha^\vee - \sigma(\alpha)^\vee)$. 
On the other hand, since $X$ is exceptional,  we have 
$\sigma(\Theta) = - \Theta$
so that $\Thetab^\vee = \frac{1}{2}\Theta^\vee$. 
Since the coefficient of $\alpha$ in the expansion of $\Theta$ 
in terms of simple roots is equal to $1$,   the same is true for the coefficient of $\alphah^\vee$ in the expansion of $\Thetab^\vee$ 
in terms of simple coroots, 
as $R$ is simply laced.
\end{proof}

Note that if $R$ is not of type $\A_1$, there always exists a simple root $\alpha_\adj \in \Delta$ such that $\scal{\Theta^\vee,\alpha_\adj} = 1$. Such a simple root $\alpha_\adj$ is unique if $R$ is not of type $\A_r$. In type $\A_r$ with $r \geq 2$, there are two such simple roots: $\alpha_1$ and $\alpha_r$ with simple roots labeled as in \cite{bourbaki}.

\begin{proposition}
  \label{prop-primitif-ap}
  Assume that $R$ is not of type $\A_1$ and let $\alpha_\adj \in \Delta$ be any simple root such that 
  $\scal{\Theta^\vee,\alpha_\adj} = 1$.
  \begin{enumerate}
  \item We have the equivalences: $\sigma(\Theta) \neq - \Theta$ 
  $\Leftrightarrow$ $\sigma(\alpha_\adj) = \alpha_\adj$ $\Leftrightarrow \scal{\Theta^\vee,\sigma(\Theta)} = 0$.
  \item If $\Rb$ is not of type ${\rm A}_1$, there exists a simple root $\alphab \in \Db$ such that $\scal{\Thetab^\vee,\alphab} = 1$.
      \item If $\Rb$ is not of type ${\rm A}_1$, then $2\Thetab^\vee$ is indivisible as a cocharacter of $T_{\s}$.
    \item If $\sigma(\Theta) \neq - \Theta$, then $-\sigma(\Theta)$ is the highest root of a connected component of the subsystem $R^\perp$ of $R$ generated by simple roots orthogonal to $\Theta$.
  \end{enumerate}
\end{proposition}

\begin{proof}
  (1) Note that since $\Theta$ is dominant and since $\sigma(\Theta) < 0$ (otherwise $P = G$ and $\sigma$ is trivial), we have $\scal{\Theta^\vee,\sigma(\Theta)} \leq 0$ and therefore either $\sigma(\Theta) = - \Theta$ or $\scal{\Theta^\vee,\sigma(\Theta)} = 0$ (recall from Subsection \ref{subsection:rrs} that for $\alpha$ such that $\sigma(\alpha) < 0$, there are 3 possibilities for $\scal{\alpha^\vee,\sigma(\alpha)}$ and in particular either $\sigma(\alpha) = -\alpha$ or $\scal{\alpha^\vee,\sigma(\alpha)} \geq 0$). We therefore only need to prove the first equivalence. If $\sigma(\Theta) = -\Theta$, then $\scal{\Theta^\vee,\sigma(\alpha_\adj)} = \scal{\sigma(\Theta^\vee),\alpha_\adj} = - \scal{\Theta^\vee,\alpha_\adj} = -1$ therefore $\sigma(\alpha_\adj) < 0$. Conversely, if $\scal{\Theta^\vee,\sigma(\Theta)} = 0$, then $\alpha_\adj$ does not occur in the support of $\sigma(\Theta)$. Since $\sigma(\Theta)$  is a negative root we have $\scal{\alpha_\adj^\vee,\sigma(\Theta)} \geq 0$ and thus $\scal{\sigma(\Theta)^\vee,\alpha_\adj} \geq 0$. We get $\scal{\Theta^\vee,\sigma(\alpha_\adj)} = \scal{\sigma(\Theta)^\vee,\alpha_\adj} \geq 0$; thus, $\sigma(\alpha_\adj) > 0$ and $\sigma(\alpha_\adj) = \alpha_\adj$.

  (2) If $\Rb$ is reduced then the result follows, since $\Thetab$ is the highest root of $\Rb$: take $\alphab = \alphab_\adj \in \Db$ a simple root such that $\scal{\Thetab^\vee,\alphab_\adj} = 1$.

  If $\Rb$ is non-reduced, then $\Thetab$ is the highest root; therefore, there exists a root $\betab$ such that $\Thetab = 2\betab$. We have $\scal{\Thetab^\vee,\betab} = \frac{1}{2} \scal{\Thetab^\vee,\Thetab} = 1$. Since $\Thetab^\vee$ is dominant, this implies the result.

(3) Just observe that the cocharacter $\Thetab^\vee$ of $\Sb$
is not divisible by $2$.
  
(4) If $\sigma(\Theta) \neq - \Theta$, then $\scal{\Theta^\vee,\sigma(\Theta)} = 0$ and $-\sigma(\Theta) \in R^\perp$ (the subsystem generated by simple roots orthogonal to $\Theta$). Let $\alpha \in R^\perp$. If $\sigma(\alpha) = \alpha$, then $\scal{-\sigma(\Theta)^\vee,\alpha} = - \scal{\Theta^\vee,\alpha}  = 0$. If $\alpha \in \Delta_1$, then $\sigma(\alpha) < 0$ and $\scal{-\sigma(\Theta)^\vee,\alpha} = - \scal{\Theta^\vee,\sigma(\alpha)} \geq 0$; thus, $-\sigma(\Theta)$ is dominant in $R^\perp$ and the result follows, since $-\sigma(\Theta)$ and $\Theta$ are long roots.
\end{proof}

We finish this Subsection by a general result on root systems used in Corollary \ref{coro:vir-cov}.

\begin{lemma}
\label{lemm:minimal}
Let $R$ be an irreducible root system, then $\Theta^\vee$ is the unique smallest element in the monoid of dominant cocharacters in the coroot lattice of $R$.
\end{lemma}

\begin{proof}
Note that replacing $R$ by its associated reduced root system, we may assume that $R$ is reduced. Let $p$ be a minimal element in the monoid of dominant cocharacters in the coroot lattice of $R$ and let $S(p)$ be the minimal saturated subset of the group of cocharacters containing $p$ (a subset $A$ of cocharacters is saturated if, for $a \in A$, $\alpha \in R$ and $i \in [0,\scal{a,\alpha}]$, we have $a - i \alpha^\vee \in A$). Then $S(p)$ is stable under the action of $W$ and, by \cite[Exercice VI.2.5.(b)]{bourbaki}, there exists $\alpha^\vee \in R^\vee \cap S(p)$. Letting $W$ act, we get that there exists a dominant coroot $\alpha^\vee \in S(p)$. By \cite[Exercice VI.1.23.(c)]{bourbaki}, we have $\alpha^\vee \leq p$ and thus $p = \alpha^\vee$ by minimality. 

Now there is a unique dominant element in each $W$-orbit of coroots: the highest root for long roots and the highest short root for short roots. The claim follows from this.
\end{proof}

\subsection{Marked Kac diagrams}

Our description of the components of $H \cdot C$ is based on the fact that 
$H \cdot C \simeq H \cdot [m] \subset \bP(\fg^{-\sigma})$ for
$m \in T_xC \setminus \{ 0 \}$ (Lemma \ref{lem:curve}) 
together with the following result.

\begin{lemma}[Lemma 26.8 of \cite{Timashev}] \label{affinerootS} 
The simple roots of $H^0$ and the lowest weights of $\fg^{-\sigma}$ with respect to the $H^0$-representation form an affine simple root system.
\end{lemma}

Furthermore, the lowest weights of $\fg^{-\sigma}$ together with the Dynkin diagram of $H^0$ can be encoded in the so-called Kac diagram of $G/H$. We refer to \cite[Sections 26.3 and 26.5]{Timashev} for more on these diagrams.

\begin{proposition}
 Let $X$ be the wonderful compactification of an adjoint 
 indecomposable symmetric space. The irreducible components of the orbits $H \cdot C$, where $C$ runs over the highest weight curves on $X$, are exactly
the homogeneous spaces $H^0/Q_{\delta}$,
where $\delta$ is a white node in the Kac diagram,
and $Q_\delta$ denotes the parabolic subgroup of $H^0$ associated to the set of simple roots of $H$ not adjacent to $\delta$.
\end{proposition}

\begin{proof}
The result follows from the fact that $m$ is a highest weight 
vector of $\fg^{-\sigma}$, because this highest weight is conjugate 
in $H^0$ to a lowest weight of $\fg^{-\sigma}$ corresponding to 
a white node $\delta$. (We use the fact that the representation
of $H$ in $\fg^{-\sigma}$ is self-dual).
\end{proof}

\begin{remark}
  We make the following observations.
  \begin{enumerate}
  \item There are two white nodes in the Kac diagram if and only if $X$ is Hermitian.
  \item If $X$ is Hermitian, the two corresponding parabolic subgroups are conjugated by an automorphism of $H^0$. 
  \end{enumerate}
\end{remark}

We call a Kac diagram with a marked white node a 
\emph{Marked Kac Diagram}.

\begin{example}
We illustrate the above proposition by a few examples,
where we use some notation from Subsection \ref{subsec:ct}. 
We picture the Kac diagram on the left and on the right we picture the Dynkin diagram of $H^0$ with the simple roots that are not roots of $H^0_C$ crossed.
  \begin{enumerate}
\item $G/H = \SL_8 \times \SL_8 /(Z(G) \cdot \SL_8)$ and $H^0/H^0_C \simeq \Flag(1,7)$ as $H^0$-varieties,
where $\Flag(1,7)$ denotes the variety of nested lines
and hyperplanes in $\bC^8$.

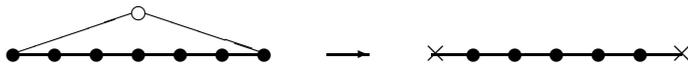
\mbox{\setlength{\unitlength}{0.55cm}
\begin{picture}(1,2.2)
\put(3,2){\circle{.3}}
\put(-0.12,1){\line(3,1){2.95}}
\put(3.16,2){\line(3,-1){2.78}}
\put(0,1){\line(1,0){1}}
\put(1,1){\line(1,0){.85}}
\put(2.15,1){\line(1,0){.9}}
\put(3,1){\line(1,0){.9}}
\put(4,1){\line(1,0){1}}
\put(5,1){\line(1,0){1}}
\put(0,1){\circle*{.3}}
\put(1,1){\circle*{.3}}
\put(2,1){\circle*{.3}}
\put(3,1){\circle*{.3}}
\put(4,1){\circle*{.3}}
\put(5,1){\circle*{.3}}
\put(6,1){\circle*{.3}}
\put(7.5,1){\vector(1,0){1}}
\put(10,1){\line(1,0){1}}
\put(11,1){\line(1,0){1}}
\put(12,1){\line(1,0){1}}
\put(13,1){\line(1,0){1}}
\put(14,1){\line(1,0){1}}
\put(15,1){\line(1,0){1}}
\put(9.8, 0.85){$\times$}
\put(15.7, 0.85){$\times$}
\put(11,1){\circle*{.3}}
\put(12,1){\circle*{.3}}
\put(13,1){\circle*{.3}}
\put(14,1){\circle*{.3}}
\put(15,1){\circle*{.3}}
\end{picture}}

\item $G/H= F_4/B_4$ and $H^0/H^0_C \simeq \OG(4,9)$ as $H^0$-varieties,
where $\OG(4,9)$ denotes the orthogonal Grassmannian
of isotropic $4$-dimensional subspaces of $\bC^9$.

\setlength{\unitlength}{0.6cm}

\begin{picture}(1,2)
\put(1.15,1){\line(1,0){0.9}}
\put(2,1.05){\line(1,0){1}}
\put(2.25,0.85){$<$}
\put(2,0.95){\line(1,0){1}}
\put(3,1){\line(1,0){.85}}
\put(4.15,1){\line(1,0){.9}}

\put(1,1){\circle{.3}}
\put(2,1){\circle*{.3}}
\put(3,1){\circle*{.3}}
\put(4,1){\circle*{.3}}
\put(5,1){\circle*{.3}}

\put(7.5,1){\vector(1,0){1}}

\put(12,1.05){\line(1,0){1}}
\put(12.25,0.85){$<$}
\put(12,0.95){\line(1,0){1}}
\put(13,1){\line(1,0){.85}}
\put(14.15,1){\line(1,0){.9}}

\put(11.6,0.85){$\times$}
\put(13,1){\circle*{.3}}
\put(14,1){\circle*{.3}}
\put(15,1){\circle*{.3}}

\end{picture}

\item $G/H= \SL_8/S(\GL_3\times \GL_5)$ and $H^0/H^0_C \simeq \bP^2 \times \bP^4$ as varieties..

\setlength{\unitlength}{0.6cm}
\begin{picture}(2,2.2)

\put(3,2){\circle{.3}}
\put(0,1){\line(3,1){2.9}}
\put(3.1,2){\line(3,-1){2.9}}

\put(0,1){\line(1,0){1}}
\put(1,1){\line(1,0){.85}}
\put(2.15,1){\line(1,0){.9}}
\put(3,1){\line(1,0){.9}}
\put(4,1){\line(1,0){1}}
\put(5,1){\line(1,0){1}}

\put(0,1){\circle*{.3}}
\put(1,1){\circle*{.3}}
\put(2,1){\circle{.3}}
\put(3,1){\circle*{.3}}
\put(4,1){\circle*{.3}}
\put(5,1){\circle*{.3}}
\put(6,1){\circle*{.3}}

\put(7.5,1){\vector(1,0){1}}

\put(10,0.5){\line(1,0){1}}
\put(13,0.5){\line(1,0){1}}
\put(14,0.5){\line(1,0){1}}
\put(15,0.5){\line(1,0){1}}

\put(11,0.5){\circle*{.3}}
\put(13,0.5){\circle*{.3}}
\put(9.7,0.32){$\times$}
\put(15.8,0.32){$\times$}
\put(14,0.5){\circle*{.3}}
\put(15,0.5){\circle*{.3}}

\put(10,1.5){\line(1,0){1}}
\put(13,1.5){\line(1,0){1}}
\put(14,1.5){\line(1,0){1}}
\put(15,1.5){\line(1,0){1}}

\put(10,1.5){\circle*{.3}}
\put(14,1.5){\circle*{.3}}
\put(10.8, 1.32){$\times$}
\put(12.7, 1.32){$\times$}
\put(15,1.5){\circle*{.3}}
\put(16,1.5){\circle*{.3}}

\put(17,0.3){$\mathbb P^2 \times {\mathbb P^4}^\vee$}
\put(16.8,1.3){${\mathbb P^{2}}^\vee \times \mathbb P^{4}$}
\end{picture}

\item $G/H= \Sp_{12}/\GL_{6}$ and $H^0/H^0_C \simeq \bP^5$ as varieties.

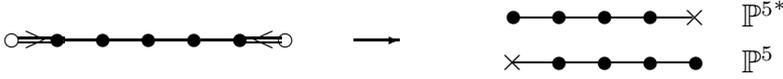
\setlength{\unitlength}{0.6cm}

\begin{picture}(1,2)
\put(0.15,1.05){\line(1,0){1}}
\put(0.25,0.85){$>$}
\put(0.15,0.95){\line(1,0){1}}
\put(1.15,1){\line(1,0){0.9}}
\put(2,1){\line(1,0){.85}}

\put(5,1.05){\line(1,0){0.9}}
\put(5.25,0.85){$<$}
\put(5,0.95){\line(1,0){0.9}}

\put(3,1){\line(1,0){.85}}
\put(4.15,1){\line(1,0){.9}}

\put(0,1){\circle{.3}}
\put(1,1){\circle*{.3}}
\put(2,1){\circle*{.3}}
\put(3,1){\circle*{.3}}
\put(4,1){\circle*{.3}}
\put(5,1){\circle*{.3}}
\put(6,1){\circle{.3}}

\put(7.5,1){\vector(1,0){1}}

\put(11,1.5){\circle*{.3}}
\put(12,1.5){\circle*{.3}}
\put(13,1.5){\circle*{.3}}
\put(14,1.5){\circle*{.3}}

\put(11,1.5){\line(1,0){1}}
\put(12,1.5){\line(1,0){1}}
\put(13,1.5){\line(1,0){.85}}
\put(14,1.5){\line(1,0){.9}}

\put(12,0.5){\circle*{.3}}
\put(13,0.5){\circle*{.3}}
\put(14,0.5){\circle*{.3}}
\put(15,0.5){\circle*{.3}}

\put(11,0.5){\line(1,0){1}}
\put(12,0.5){\line(1,0){1}}
\put(13,0.5){\line(1,0){.85}}
\put(14,0.5){\line(1,0){.9}}

\put(14.7,1.33){$\times$}
\put(10.7,0.33){$\times$}

\put(16,1.3){${\mathbb P^5}^{\vee}$}
\put(16,0.3){$\mathbb P^5$}
\end{picture}

\end{enumerate}
\end{example}

\subsection{Some examples}
\label{subsec:ex}

We describe some families of examples. Recall
that $G_\ad = G/Z(G)$.

\paragraph{Hermitian types.} Assume that $G/H$ is of Hermitian type and let $T = T_f$ be a maximal torus of fixed type. The involution $\sigma$ is given on $G_\ad$ by conjugation with respect to 
$\varpi_\alpha^\vee(-1)$ (the one-parameter subgroup $\varpi_\alpha^\vee$
evaluated at $-1$), 
where $\alpha$ is a simple cominuscule root, i.e., it 
appears with coefficient $1$ in $\Theta$.
In this case, $\sigma(\Theta) = - \Theta$ and every irreducible component of $H \cdot C$ is a smooth Schubert variety 
in $\bP(\cO_{\min})$,
of dimension $\frac{1}{2}(\dim \cO_{\min} - 1)$. The exceptional cases correspond to the simple cominuscule roots $\alpha$ which are mapped to a different simple cominuscule root by the involution $-w_0$. 

\paragraph{Subadjoint case.}
\label{para:subadjoint}
Let $G$ be a simple adjoint group and let $T \subset B \subset G$ be a maximal torus and a Borel subgroup. Set $\aleph = \{ \alpha \in \Delta \ | \ \scal{\Theta^\vee,\alpha} \neq 0 \}$ and  $\varpi^\vee = \sum_{\alpha \in \aleph} \varpi_\alpha^\vee$. Then $|\aleph| = 1$ except in type $\A_r$ with $r \geq 2$, where $|\aleph| = 2$. Define the involution $\sigma$ on $G$ by conjugation by $\varpi^\vee(-1)$. Note that, except in type $\A_1$, we have $\Theta^\vee = \sum_{\alpha \in \aleph} \varpi_\alpha^\vee = \varpi^\vee$ thus $\Theta^\vee(-1) = \varpi^\vee(-1)$. We exclude type $\A_1$ of the following discussion, so that $\varpi^\vee = \Theta^\vee$. 

The symmetric space $G/G^\sigma$ is not Hermitian, except in type $\A_r$. It is worth noting that the maximal torus $T$ is of fixed type and not of split type for $\sigma$. According to the classification, we get the symmetric spaces of the following types: A III in rank $2$, BD I in rank $4$, C II in rank $1$, E II, E VI, E IX, F I and G.

We have 
$\fh = \fg^\sigma = \fg(\Theta) \oplus \fk$
where $\fg(\Theta) = \scal{e,h,f}$ with 
$e \in \fg_\Theta \setminus \{ 0 \}$, 
$f \in \fg_{-\Theta} \setminus \{ 0 \}$ and $h = [e,f] = \Theta^\vee$
(in particular, $\fg(\Theta) \simeq \fsl_2$), and $\fk$ is a reductive
Lie subalgebra of $\fg$ with root system 
$R_\fk = \{ \beta \in R \ | \ \scal{\Theta^\vee,\beta} = 0 \}$. 
Furthermore, we have $\fg^{-\sigma} = \bC^2 \otimes V_\fk$, 
where $\bC^2$ is the standard representation of $\fg(\Theta)$
and $V_\fk$ is a $\fk$-representation which is irreducible 
in all types except $\A_r$. In type $\A_r$, we have 
$V_\fk = V^+ _\fk\oplus V^-_\fk$ which are dual irreducible representations.
In any type, $G$ has a finite cover of the form $G(\Theta) \times K$, 
where $G(\Theta) \simeq \SL_2$ has Lie algebra
$\fg(\Theta)$ and acts linearly on $\bC^2$, while $K$ is a connected
reductive group with Lie algebra $\fk$ and acts linearly on $V_\fk$.
Let $u^+,u^- \in \bC^2$ be a highest and a lowest weight vector
for $G(\Theta)$ and let $v \in V_\fk$ be a highest weight vector for $K$. 
We identify $V_\fk$ to the subspace 
$\scal{u^+} \otimes V_\fk$ of  $\fg^{-\sigma}$.

We have $\cO_{\min} = G \cdot e$. We will use the following isomorphism of $T$-representations: 
$T_e \cO_{\min} = \scal{e,h} \oplus V_\fk \subset \fg$. Note that on the latter space, 
the symplectic form is given by $\omega(x,y) = \kappa(f,[x,y])$ (with $\kappa$ the Killing form) 
and restricts to symplectic forms on the two orthogonal subspaces $\scal{e,h}$ and $V_\fk$. 
Recall the definition of the subadjoint variety $\bL_G$ as the set of lines in $\bP(\cO_{\min})$ 
passing through $[e]$;
therefore $\bL_G$ is identified with a subvariety of $\bP(T_e\cO_{\min})$. 
Note that $\bL_G  = \emptyset$ in type $\C_r$, since $\bP(\cO_{\min}) \subset \bP(\fg)$ 
is the second Veronese embedding in this case. In the other cases,  we have 
$\bL_G = \bP(\cO_{\min}) \cap \bP(V_\fk)$. Indeed, the locus $\mathcal{S}$ 
covered by lines through $[e]$ in $\bP(\cO_{\min})$ is $B$-stable and irreducible 
(because the set of lines is a $K$-orbit and hence irreducible), thus it is a Schubert variety 
(in type $\A_r$ we have two families of lines thus two Schubert varieties). 
Furthermore $\mathcal{S}$ is contained in $\bP(\cO_{\min}) \cap \bP(\scal{e} + V_\fk)$. 
On the other hand, the set of weights in $\scal{e}+ V_\fk$ has a minimal element (two in type $\A_r$): the simple root $\alpha$ with $\aleph = \{ \alpha \}$ (the two simple roots in $\aleph$ in type $\A_r$); thus the intersection $\bP(\cO_{\min}) \cap \bP(\scal{e} + V_\fk)$ is a unique Schubert variety (two in type $\A_r$) and we get the equality $\mathcal{S} = \bP(\cO_{\min}) \cap \bP(\scal{e} + V_\fk)$. 
Intersecting with $\bP(V_\fk)$ yields the desired equality
$\bL_G = \bP(\cO_{\min}) \cap \bP(V_\fk)$, 
since $[e]$ is not in $\bP(V_\fk)$ and hence the lines through $[e]$ meet $\bP(V_\fk)$ in one point. 
Note that $\bL_G = \bP(\cO_{\min}) \cap \bP(V_\fk)$ also holds in type $\C_r$, since 
$\bP(\cO_{\min}) \cap \bP(\fg^{-\sigma}) = \emptyset$ in this case. 

In types different from $\A_r$ and $\C_r$, we have $\bL_G = K \cdot [v]$; this is the closed $K$-orbit in 
$\bP(V_\fk)$, and spans this projective space. In type $\A_r$, the variety $\bL_G$ has two connected components given by the closed $K$-orbits in $\bP(V^+_\fk)$ and $\bP(V^-_\fk)$. 
Let $\fl_G$ be the inverse image in $V_\fk$ of $T_{[v]}\bL_G$, 
then $\fl_G$ is a Lagrangian subspace in $V_\fk$. We will recover this fact by using 
$X_\ad$, the wonderful compactification of the adjoint symmetric space $G/\N_G(G^\sigma)$.

Assume that $G$ is not of type $\C_r$. If $G$ is also not of type $\A_r$, 
let $\alpha \in \Delta$ be the unique simple root such that $\aleph = \{\alpha\}$. 
Let $G_{\alpha}$ be the closed connected subgroup of $G$ with Lie algebra 
$\fg_{-\alpha}\oplus\scal{\alpha^\vee}\oplus\fg_\alpha$ and pick a representative 
$n_\alpha$ of the simple reflection $s_\alpha$ in $G_{\alpha}$. 
Let $m = u^+ \otimes v$ be the highest weight vector of $\fg^{-\sigma}$. 
Then $m = n_\alpha(e)$ lies in $\cO_{\min}$ and $\cO_{\min} = G \cdot [m]$. 
Theorem \ref{theo-adj} implies that 
$$\VMRT(X_\ad) = (G(\Theta) \times K) \cdot [m] = \bP^1 \times \bL_G.$$
In type $\A_r$ with $r \geq 2$, we have $|\aleph| = 2$ and defining 
$m_\alpha = n_\alpha(e)$ for $\alpha \in \aleph$, we get two $\VMRT$s for $X_\ad$ 
given by $\VMRT_\alpha(X_\ad) =  (G(\Theta) \times K) \cdot [m_\alpha]$; 
both are isomorphic to $\bP^1 \times \bL_G$.

Assume that $G$ is not of type $\A_r$ or $\C_r$. 
Set $\widehat{M} = \cO_{\min}$, $M = \bP(\cO_{\min})$, $L = \VMRT(X_\ad)$ 
and let $\widehat{L}$ be the cone over $L$ in $\fg^{-\sigma} \subset \fg$. 
%The above description gives that $\bL_G = M \cap \bP(V_\fk)$. 
Write $\fl_G = \scal{m} \oplus \overline{\fl}_G$ where $\overline{\fl}_G$ is the unique $T$-stable complement of $\scal{m}$ in $\fl_G$ (actually $\overline{\fl}_G = T_{[m]}\bL_G$ as $T$-module) and $\fl_G \subset T_{m}\widehat{L}$. Proposition \ref{Legendrian} implies that $\widehat{L}$ is Lagrangian in $\widehat{M}$; in particular $\overline{\fl}_G$ is isotropic in $T_{m}\cO_{\min}$. Using the fact that $\fl_G$ is the cone over $T_m(K \cdot m)$, it is easy to see that the set of $T$-weights in $\overline{\fl}_G$ is $\{\beta \in R \ | \ \scal{\alpha^\vee,\beta} = 0 \textrm{ and } \scal{\Theta^\vee,\beta} = 1\}$ 
thus $G_{\alpha}$ acts trivially on $\overline{\fl}_G$ (for $\beta$ a weight of $\overline{\fl}_G$, $\alpha + \beta$ is not a root of $G$). 
In particular, $n_\alpha$ acts trivially on $\overline{\fl}_G$. 
Letting $n_\alpha$ act on the inclusion $\overline{\fl}_G \subset T_m\cO_{\min}$, 
we get that $\overline{\fl}_G \subset V_\fk \subset T_e\cO_{\min}$ is isotropic for $\omega$. 
To recover that $\fl_G$ is Lagrangian in $V_\fk$, we are left to prove that $\omega(\overline{\fl}_G,m) = 0$. But the only $T$-eigenspace in $V_\fk$ not orthogonal to $m$ for $\omega$ is $\fg_\alpha$ and $\alpha$ is not a weight of $\overline{\fl}_G$ since $n_\alpha$ does not act trivially on $\fg_\alpha$. This implies that $\fl_G$ is a Lagrangian subspace of $V_\fk$.

\paragraph{Non-Fano cases.} The wonderful compactifications $X_\ad$ 
of adjoint indecomposable
symmetric spaces are not always Fano. The Fano and non-Fano cases have been classified in \cite[Theorem 2.1, Table 2]{Ruzzi}. We summarise the results here: The types for which $X_\ad$ is not Fano are CI, DI, EI, EV, EVIII, FI and G. An easy way to find them is to use both the restricted root sytem and the Satake diagram (see \cite[Table 26.3]{Timashev}): the non-Fano cases are those for which the restricted root system is not of type A nor of type B and the Satake diagram has only white nodes and no arrow.

\subsection{Classification table}
\label{subsec:ct}

We list all symmetric spaces $G/H$ (up to finite coverings) 
with $X_\ad$ indecomposable, 
their varieties of minimal rational tangents $\cC_x$ and the restriction of $\cO_{\bP(\fg^{-\sigma})}(1)$ to the \VMRT~giving the embedding $\cC_x \subset \bP(\fg^{-\sigma})$. For $\cC_1 \bigsqcup \cC_2$, the notation $\cO(1)$ corresponds to the embedding in $\bP(H^0(\cC_1,\cO_{\cC_1}(1)) \oplus H^0(\cC_2,\cO_{\cC_2}(1)))$. 

\vskip 0.2 cm

\noindent
{\bf Some notations.} All unmarked cases are 
non-Hermitian. H.n.e = Hermitian non-exceptional. 
H.e = Hermitian exceptional. 
$\Q_n$ = smooth quadric of dimension $n$. 
$\Gr(a,b)$ = Grassmannian of vector subspaces of dimension $a$ in $\bC^b$. 
$\OG(a,b)$ = closed subset of $\Gr(a,b)$ of isotropic subspaces for a non-degenerate quadratic form on $\bC^b$ (with $a < 2b$). 
$\OG(b,2b)$ = a connected component of the Grassmannian of maximal isotropic subspaces in $\bC^{2b}$ for a non-degenerate quadratic form. 
$\IG(a,2b)$ = closed subset of $\Gr(a,2b)$ of isotropic subspaces for a non-degenerate symplectic form on $\bC^{2b}$. 
$\LG(b,2b)$ = Grassmannian of maximal isotropic subspaces in 
$\bC^{2b}$ for a non-degenerate symplectic form. 
$\Flag(1,r)$ = nested subspaces of dimension $1$ and $r$ in $\bC^{r+1}$.

\begin{sidewaystable}
$$\begin{array}{|c|c|c|c|c|c|c|c|c|c|}
    \textrm{Type} & G/H & {\rm Condition} & 
    \Rb & N \cdot C & (\VMRT,\cO_{\bP(\fg^{-\sigma})}(1)) & 
    \sigma(\Theta) = - \Theta
    & {\rm Herm/Exc} & {\rm Fano} \\ 
\hline 
\textrm{Group} & H \times H/H & 
\textrm{Type}(H^0) \neq \textrm{A}_r &
\textrm{Type of $H^0$} & \bP(\cO_{\min,H}) & (H \cdot C,\cO(1)) & \yes & & \yes  \\
\textrm{Group} & \SL_{r+1} \times \SL_{r+1}/ \SL_{r+1} & r \geq 2 & \textrm{A}_r & \Flag(1,r) & (\bP^r \times \bP^r,\cO(1,1)) & \yes & & \yes \\
\textrm{Group} & \SL_2 \times \SL_2/\SL_2 & & 
\textrm{A}_1 & \Q_1 & (\bP^2,\cO(1)) & \yes & & \yes \\
\textrm{A I} & \SL_{r+1}/\SO_{r+1} & r \geq 2 & \textrm{A}_{r} & \Q_{r-1} & (\bP^{r},\cO(2)) & \yes & & \yes \\ 
\textrm{A I} & \SL_{2}/\SO_{2} & & \textrm{A}_{1} & \{{\rm pt}\} \bigsqcup \{{\rm pt}\} & (\bP^{1},\cO(1)) & \yes & \textrm{H.n.e} & \yes \\ 
\textrm{A II} & \SL_{2r +2}/\Sp_{2r+2} & r \geq 2 & \textrm{A}_{r} &
\IG(2,2r+2) & (\Gr(2,2r+2),\cO(1)) & \no & & \yes \\ 
\textrm{A III} & \SL_n/S(\GL_r\times \GL_{n-r}) & 1 \leq r < n/2 & \textrm{BC}_r & \bP^{r-1} \times \bP^{n - r - 1} & (H \cdot C,\cO(1,1)) & \yes & \textrm{H.e} & \yes \\
\textrm{A III} & \SL_{2r}/S(\GL_r\times \GL_{r}) & & \textrm{C}_r & (\bP^{r-1})^2 \bigsqcup ({\bP^{r-1}}^\vee)^2 &(H\cdot C,\cO(1,1)) & \yes & \textrm{H.n.e} & \yes \\ 
\textrm{BD I} & \SO_{n}/S(\OO_{r}\times \OO_{n-r}) & 3 \leq r \leq \frac{n-1}{2} &  \textrm{B}_r & \Q_{r-2} \times \Q_{n - r - 2} & (H \cdot C ,\cO(1,1)) & \yes & & \yes \\
\textrm{BD I} & \SO_{n}/S(\OO_{2}\times \OO_{n-2}) & &  \textrm{B}_2 & \Q_{n - 4} \bigsqcup \Q_{n-4} & (H\cdot C,\cO(1)) & \yes & \textrm{H.n.e} & \yes \\
\textrm{BD II} & \SO_{n}/S(\OO_{1}\times \OO_{n-1}) & & \textrm{A}_1 & \Q_{n-3} & (\bP^{n-2}, \cO(1)) & \no & & \yes \\
\textrm{C I} & \Sp_{2r}/\GL_{r} & r \geq 3 & \textrm{C}_r & \bP^{r-1} \bigsqcup {\bP^{r-1}}^\vee & (H\cdot C,\cO(2)) & \yes & \textrm{H.n.e} & \no \\
\textrm{C II} & \Sp_{2n}/\Sp_{2r}\times \Sp_{2n-2r} & 1 \leq r \leq
\frac{(n-1)}{2} & \textrm{BC}_r & \bP^{2r-1} \times \bP^{2n-2r-1} & (H \cdot C, \cO(1,1)) & \no & & \yes \\
\textrm{C II} & \Sp_{4r}/\Sp_{2r}\times \Sp_{2r} & r \geq 2 & \textrm{C}_r &
\bP^{2r-1} \times \bP^{2r-1} & (H \cdot C, \cO(1,1)) & \no & & \yes \\ 
\textrm{D I} & \SO_{2r}/S(\OO_{r}\times \OO_{r}) & r \geq 4 & \textrm{D}_r &
\Q_{r-2} \times \Q_{r-2} & (H \cdot C, \cO(1,1)) & \yes & & \no \\
\textrm{D III} & \SO_{4r}/\GL_{2r} & & \textrm{C}_r & \Gr(2,2r) \bigsqcup \Gr(2r-2,2r) & (H\cdot C, \cO(1)) & \yes & \textrm{H.n.e} & \yes \\
\textrm{D III} & \SO_{4r+2}/\GL_{2r+1} & & 
\textrm{BC}_r & \Gr(2,2r+1) & (H \cdot C, \cO(1)) & \yes & \textrm{H.e.} & \yes \\
\textrm{E I} & E_6/C_4 & & 
\textrm{E}_6 & \LG(4,8) & (H \cdot C, \cO(1)) & \yes & & \no \\
\textrm{E II} & E_6/A_5 \times A_1 & & 
\textrm{F}_4 & \Gr(3,6) \times \bP^1 & (H \cdot C, \cO(1,1)) & \yes & & \yes \\ 
\textrm{E III} & E_6/D_5 \times \bC^* & & \textrm{BC}_2 & \OG(5,10) & (H \cdot C, \cO(1)) & \yes & \textrm{H.e.} & \yes \\
\textrm{E IV} & E_6/F_4 & & \textrm{A}_2 & F_4/P_4 & (E_6/P_6, \cO(1)) & \no & & \yes \\
\textrm{E V} & E_7/A_7 & & \textrm{E}_7 & \Gr(4,8) & (H \cdot C, \cO(1)) & \yes & &  \no \\
\textrm{E VI} & E_7/D_6 \times A_1 & & \textrm{F}_4 & \OG(6,12) \times \bP^1 & (H \cdot C, \cO(1,1)) & \yes & & \yes \\ 
\textrm{E VII} & E_7/E_6 \times \bC^* & & \textrm{C}_3 & E_6/P_1 \bigsqcup E_6/P_6 & (H\cdot C,\cO(1)) & \yes & \textrm{H.n.e.} & \yes \\ 
\textrm{E VIII} & E_8/D_8 & & \textrm{E}_8 & \OG(8,16) & (H \cdot C, \cO(1)) & \yes & & \no \\ 
\textrm{E IX} & E_8/E_7 \times A_1 & & \textrm{F}_4 & E_7/P_7 \times \bP^1 & (H \cdot C, \cO(1,1)) & \yes & & \yes \\ 
\textrm{F I} & F_4/C_3 \times A_1 & &  \textrm{F}_4 & \LG(3,6) \times \bP^1 & (H \cdot C, \cO(1,1)) & \yes & & \no \\ 
\textrm{F II} & F_4/B_4 & & \textrm{BC}_1 & \OG(4,9) & (H \cdot C, \cO(1)) & \no &  & \yes \\ 
\textrm{G} & G_2/A_1 \times A_1 & & \textrm{G}_2 & \bP^1 \times \bP^1 & (H \cdot C, \cO(1,3)) & \yes &  & \no \\ 
\hline
  \end{array}$$
  \caption{\label{table-class} Wonderful compactifications and their \VMRT.}
\end{sidewaystable}

\end{document}